\documentclass{amsart}
\usepackage{amsmath,amssymb,amsthm,mathrsfs}

\theoremstyle{plain}
\newtheorem{theorem}{Theorem}[section]
\newtheorem{lemma}[theorem]{Lemma}
\newtheorem{corollary}[theorem]{Corollary}
\newtheorem{proposition}[theorem]{Proposition}

\theoremstyle{definition}
\newtheorem{remark}[theorem]{Remark}

\numberwithin{equation}{section}

\newcommand{\BC}{{\mathbb C}}\newcommand{\BD}{{\mathbb D}}

\newcommand{\BR}{{\mathbb R}}
\newcommand{\BT}{{\mathbb T}}

\newcommand{\BZ}{{\mathbb Z}}

\newcommand{\cD}{{\mathcal D}}

\newcommand{\cH}{{\mathcal H}}

\newcommand{\cK}{{\mathcal K}}\newcommand{\cL}{{\mathcal L}}

\newcommand{\cS}{{\mathcal S}}\newcommand{\cT}{{\mathcal T}}
\newcommand{\cU}{{\mathcal U}}
\newcommand{\cX}{{\mathcal X}}
\newcommand{\cY}{{\mathcal Y}}\newcommand{\cZ}{{\mathcal Z}}
\newcommand{\bA}{{\mathbf A}}\newcommand{\bB}{{\mathbf B}}
\newcommand{\bC}{{\mathbf C}}\newcommand{\bD}{{\mathbf D}}

\newcommand{\bW}{{\mathbf W}}

\newcommand{\fH}{{\mathfrak H}}

\newcommand{\fL}{{\mathfrak L}}

\newcommand{\fT}{{\mathfrak T}}

\newcommand{\ep}{\epsilon}

\newcommand{\la}{\lambda}

\newcommand{\Si}{\Sigma}

\newcommand{\im}{\textup{Im\,}}

\newcommand{\kr}{\textup{Ker\,}}

\newcommand{\spec}{r_\textup{spec}}

\newcommand{\mat}[2]{\ensuremath{\left[\begin{array}{#1}#2\end{array}\right]}}

\newcommand{\sbm}[1]{\left[\begin{smallmatrix}#1\end{smallmatrix}\right]}

\newcommand{\tu}[1]{\textup{#1}}

\newcommand{\half}{\frac{1}{2}}
\newcommand{\ands}{\quad\mbox{and}\quad}

\newcommand{\wtil}{\widetilde}

\newcommand{\Obs}{\operatorname{Obs}\,}
\newcommand{\Rea}{\operatorname{Rea}\,}

\newcommand{\bu}{{\mathbf u}}
\newcommand{\bv}{{\mathbf v}}
\newcommand{\bx}{{\mathbf x}}
\newcommand{\by}{{\mathbf y}}

\newcommand{\bcU}{{\boldsymbol{\mathcal U}}}
\newcommand{\frakH}{{\mathfrak H}}
\newcommand{\frakL}{{\mathfrak L}}
\newcommand{\frakT}{{\mathfrak T}}
\newcommand{\oX}{\overline{X}}

\newcommand{\uS}{\underline{S}}

\begin{document}

\title[Infinite Dimensional Bounded Real Lemma II]{Standard
versus Bounded Real Lemma
with infinite-dimensional state space II:\\ The storage function approach}

\author{J.A. Ball}
\address{J.A. Ball, Department of Mathematics, Virginia Tech,
Blacksburg, VA 24061-0123, USA}
\email{joball@math.vt.edu}

\author{G.J. Groenewald}
\address{G.J. Groenewald, Department of Mathematics, Unit for BMI,
North-West University, Potchefstroom 2531, South Africa}
\email{Gilbert.Groenewald@nwu.ac.za}

\author{S. ter Horst}
\address{S. ter Horst, Department of Mathematics, Unit for BMI,
North-West University, Potchefstroom 2531, South Africa}
\email{Sanne.TerHorst@nwu.ac.za}

\thanks{This work is based on the research supported in part by the National
Research Foundation of South Africa (Grant Numbers 93039, 90670, and 93406).}

\begin{abstract}
For discrete-time causal linear input/state/output systems, the Bounded Real Lemma explains
(under suitable hypotheses) the contractivity of the values of the transfer function over the unit disk
for such a system in terms of the existence of a positive-definite solution of a certain Linear Matrix Inequality (the
Kalman-Yakubovich-Popov (KYP) inequality).  Recent work has extended this result to the setting of
infinite-dimensional state space and associated non-rationality of the transfer function, where at least in some cases unbounded solutions of the generalized KYP-inequality are required.  This paper is the second installment in a series of papers on the Bounded Real Lemma and the KYP inequality. We adapt Willems' storage-function approach to the infinite-dimensional linear setting, and in this way reprove various results presented in the first installment, where they were obtained as applications of infinite-dimensional State-Space-Similarity theorems, rather than via explicit computation of storage functions.
\end{abstract}



\maketitle


\section{Introduction}\label{S:intro}

This paper is the second installment, following \cite{KYP1}, on the infinite dimensional bounded real lemma for discrete-time systems and the discrete-time Kalman-Yakubovich-Popov (KYP) inequality. In this context, we consider the discrete-time linear system
\begin{equation}\label{dtsystem}
\Si:=\left\{
\begin{array}{ccc}
\bx(n+1)&=&A \bx(n)+B \bu(n),\\
\by(n)&=&C \bx(n)+D \bu(n),
\end{array}
\right. \qquad (n\in\BZ)
\end{equation}
where $A:\cX\to\cX$, $B:\cU\to\cX$, $C:\cX\to\cY$ and $D:\cU\to\cY$
are bounded linear Hilbert space operators, i.e., $\cX$, $\cU$ and
$\cY$ are Hilbert spaces and the {\em system matrix} associated with
$\Si$ takes the form
\begin{equation}\label{sysmat}
M=\mat{cc}{A&B\\ C& D}:\mat{cc}{\cX\\ \cU}\to\mat{c}{\cX\\ \cY}.
\end{equation}
We refer to the pair $(C,A)$ as the {\em output pair} and to the pair
$(A,B)$ as the {\em input pair}. In this case input sequences
$\bu=(\bu(n))_{n\in\BZ}$, with $\bu(n)\in\cU$, are mapped to output
sequences $\by=(\by(n))_{n\in\BZ}$, with $\by(n)\in\cY$, through the
state sequence $\bx=(\bx(n))_{n\in\BZ}$, with $\bx(n)\in \cX$. A system trajectory of the system $\Si$ is then any triple $(\bu(n),\bx(n),\by(n))_{n\in\BZ}$ of input, state and output sequences that satisfy the system equations \eqref{dtsystem}.

With
the system $\Si$ we associate the {\em transfer function} given by
\begin{equation}\label{trans}
F_\Si(\lambda)=D+\lambda C(I-\lambda A)^{-1}B.
\end{equation}
Since $A$ is bounded, $F_\Si$ is defined and analytic on a
neighborhood of $0$ in $\BC$. We are interested in the case
where $F_\Si$ admits an analytic continuation to the open unit disk
$\BD$ such that the supremum norm $\|F_\Si\|_\infty$ of $F_\Si$ over
$\BD$ is at most one, i.e., $F_\Si$ has analytic continuation to a
function in the Schur class
\[
  \cS(\cU, \cY) = \left\{ F \colon {\mathbb D}
  \underset{\text{holo}}\mapsto \cL(\cU, \cY) \colon \| F(\lambda) \| \le 1
  \text{ for all } z \in {\mathbb D}\right\}.
\]

Sometimes we also consider system trajectories $(\bu(n),\bx(n),\by(n))_{n\ge n_0}$
of the system $\Si$ that are initiated at a certain time $n_0\in\BZ$, in which case the input, state and output at time $n<n_0$ are set equal to zero, and we only require that the system equations \eqref{dtsystem} are satisfied for $n\geq n_0$. Although technically such trajectories are not system trajectories for $\Si$, but rather correspond to trajectories of the corresponding singly-infinite forward-time system rather than the bi-infinite system $\Si$, the transfer function of this singly-infinite forward-time system coincides with the transfer function $F_\Si$ of $\Si$. Hence for the sake of the objective, determining whether $F_\Si\in \cS(\cU,\cY)$, there is no problem with considering such singly infinite system trajectories.

Before turning to the infinite-dimensional setting, we first discuss the case where $\cU$, $\cX$, $\cY$ are all finite-dimensional. If in this case one considers the parallel situation in continuous time rather than in
discrete time, these ideas have origins in circuit theory, specifically conservative or passive circuits.  An important question in this context is to identify which rational matrix functions, analytic on the left half-plane (rather than the unit disk $\BD$), arise from a lossless or dissipative circuit in this way (see e.g. Belevitch \cite{Bel}).

According to Willems \cite{Wil72a, Wil72b}, a linear system $\Sigma$ as in  \eqref{dtsystem}
is {\em dissipative}  (with respect to {\em supply rate} $s(u,y)  = \| u \|^2 - \| y \|^2$) if it has a
{\em storage function} $S \colon  \cX \to {\mathbb R}_+$, where $S(x)$ is to be interpreted as a measure of
the {\em energy} stored by the system when it is in state $x$.
Such a storage function $S$ is assumed to satisfy the dissipation inequality
\begin{equation}   \label{diss}
  S(\bx(n+1)) - S(\bx(n)) \le \|\bu(n) \|^2 - \| \by(n) \|^2
\end{equation}
over all trajectories $(\bu(n), \bx(n), \by(n))_{n\in\BZ}$ of the system $\Sigma$ as well as the additional normalization condition that $S(0) = 0$. The dissipation inequality can be interpreted as saying that for the given system trajectory, the energy stored in the system ($S(\bx(n+1)) - S(\bx(n))$) when going from state $x(n)$ to $x(n+1)$ can be no more than the difference between the energy that enters the system ($\|\bu(n) \|^2$) and the energy that leaves the system ($\| \by(n) \|^2$) at time $n$.

For our discussion here we shall only be concerned with the so-called {\em scattering supply rate} $s(u,y)
= \| u \|^2 - \| y \|^2$.
 It is not hard to see that a consequence
of the dissipation inequality \eqref{diss} on system trajectories is that the transfer function $F_\Sigma$ is in the
Schur class $\cS(\cU, \cY)$.  The results extend to nonlinear systems as well
(see \cite{Wil72a}), where one talks about the system having $L^2$-gain at most $1$ rather the system
having transfer function in the Schur class.

In case the system $\Sigma$ is finite-dimensional and minimal (as defined in the statement of
Theorem \ref{T:BRLfinstan} below), one can show that the smallest storage function, the {\em available storage} $S_a$, and the largest storage function, the {\em required supply} $S_r$, are {\em quadratic}, provided storage functions for $\Si$ exist. That $S_a$ and $S_r$ are quadratic means that there are positive-definite matrices $H_a$ and $H_r$ so that $S_a$ and $S_r$ have the quadratic form
$$
  S_a(x) = \langle H_a x, x \rangle, \quad S_r(x) = \langle H_r x, x \rangle
$$
with $H_a$ and $H_r$ actually being positive-definite.
For a general quadratic storage function $S_H(x) = \langle H x, x \rangle$ for a positive-definite matrix $H$,
it is not hard to see that the dissipation inequality \eqref{diss} assumes the form of a linear matrix inequality (LMI):
\begin{equation}  \label{KYP1}
\begin{bmatrix} A & B \\ C & D \end{bmatrix}^{*}
\begin{bmatrix} H &  0 \\ 0 & I_{\cY} \end{bmatrix}
\begin{bmatrix} A & B \\ C & D\end{bmatrix}
\preceq \begin{bmatrix} H & 0 \\ 0 & I_{\cU}\end{bmatrix}.
\end{equation}
This is what we shall call the {\em Kalman-Yakubovich-Popov} or KYP  inequality (with solution $H$
for given system matrix $M = \sbm{ A & B \\ C & D}$).

Conversely, if one starts with a finite-dimensional, minimal, linear system $\Si$ as in \eqref{dtsystem} for which the transfer function
$F_\Sigma$ is in the Schur-class, it is possible to show that there exist
quadratic storage functions $S_H$  for the system satisfying the coercivity condition
$S_H(x) \ge \delta \| x \|^2$ for some $\delta > 0$
(i.e., with $H$ strictly positive-definite). This is the storage-function interpretation behind the following
result, known as the {\em Kalman-Yakubovich-Popov lemma}.

\begin{theorem}[Standard Bounded  Real Lemma (see \cite{AV})]
\label{T:BRLfinstan}
Let $\Si$ be a discrete-time linear system as in
\eqref{dtsystem} with $\cX$, $\cU$ and $\cY$ finite dimensional, say
$\cU = {\mathbb C}^{r}$, $\cY = {\mathbb C}^{s}$, $\cX = {\mathbb
C}^{n}$, so that the system matrix $M$ has the form
\begin{equation}\label{findimsys}
M = \begin{bmatrix} A & B \\ C & D \end{bmatrix} \colon
 \begin{bmatrix} {\mathbb C}^{n} \\ {\mathbb C}^{r} \end{bmatrix} \to
     \begin{bmatrix} {\mathbb C}^{n} \\ {\mathbb C}^{s} \end{bmatrix}
\end{equation}
and the transfer function  $F_{\Sigma}$ is equal to a
rational matrix function of size $s \times r$. Assume that the
realization $(A,B,C,D)$ is {\em minimal}, i.e., the output pair
$(C,A)$ is {\em observable} and the input pair $(A,B)$ is {\em
controllable}:
\begin{equation}\label{obscontr}
 \bigcap_{k=0}^{n} \kr C A^{k} = \{0\}\ands
\textup{span}_{k=0,1,\dots, n-1} \im A^{k} B  = \cX = {\mathbb C}^{n}.
\end{equation}
Then $F_{\Sigma}$ is in the Schur class $\cS({\mathbb C}^{r},
{\mathbb C}^{s})$
if and only if
there is an $n \times n$ positive-definite matrix $H$ satisfying the
KYP-inequality \eqref{KYP1}.
\end{theorem}

There is also a {\em strict} version of the Bounded Real Lemma.  The associated storage function required
is a {\em strict storage function}, i.e., a function $S \colon \cX \to {\mathbb R}_+$ for which there is a number $\delta > 0$
so that
\begin{equation}   \label{diss-strict}
  S(\bx(n+1)) -  S(\bx(n))  +  \delta \| x(n) \|^2 \le (1- \delta) \| \bu(n)\|^2 - \| \by(n) \|^2
\end{equation}
holds over all system trajectories $(\bu(n), \bx(n), \by(n))_{n\in\BZ}$, in addition to the normalization condition $S(0)=0$.
If $S_H(x) = \langle H x, x \rangle$ is a quadratic strict storage function, then the associated linear matrix inequality
is the {\em strict KYP-inequality}
\begin{equation}   \label{KYP2}
\begin{bmatrix} A & B \\ C & D \end{bmatrix}^{*}
\begin{bmatrix} H &  0 \\ 0 & I_{\cY} \end{bmatrix}
\begin{bmatrix} A & B \\ C & D\end{bmatrix}
\prec \begin{bmatrix} H & 0 \\ 0 & I_{\cU}\end{bmatrix}.
\end{equation}
In this case, one also arrives at a stronger condition on the transfer function $F_\Si$, namely that it has an analytic continuation to a function in the
{\em strict Schur class}:
\[
  \cS^{o}(\cU, \cY) =\left \{ F \colon {\mathbb D}
\underset{\text{holo}}
  \mapsto \cL(\cU, \cY) \colon \sup_{z \in {\mathbb D}} \| F(z) \|
\le \rho \text{ for some }
  \rho < 1\right\}.
\]
Note, however, that the strict KYP-inequality implies that $A$ is stable, so that in case \eqref{KYP2} holds, $F_\Si$ is in fact analytic on $\BD$. This is the storage-function interpretation of the following strict Bounded Real Lemma, in which one replaces the minimality condition with a stability condition.

\begin{theorem}[Strict Bounded  Real Lemma (see
\cite{PAJ})]\label{T:BRLfinstrict}
Suppose that the dis\-crete-time linear system $\Si$  is as in
\eqref{dtsystem} with  $\cX$, $\cU$ and $\cY$ finite dimensional, say
$\cU = {\mathbb C}^{r}$, $\cY = {\mathbb C}^{s}$, $\cX = {\mathbb
C}^{n}$, i.e., the system matrix $M$ is as in \eqref{findimsys}.
Assume that $A$ is {\em stable}, i.e., all eigenvalues of $A$ are
inside the open unit disk $\BD$, so that $\spec(A) < 1$ and the
transfer function $F_{\Si}(z)$ is analytic on a neighborhood of
$\overline{\BD}$.  Then
$F_{\Si}(z)$ is in the strict Schur class $\cS^{o}({\mathbb C}^{r},
{\mathbb C}^{s})$ if and only if there is a positive-definite matrix
$H \in {\mathbb C}^{n \times n}$ so that the strict KYP-inequality
\eqref{KYP2} holds.
\end{theorem}

We now turn to the general case, where  the state space $\cX$ and the input space $\cU$ and output space $\cY$ are allowed to be infinite-dimensional. In this case, the results are
more recent, depending on the precise hypotheses.

For generalizations of Theorem \ref{T:BRLfinstan}, much depends on what is meant by minimality of $\Si$, and hence by the corresponding notions of controllable and
observable.  Here are the three possibilities for controllability of an input pair $(A,B)$ which we shall consider.
The third notion involves the controllability operator $\bW_c$ associated with the pair $(A,B)$ tailored to the Hilbert space setup which in general is a closed, possibly unbounded operator with domain $\cD(\bW_c)$ dense in $\cX$ mapping into the Hilbert space $\ell^2_\cU({\mathbb Z}_-)$ of $\cY$-valued sequences supported on the negative integers ${\mathbb Z}_- =\{ -1, -2, -3, \dots \}$, as well as the observability operator $\bW_o$ associated with the pair $(C,A)$, which has similar properties. We postpone precise definitions and properties of these operators to Section \ref{S:review}.

For an input pair $(A,B)$ we define the following notions of controllability:
\begin{itemize}
\item  $(A,B)$ is {\em (approximately) controllable} if the reachability space
\begin{equation}  \label{ReachSpace}
\Rea(A|B) = \operatorname{span}\{\im A^k B \colon k=0,1,2,\dots\}
\end{equation}
 is dense in $\cX$.
\item $(A,B)$ is {\em exactly controllable} if the reachability space $\Rea(A|B)$ is equal
to $\cX$,  i.e., each state vector $x \in \cX$ has a representation as a finite linear combination
$x = \sum_{k=0}^K  A^k B u_k$ for a choice of finitely many input vectors $u_0, u_1, \dots, u_K$
(also known as every $x$ is a {\em finite-time reachable state} (see \cite[Definition 3.3]{OpmeerStaffans2008}).

\item $(A,B)$ is {\em $\ell^2$-exactly controllable} if the $\ell^2$-adapted controllability operator $\bW_c$
has  range equal to all of $\cX$:  $ \bW_c\, \cD(\bW_c) = \cX$.
\end{itemize}
If $(C,A)$ is an output pair, we have the dual notions of observability:
\begin{itemize}
\item $(C,A)$ is {\em (approximately) observable} if the input pair $(A^*, C^*)$ is (approximately) controllable, i.e.,
if the observability space
\begin{equation}   \label{ObsSpace}
  \Obs(C|A) = \operatorname{span} \{ \im A^{*k} C^*  \colon k=0,1,2,\dots\}
\end{equation}
is dense in $\cX$, or equivalently, if $\cap_{k=0}^\infty \ker C A^k = \{0\}$.

\item  $(C,A)$ is {\em exactly observable} if the observability subspace $\Obs(C|A)$ is the whole space $\cX$.

\item $(C,A)$ is {\em $\ell^2$-exactly observable} if the adjoint input pair $(A^*, C^*)$ is
$\ell^2$-exactly controllable,
i.e.,  if the adjoint $\bW_o^*$ of the $\ell^2$-adapted observability operator $\bW_o$ has full range:
$\bW_o^*\, \cD(\bW_o^*) = \cX$.
\end{itemize}
Then we say that the system $\Sigma \sim (A,B,C,D)$ is
\begin{itemize}
\item {\em minimal} if $(A,B)$ is controllable and $(C,A)$ is observable,
\item {\em exactly minimal} if both $(A,B)$ is exactly controllable and $(C,A)$ is exactly observable, and
\item {\em $\ell^2$-exactly minimal} if both $(A,B)$ is $\ell^2$-exactly controllable and $(C,A)$ is $\ell^2$-exactly
observable.
\end{itemize}

Despite the fact that the operators $A$, $B$, $C$ and $D$ associated with the system $\Si$ are all bounded, in the infinite dimensional analogue of the KYP-inequality \eqref{KYP1} unbounded solutions $H$ may appear. We therefore have to be more precise concerning the notion of positive-definiteness we employ. Suppose that $H$ is a (possibly unbounded) selfadjoint operator $H$ on a Hilbert space $\cX$ with domain $\cD(H)$ dense in $\cX$; we refer to \cite{RS} for background and details on this class and other classes of unbounded operators. Then we shall say:
\begin{itemize}
  \item $H$ is {\em strictly positive-definite} (written $H \succ 0$) if there is a $\delta > 0$ so that
$\langle Hx, x \rangle \ge \delta \| x \|^2$ for all $x \in \cD(H)$;

  \item $H$ is {\em positive-definite} if $\langle H x, x \rangle > 0$ for all nonzero $x \in \cD(H)$;

  \item $H$ is {\em positive-semidefinite} (written $H \succeq 0$) if $\langle H x, x \rangle \ge0$ for
all $x \in \cD(H)$.

\end{itemize}
We also note that any (possibly unbounded) positive-semidefinite  operator $H$ has a positive-semidefinite square root $H^\half$; as
$H = H^\half \cdot H^\half$, we have
$$
  \cD(H) = \{ x \in \cD(H^\half) \colon H^\half x \in \cD(H^\half) \} \subset \cD(H^\half).
 $$
 See e.g.\  \cite{RS} for details.

Since solutions $H$ to the corresponding KYP-inequality may be unbounded, the KYP-inequality cannot necessarily be written in the LMI form \eqref{KYP1}, but rather, we require a spatial form of \eqref{KYP1} on the appropriate domain: For a (possibly unbounded) positive-definite operator $H$ on $\cX$ satisfying
\begin{equation} \label{KYP1b'}
A \cD(H^{\half}) \subset \cD(H^{\half}), \quad B \cU \subset
\cD(H^{\half}),
\end{equation}
the spatial form of the KYP inequality takes the form:
\begin{equation}\label{KYP1b}
\left\| \begin{bmatrix} H^{\half} \! & \! 0 \\ 0 \! & \! I_{\cU}
\end{bmatrix} \begin{bmatrix} x \\ u \end{bmatrix} \right\|^{2}
- \left\| \begin{bmatrix} H^{\half} \! & \! 0 \\ 0 \! & \! I_{\cY}
\end{bmatrix} \begin{bmatrix} A\! & \! B \\ C \! & \! D \end{bmatrix}
\begin{bmatrix} x \\ u \end{bmatrix} \right\|^{2} \ge 0 \ \
(x \in \cD(H^{\half}),\, u \in \cU).
\end{equation}
The corresponding notion of a storage function will then be allowed to assume $+\infty$ as a value; this will be made precise in Section \ref{S:Storage}.

With all these definitions out of the way, we can state the following three distinct generalizations of Theorem
\ref{T:BRLfinstan} to the infinite-dimensional situation.

\begin{theorem}[Infinite-dimensional standard Bounded Real Lemma]
\label{T:BRLinfstan}
Let $\Si$ be a discrete-time linear system as in
\eqref{dtsystem} with system matrix $M$ as in \eqref{sysmat} and
transfer function $F_\Si$ defined by \eqref{trans}.
\begin{enumerate}
\item[(1)] Suppose that the system $\Si$ is minimal, i.e., the input pair
$(A,B)$ is controllable and the output pair $(C,A)$ is observable.
Then the transfer function $F_{\Sigma}$ has an analytic continuation
to a function in the Schur class $\cS(\cU, \cY)$ if and only if there
exists a positive-definite solution $H$ of the KYP-inequality in the following generalized sense:  $H$
is a closed, possibly unbounded, densely defined,  positive-definite (and hence injective) operator on $\cX$ such that $\cD(H^\half)$ satisfies \eqref{KYP1b'} and $H$ solves the spatial KYP-inequality \eqref{KYP1b}.

\item[(2)] Suppose that $\Sigma$ is exactly minimal.  Then the transfer function
$F_{\Sigma}$ has an
analytic continuation to a function in the Schur class $\cS(\cU,
\cY)$ if and only if there exists a bounded, strictly positive-definite
solution $H$ of the KYP-inequality \eqref{KYP1}. In this case $A$ has a spectral radius of at most one, and hence $F_{\Sigma}$ is in fact analytic on $\BD$.

\item[(3)]  Statement {\rm(}2{\rm)} above continues to hold if the ``exactly minimal'' hypothesis is replaced
by the hypothesis that $\Sigma$ be ``$\ell^2$-exactly minimal.''

\end{enumerate}
\end{theorem}

We shall refer to a closed, densely defined, positive-definite solution $H$ of \eqref{KYP1b'}--\eqref{KYP1b} as a positive-definite solution of the {\em generalized KYP-inequality}.

The paper of Arov-Kaashoek-Pik \cite{AKP06} gives a penetrating treatment of item (1) in Theorem \ref{T:BRLinfstan}, including examples to illustrate various subtleties surrounding this result---e.g., the fact that the result can fail if one insists on classical bounded  and boundedly invertible selfadjoint solutions of the KYP-inequality.  We believe that items (2) and (3) appeared for the first time  in \cite{KYP1}, where also a sketch of the proof of item (1) is given. The idea behind the proofs of items (1)--(3) in \cite{KYP1} is to combine the result that a Schur-class function  $S$ always has
a contractive realization (i.e., such an $S$ can be realized as $S = F_\Sigma$ for a system $\Sigma$ as in \eqref{dtsystem} with system matrix $M$ in \eqref{sysmat} a contraction operator) with variations of the State-Space-Similarity Theorem (see \cite[Theorem 1.5]{KYP1}) for the infinite-dimensional situation under the conditions that hold in items (1)--(3); roughly speaking, under appropriate hypothesis, a State-Space-Similarity Theorem says that two systems $\Si$ and $\Si'$ whose transfer functions coincide on a neighborhood of zero,  necessarily can be transformed (in an appropriate sense) from one to other via a change
of state-space coordinates.

In the present paper we revisit these three results from a different point of view: we adapt Willems' variational formulas to the infinite dimensional setting, and in this context present the available storage $S_a$ and required supply $S_r$, as well as an $\ell_2$-regularized version $\uS_r$ of the required supply. It is shown, under appropriate hypothesis, that these are storage functions, with $S_a$ and $\uS_r$ being quadratic storage functions, i.e., $S_a$ agrees with $S_{H_a}(x)=\|H_a^\half x\|^2$ and
$\uS_r(x)=S_{H_r}(x)=\|H_r^\half x \|^2$ for $x$ in a suitably large subspace of $\cX$, where $H_a$ and $H_r$ are
possibly unbounded, positive-definite density operators, which turn out to be positive-definite solutions to the generalized KYP-inequality. In this way we will arrive at a proof of item (1). Further analysis of the behavior of $H_a$ and $H_r$, under additional restrictions on $\Si$, lead to proofs of items (2) and (3), as well as the following version of the strict Bounded Real Lemma for infinite dimensional systems, which is a much more straightforward generalization of the result in the finite-dimensional case (Theorem \ref{T:BRLfinstrict}).

\begin{theorem}[Infinite-dimensional strict Bounded Real Lemma]
\label{T:BRLinfstrict}
Let $\Si$ be a dis\-crete-time linear system as in
\eqref{dtsystem} with system matrix $M$ as in \eqref{sysmat} and
transfer function $F_\Si$ defined by \eqref{trans}. Assume that $A$
is exponentially stable, i.e., $\spec(A) < 1$.  Then the transfer
function $F_{\Sigma}$ is in the strict Schur class $\cS^{o}(\cU,
\cY)$ if and only if there exists a bounded strictly positive-definite
solution $H$ of the  strict KYP-inequality \eqref{KYP2}.
\end{theorem}

Theorem \ref{T:BRLfinstrict} was proved by Petersen-Anderson-Jonkheere \cite{PAJ}  for the con\-tinuous-time
finite-dimensional setting by using what we shall call an $\epsilon$-regulariza\-tion procedure to reduce the result to the
standard case Theorem \ref{T:BRLfinstan}.  In \cite{KYP1} we show how this same idea can be used in the infinite-dimensional
setting to reduce the hard direction of Theorem \ref{T:BRLinfstrict} to the result of either of item (2) or item (3) in
Theorem \ref{T:BRLinfstan}.
For the more general nonlinear setting, Willems \cite{Wil72a} was primarily interested in what storage functions look like assuming
that they exist, while in \cite{Wil72b} for the finite-dimensional linear setting he reduced the existence problem to the existence
theory for Riccati matrix equations.  Here we solve the existence problem for the more general infinite-dimensional linear setting
by converting Willems' variational formulation of the available storage $S_a$ and an $\ell^2$-regularized version $\uS_r$ of his
required supply $S_r$ to an operator-theoretic formulation amenable to explicit analysis.

This paper presents a more unified approach to the different variations of the Bounded Real Lemma, in the sense that we present
a pair of concretely defined, unbounded, positive-definite operators $H_a$ and $H_r$ that, under the appropriate conditions,
form positive-definite solutions to the generalized KYP-inequality, and that have the required additional features under the
additional conditions in items (2) and (3) of Theorem \ref{T:BRLinfstan} as well as Theorem \ref{T:BRLinfstrict}.
We also make substantial use of connections with corresponding objects for the adjoint system $\Sigma^*$ (see
\eqref{dtsystem*}) to complete the analysis and arrive at  some order properties for the set of all solutions of the generalized
KYP-inequality which are complementary to those in \cite{AKP06}.

 The paper is organized as follows. Besides the current introduction, the paper consists of seven sections. In Section \ref{S:review}
 we give the definitions of the observability operator $\bW_o$ and controllability operator $\bW_c$ associated with the system $\Si$
 in \eqref{dtsystem} and recall some of their basic properties. In Section \ref{S:Storage} we define what is meant by a storage function
 in the context of infinite dimensional discrete-time linear systems $\Si$ of the form \eqref{dtsystem} as well as  strict and quadratic
 storage functions and we clarify the relations between quadratic (strict) storage functions and solutions to the (generalized)
 KYP-inequality. Section \ref{S:ASRS} is devoted to the available storage $S_a$ and required supply $S_r$, two examples of
 storage functions, in case the transfer function of $\Si$ has an analytic continuation to a Schur class function. It is shown that $S_a$
 and an $\ell^2$-regularized version $\uS_r$ of $S_r$ in fact agree with quadratic storage functions on suitably
 large domain via explicit constructions of two closed, densely defined, positive-definite operators $H_a$ and $H_r$
 that exhibit $S_a$  and $\uS_r$ as quadratic storage functions $S_{H_a}$ and $S_{H_r}$.  In Section \ref{S:dual} we
 make explicit the theory for the adjoint system $\Sigma^*$ and the duality connections between $\Sigma$ and $\Sigma^*$.
 In Section \ref{S:order} we study the order properties of a class of solutions of the generalized KYP-inequality, and
 obtain the conditions under which $H_a$ and $H_r$ are bounded and/or boundedly invertible and thereby
 solutions of the classical KYP-inequality.  These results are then used in Section \ref{S:BRLproof} to give
 proofs of Theorems \ref{T:BRLinfstan} and \ref{T:BRLinfstrict} via the storage function approach.

\section{Review: minimality, controllability, observability}   \label{S:review}

In this section we recall the definitions of the observability operator $\bW_o$ and controllability operator $\bW_c$ associated with the
discrete-time linear system $\Si$ given by \eqref{dtsystem} and various of their basic properties which will be needed in the
sequel. Detailed proofs of most of these results as well as additional properties can be found in \cite[Section 2]{KYP1}.

For the case of a general system $\Sigma$, following
\cite[Section 2]{KYP1}, we define the {\em observability operator} $\bW_{o}$ associated with $\Si$ to be the possibly
unbounded operator with domain $\cD(\bW_{o})$ in $\cX$ given by
\begin{equation} \label{bWo1}
\cD(\bW_{o}) = \{ x \in \cX \colon \{ C A^{n} x\}_{n \ge 0}
\in\ell^{2}_{\cY}({\mathbb Z}_{+})\}
\end{equation}
with action given by
\begin{equation}   \label{bWo2}
\bW_{o} x =  \{ C A^{n}  x\}_{n \ge 0} \text{ for } x \in
\cD(\bW_{o}).
\end{equation}
Dually, we define the {\em adjoint controllability operator} $\bW_{c}^{*}$ associated with $\Si$ to have domain
\begin{equation}   \label{bWc*1}
\cD(\bW_{c}^{*}) = \{ x \in \cX \colon \{B^* A^{*(-n-1)} x\}_{n\le
-1} \in\ell^{2}_{\cU}({\mathbb Z}_{-})\}
\end{equation}
with action given by
\begin{equation}   \label{bWc*2}
\bW_{c}^{*} x = \{B^* A^{*(-n-1)} x\}_{n\le -1} \text{ for } x \in
\cD(\bW_{c}^{*}).
\end{equation}
It is directly clear from the definitions of $\bW_o$ and $\bW_c^*$ that
\begin{equation}\label{KerWoWc}
\ker \bW_o=\Obs(C|A)^\perp\ands
\ker \bW_c^*=\Rea(A|B)^\perp.
\end{equation}

We next summarize the basic properties of $\bW_c$ and $\bW_o$.

\begin{proposition}[Proposition 2.1 in \cite{KYP1}]  \label{P:WcWo'}
Let $\Sigma$ be a system as in \eqref{dtsystem} with observability operator $\bW_o$ and adjoint controllability
operator $\bW_c^*$ as in \eqref{bWo1}--\eqref{bWc*2}.  Basic properties of the controllability operator $\bW_c$ are:
\begin{enumerate}
\item[(1)] It is always the case that $\bW_o$ is a closed operator on its domain \eqref{bWo1}.

\item[(2)] If $\cD(\bW_o)$ is dense in $\cX$, then
the adjoint $\bW_o^*$ of $\bW_o$ is a closed and
densely defined operator, by a general property of adjoints of closed operators with
dense domain.  Concretely for the case here,
$\cD(\bW_o^*)$ contains the dense linear manifold
$\ell_{\tu{fin},\cY}(\BZ_+)$ consisting of  finitely supported sequences in
$\ell^2_\cY(\BZ_+)$.  In general, one can characterize $\cD(\bW_{o}^{*})$ explicitly
as the set of all $\by \in \ell^{2}_{\cY}({\mathbb Z}_{+})$ such that
there exists a vector $x_{o} \in \cX$ such that the limit
$$
\lim_{K \to \infty}\langle  x, \sum_{k=0}^{K} A^{*k} C^{*} \by(k)
 \rangle_{\cX}
 $$
  exists for each $x \in \cD(\bW_o)$ and is given  by
\begin{equation}   \label{limit-o}
 \lim_{K \to \infty}\langle  x, \sum_{k=0}^{K} A^{*k} C^{*} \by(k)
 \rangle_{\cX} = \langle x, x_{o} \rangle_{\cX},
\end{equation}
with action of $\bW_c$ then given by
\begin{equation}   \label{Wo*act}
    \bW_{o}^{*} \by = x_{o}
 \end{equation}
 where $x_{o}$ is as in \eqref{limit-o}.
  In particular, $\ell_{\tu{fin}, \cY}({\mathbb Z}_+)$ is contained in $\cD(\bW_o^*)$ and the observability
space defined in \eqref{ObsSpace} is given by
$$
\Obs (C|A) = \bW_{o}^{*} \ell_{\tu{fin}, \cY}({\mathbb Z}_{+}).
$$
Thus, if in addition $(C,A)$ is observable, then $\bW_o^*$ has dense range.
\end{enumerate}
Dual properties of the controllability operator $\bW_c^*$ are:
\begin{enumerate}
\item[(3)] It is always the case that the adjoint controllability operator $\bW_c^*$ is closed on its domain \eqref{bWc*1}.

\item[(4)]  If $\cD(\bW_c^*)$ is dense in $\cX$,  then
the controllability operator $\bW_c = (\bW_c^*)^*$ is closed and densely defined by a general property of the adjoint
of a closed and densely defined operator.
Concretely for the case here,  $\cD(\bW_c)$ contains the dense  linear manifold
$\ell_{\tu{fin},\cU}(\BZ_-)$ of finitely supported sequences in
$\ell^2_\cU(\BZ_-)$.  In general,  one can characterize $\cD(\bW_{c})$ explicitly as
the set of all $\bu \in \ell^{2}_{\cU}({\mathbb Z}_{-})$ such that
there exists a vector $x_{c} \in \cX$ so that
$$ \lim_{K \to \infty} \langle x, \sum_{k=-K}^{-1} A^{-k-1} B \bu(k)
    \rangle_{\cX}
 $$
 exists for each $x \in \cD(\bW_{c}^{*})$ and is given by
\begin{equation}  \label{limit-c}
\lim_{K \to \infty} \langle x, \sum_{k=-K}^{-1} A^{-k-1} B \bu(k)
    \rangle_{\cX} = \langle x, x_{c} \rangle_{\cX},
\end{equation}
and  action of $\bW_{c}$ then given by
\begin{equation}  \label{Wc-act}
    \bW_{c} \bu = x_{c}
\end{equation}
where $x_{c}$ is as in \eqref{limit-c}.
 In particular, the reachability space
$\Rea (A|B)$ is equal to $\bW_{c} \ell_{{\rm fin}, \cU}({\mathbb Z}_{-})$.
Thus, if in addition $(A,B)$ is controllable, then $\bW_c$ has dense range.
\end{enumerate}
\end{proposition}

For systems $\Si$ as in \eqref{dtsystem}, without additional conditions, it can happen that $\bW_o$ and/or $\bW_c^*$ are not densely defined, and therefore
the adjoints $\bW_o^*$ and $\bW_c$ are at best linear relations and difficult to work with.  However, our interest here
is the case where the transfer function $F_\Sigma$ has analytic continuation to a  bounded function on the unit disk (or even in the Schur class, i.e., norm-bounded by $1$ on the unit disk).  In this case the multiplication operator
\begin{equation}  \label{mult-op}
  M_{F_\Sigma} \colon f(\lambda) \mapsto F_\Sigma(\lambda) f(\lambda)
\end{equation}
is a bounded operator from $L^2_\cU({\mathbb T})$ to $L^2_\cY({\mathbb T})$
and hence also its compression to a map ``from past to future''
\begin{equation}  \label{freq-Hankel}
  {\mathbb H}_{F_\Sigma} = P_{H^2_\cY({\mathbb D})} M_{F_\Sigma}|_{H^2_\cU({\mathbb D})^\perp},
\end{equation}
often called the {\em Hankel operator} with symbol $F_\Sigma$, is also bounded (by $\| M_{F_\Sigma} \|$).
If we take inverse $Z$-transform to represent $L^2({\mathbb T})$ as $\ell^2({\mathbb Z})$, $H^2({\mathbb D})$ as
$\ell^2({\mathbb Z}_+)$ and $H^2({\mathbb D})^\perp$ as $\ell^2({\mathbb Z}_-)$, then the frequency-domain Hankel
operator
$$
{\mathbb H}_{F_\Sigma} \colon H^2_\cU({\mathbb D})^\perp \to H^2_\cY({\mathbb D})
$$
given by \eqref{freq-Hankel}  transforms via inverse $Z$-transform to the time-domain Hankel operator $\fH_{F_\Sigma}$
with matrix representation
\begin{equation}  \label{Hankel-matrix}
  \fH_{F_\Sigma} = [ C A^{i-j-1} B ]_{i \ge 0, j<0} \colon \ell^2_\cU({\mathbb Z}_-) \to \ell^2_\cY({\mathbb Z}_+).
\end{equation}
We conclude that the Hankel matrix $\fH_{F_\Sigma}$ is bounded as an operator from $\ell^2_\cU({\mathbb Z}_-)$ to
$\ell^2_\cY({\mathbb Z}_+)$ whenever $F_\Sigma$ has analytic continuation to an $H^\infty$ function.
From the matrix representation \eqref{Hankel-matrix} we see that the Hankel matrix  formally has a factorization
\begin{equation} \label{formal-Hank-fact}
 \fH_{F_\Sigma} = \tu{col} [C A^i]_{i \ge 0} \cdot \tu{row} [A^{-j-1} B]_{j<0} = \bW_o \cdot \bW_c.
\end{equation}
It can happen that $\fH_{F_\Sigma}$ is bounded while $\bW_o$ and $\bW_c$ are unbounded.
Nevertheless, from the fact that $\fH_{F_\Sigma}$ is bounded one can see that $\Rea (A|B)$ is in $\cD(\bW_o)$
and
$$
 \fH_{F_\Sigma} \bu = \bW_o \left( \sum_{k=K}^{-1} A^{-1-k} B \bu(k) \right) \in \ell^2_\cY({\mathbb Z}_+).
$$
for each finitely supported input string $\bu(K), \dots, \bu(-1)$.  If we assume that $(A,B)$ is controllable,
we conclude that $\bW_o$ is densely defined.  Similarly, by working with boundedness of $\fH_{F_\Sigma}^*$
one can show that boundedness of $F_\Sigma$ on ${\mathbb D}$ leads to $\cD(\bW_c^*)$ containing the
observability space $\Obs(C|A)$; hence if we assume that $(C,A)$ is observable, we get that $\bW_c^*$ is
densely defined.    With these observations in hand, the following precise version of the formal factorization
 \eqref{formal-Hank-fact} for the case where $\bW_o$ and $\bW_c$ may be unbounded becomes plausible.

 \begin{proposition}[Corollary 2.4 and Proposition 2.6 in \cite{KYP1}]  \label{P:HankelDecs}
 Suppose that the system $\Sigma$ given by \eqref{dtsystem} has transfer function $F_\Sigma$ with analytic continuation
 to an $H^\infty$-function on the unit disk ${\mathbb D}$.
 \begin{enumerate}
    \item[(1)]
    Assume that $\cD(\bW_c^*)$ is dense in $\cX$ {\rm(}as is the case if $(C,A)$ is observable{\rm)}.  Then
    $\cD(\bW_{o})$ contains $\im \bW_{c} = \bW_{c} \cD(\bW_{c})$ and
\begin{equation}  \label{HankDec1}
\fH_{F_{\Sigma}}|_{\cD(\bW_{c})} = \bW_{o} \bW_{c}.
\end{equation}
In particular, as $\ell_{{\rm fin}, \cU}({\mathbb Z}_-) \subset \cD(\bW_c)$ and $\bW_c
\ell_{{\rm fin}, \cU}({\mathbb Z}_-) = \Rea(A|B)$ {\rm(}from Proposition \ref{P:WcWo'} {\rm(4))}, it follows that $\Rea(A|B) \subset \cD(\bW_o)$.

\item[(2)]
Assume that $\cD(\bW_o)$ is dense in $\cX$ {\rm(}as is the case if $(A,B)$ is controllable{\rm)}.  Then
$\cD(\bW_{c}^{*})$ contains $\im \bW_{o}^{*} = \bW_{o}^{*}\cD(\bW_{o}^{*})$ and
\begin{equation}   \label{HankDec2}
\fH_{F_{\Sigma}}^{*}|_{\cD(\bW_{o}^{*})} = \bW_{c}^{*} \bW_{o}^{*}.
\end{equation}
In particular, as $\ell_{{\rm fin}, \cY}({\mathbb Z}) \subset \cD(\bW_o^*)$ and
$\bW_o^* \ell_{{\rm fin}, \cY}({\mathbb Z}_+)
= \Obs(C|A)$ {\rm(}from Proposition \ref{P:WcWo'} {\rm(2))}, it follows that $\Obs(C|A) \subset \cD(\bW_c^*)$.

\item[(3)]  In case the system matrix $M = \sbm{A & B \\ C & D}$ is contractive,
then $\bW_o$ and $\bW_c$ also are bounded contraction operators and we have the bounded-operator factorizations
\begin{equation}   \label{HankDec3}
\fH_{F_\Sigma} = \bW_o \bW_c, \quad (\fH_{F_\Sigma})^* = \bW_c^* \bW_o^*.
\end{equation}
\end{enumerate}
\end{proposition}

The following result from \cite{KYP1} describes the implications of $\ell^2$-exact controllability and $\ell^2$-exact observability
on the operators $\bW_o$ and $\bW_c$

\begin{proposition}[Corollary 2.5 in \cite{KYP1}] \label{P:ell2implics}
Let $\Si$ be a  discrete-time linear system as in
\eqref{dtsystem} with system matrix
$M$ as in \eqref{sysmat}. Assume that the transfer function $F_\Si$
defined by \eqref{trans}
has an analytic continuation to an
$H^{\infty}$-function on ${\mathbb D}$.
\begin{itemize}
\item[(1)] If $\Sigma$ is $\ell^2$-exactly controllable, then $\bW_o$
is bounded.

\item[(2)] If $\Sigma$ is $\ell^2$-exactly observable, then $\bW_c$
is bounded.

\item[(3)] $\Sigma$ is $\ell^2$-exactly minimal, i.e., both
$\ell^2$-exactly controllable and $\ell^2$-exactly observable,
then $\bW_o$ and $\bW_c^*$ are both bounded and bounded below.
\end{itemize}
\end{proposition}

The following result will be useful in the sequel.

\begin{proposition}  \label{P:Wc}   Suppose that the discrete-time linear system $\Sigma$ given by \eqref{dtsystem} is
minimal and that its transfer function $F_\Sigma$ has analytic continuation to an $H^\infty$-function on ${\mathbb D}$,
so (by Propositions   \ref{P:WcWo'} and \ref{P:HankelDecs}) $\cD(\bW_c^*) \supset \Obs(C|A)$ is dense in $\cX$
and $\bW_c = (\bW_c^*)^*$ is densely defined with dense range $\im(\bW_c) \supset \Rea(A|B)$.

\begin{enumerate}
\item[(1)]  Suppose that $(\bu(n), \bx(n), \by(n))_{n \ge n_{-1}}$ is a system trajectory of $\Si$ with initialization $\bx(n_{-1}) = 0$.
Define an input string $\bu' \in \ell_{{\rm fin},\cU}({\mathbb Z}_-)$ by
$$
   \bu'(n) = \begin{cases}  0 &\text{if } n < n_{-1}, \\
     \bu(n) &\text{if } n_{-1} \le n < 0.  \end{cases}
$$
Then $\bx(0) = \bW_c \bu'$.

\item[(2)]  Suppose that $\bu \in \ell^2_\cU({\mathbb Z}_-)$ is in $\cD(\bW_c)$ and $\widetilde u \in \cU$.
Define a new input string $\bu' \in \ell^2_\cU({\mathbb Z}_-)$ by
$$
   \bu'(n) = \begin{cases}  \bu(n+1) & \text{if } n < -1, \\
              \widetilde u & \text{if } n= -1.   \end{cases}
$$
Then  $\bu' \in \cD(\bW_c)$ and
$$
    \bW_c \bu' = A \bW_c \bu + B \widetilde u.
$$
\end{enumerate}
\end{proposition}

\begin{proof}
We start with item (1). From item (4) of Proposition \ref{P:WcWo'} see that $\ell_{{\rm fin}, \cU}({\mathbb Z}_+)$ is contained in $\cD(\bW_c)$, and thus $\bu'\in \cD(\bW_c)$. From formula \eqref{limit-c} for the action of $\bW_c$ on its domain we obtain that
\begin{equation}  \label{Wc-fin}
  \bW_c \bu' = \sum_{k \in {\mathbb Z}_-} A^{-k-1} B \bu'(k)
  =\sum_{k = n_{-1}}^{-1} A^{-k-1} B \bu(k)
\end{equation}
where the sum is well defined since there are only finitely many nonzero terms.
By a standard induction argument, using the input-state equation in \eqref{dtsystem}, one verifies that this is the formula for $\bx(0)$ for a system trajectory $(\bu(n), \bx(n), \by(n))_{n \ge n_{-1}}$ with initialization $\bx(n_{-1}) = 0$.
This verifies (1).

As for item (2), it is easily verified that $\cD(\bW_c^*)$ is invariant under $A^*$ and that the following intertwining condition holds:
$$
    \bW_c^* A^*|_{\cD(\bW_c^*)} = \cS_- \bW_c^*,
$$
with $\cS_-$ the truncated right shift operator on $\ell^2_\cU({\mathbb Z}_-)$ given by
$$
   ( \cS_- \bu)(n) = \bu(n-1) \text{ for } n \in {\mathbb Z}_-.
$$
The adjoint version of this is that $\cD(\bW_c)$ is invariant under the untruncated left shift operator $\cS_-^*$
on $\ell^2_\cU({\mathbb Z}_-)$
$$
    (\cS_-^* \bu)(n) = \begin{cases}  \bu(n+1) &\text{if } n < -1, \\
             0 &\text{if } n=-1  \end{cases}
$$
and we have the intertwining condition
$$
  \bW_c \cS_-^*|_{\cD(\bW_c)} = A \bW_c.
$$
Next note that $\cS_-^* \bu= \bu' - \Pi_{-1} \wtil{u}$, with $\Pi_{-1}:\cU\to\ell^2_{\cU}(\BZ_-)$ the embedding of $\cU$ into the $-1$-th entry of $\ell^2_{\cU}(\BZ_-)$. This implies that
\[
\bu'=\cS_-^* \bu + \Pi_{-1} \wtil{u}\in \cS_-^* \cD(\bW_c)+ \ell_{{\rm fin},\cU}({\mathbb Z}_-)\subset \cD(\bW_c),
\]
and
\begin{equation}   \label{intertwine1}
A \bW_c\bu
=\bW_c \cS_-^*|_{\cD(\bW_c)}\bu
=\bW_c (\bu' -\Pi_{-1}\wtil{u})=\bW_c \bu' -B\wtil{u},
\end{equation}
which provides the desired identity.
\end{proof}

\begin{remark}  \label{R:Wc}  It is of interest to consider the shift $\bW_c^{(1)}$ of the controllability operator
$\bW_c$ to the interval $(-\infty, 0]$ in place of ${\mathbb Z}_- = (-\infty, 0)$, i.e.,
$$
  \bW_c^{(1)} = \bW_c \tau^{-1}
$$
where the map $\tau$ transforms sequences $\bu$ supported on ${\mathbb Z}_-  = (-\infty, 0)$ to sequences $\bu'$
supported on $(-\infty, 0]$ according to the action
$$
   (\tau \bu)(n) = \bu(n+1)\quad \text{ for } n < 0
$$
with inverse given by
$$
  (\tau^{-1} \bv)(n) = \bv(n-1)\quad \text{ for } n \le 0.
$$
For all $\bu \in \ell^2_\cU({\mathbb Z}_-)$ and $\widetilde u \in \cU$, define a sequence
$(\bu, \widetilde u) \in \ell^2_\cU((-\infty, 0])$ by
$$
  (\bu, \widetilde u)(n) = \begin{cases} \bu(n) &\text{if } n \in {\mathbb Z}_-, \\
        \widetilde u &\text{if } n=0.  \end{cases}
$$
The result of item (2) in Proposition \ref{P:Wc} can be interpreted as saying: given $\bu \in \ell^2_\cU({\mathbb Z}_-)$ and $\widetilde u \in \cU$ we have
\[
(\bu, \widetilde u) \in \cD(\bW_c^{(1)}) \quad \Longleftrightarrow \quad
\bu \in \cD(\bW_c)
\]
and in that case $ \bW_c^{(1)} (\bu, \widetilde u)  = A \bW_c \bu + B \widetilde u$.
\end{remark}

\section{Storage functions} \label{S:Storage}

In the case of systems with an infinite dimensional state space we allow
storage functions to also attain $+\infty$ as a value. Set
$[0,\infty]:= \BR_+\cup\{+\infty\}$. Then, given a discrete-time
linear system $\Si$ as in \eqref{dtsystem}, we say that a function $S
\colon \cX \to [0, \infty]$ is a {\em storage function} for the
system $\Sigma$ if the dissipation inequality
\begin{equation}\label{disineq}
S(\bx(n+1)) \le S(\bx(n)) + \| \bu(n) \|_{\cU}^{2} - \|
\by(n)\|_{\cY}^{2}  \text{ for } n \ge N_0
\end{equation}
holds along all system trajectories $(\bu(n), \bx(n),
\by(n))_{n \ge N_0}$ with state initialization $x(N_0) = x_0$ for some $x_0 \in \cX$
at some $N_0 \in {\mathbb Z}$,
and $S$ is normalized to satisfy
\begin{equation}\label{normalization'}
S(0) = 0.
\end{equation}

As a first result we show that existence of a storage function for $\Si$ is a sufficient condition for the transfer function to have an analytic continuation to a Schur class function.

\begin{proposition}\label{P:storage-Schur}
Suppose that the system $\Sigma$ in \eqref{dtsystem} has a storage
function $S$.  Then the transfer function $F_{\Sigma}$ of $\Si$ defined in
\eqref{trans} has an analytic continuation to a function
in the Schur class $\cS(\cU, \cY)$.
 \end{proposition}

 The proof of Proposition \ref{P:storage-Schur} relies on the
following observation, which will also be of use in the sequel.

\begin{lemma}\label{L:finH2}
Suppose that the system $\Sigma$ in \eqref{dtsystem} has a storage
function $S$. For each system trajectory
$(\bu(n),\bx(n),\by(n))_{n\in\BZ}$ and $N_0\in\BZ$ so that
$\bx(N_0)=0$, the following inequalities hold for all $N\in\BZ_+$:
\begin{align}
S(\bx(N_0+N+1))&\le \sum_{n=N_0}^{N_0+N} \| \bu(n)\|_{\cU}^{2} -
\sum_{n=N_0}^{N_0+N} \| \by(n)
 \|^{2}_{\cY};\label{Sbound}\\
\sum_{n=N_0}^{N_0+N} \| \by(n) \|^{2}_{\cY} &\le \sum_{n=N_0}^{N_0+N}
\| \bu(n)
\|^{2}_{\cU}. \label{IOdisineq}
\end{align}
\end{lemma}

\begin{proof}
By the translation invariance of the system $\Si$ we may assume
without loss of generality that $N_0=0$, i.e., $\bx(0)=0$. From
\eqref{disineq} and \eqref{normalization'} we get
\[
S(\bx(1)) \le \| \bu(0)\|^{2} - \| \by(0) \|^{2} + S(0) = \|
 \bu(0)\|^{2} - \| \by(0) \|^{2} < \infty.
\]
Inductively, suppose that $S(\bx(n)) < \infty$.  Then \eqref{disineq}
gives us
\[
S(\bx(n+1)) \le \| \bu(n) \|^{2}_{\cU} - \| \by(n) \|^{2}_{\cY} +
S(\bx(n)) <
   \infty.
\]
We may now rearrange the dissipation inequality for $n\in\BZ_+$ in
the form
\begin{equation}   \label{difdis}
S(\bx(n+1)) - S(\bx(n)) \le \| \bu(n) \|^{2} - \| \by(n) \|^{2} \quad
(n\in\BZ_+).
\end{equation}
Summing from $n=0$ to $n=N$ gives
\[
0 \le S(\bx(N+1)) \le \sum_{n=0}^{N} \| \bu(n)\|_{\cU}^{2} -
\sum_{n=0}^{N} \| \by(n)
 \|^{2}_{\cY},
\]
which leads to
\[
\sum_{n=0}^{N} \| \by(n) \|^{2}_{\cY} \le \sum_{n=0}^{N} \| \bu(n)
\|^{2}_{\cU} \text{ for all } N \in {\mathbb Z}_{+}.
\]
These inequalities prove \eqref{Sbound} and \eqref{IOdisineq} for
$N_0=0$. As observed above, the case of $N_0\not=0$ is then obtained
by translation of the system trajectory.
\end{proof}

\begin{proof}[Proof of Proposition \ref{P:storage-Schur}]
Let $\bu \in \ell^{2}_{\cU}({\mathbb Z}_{+})$ and run the system
$\Sigma$ with input sequence $\bu$ and initial condition $\bx(0)= 0$.
From Lemma \ref{L:finH2}, with $N_0=0$, we obtain that for each
$N\in\BZ_+$ we have
\[
\sum_{n=0}^{N} \| \by(n) \|^{2}_{\cY} \le \sum_{n=0}^{N} \| \bu(n)
\|^{2}_{\cU}\ \text{ for all }\ N \in {\mathbb Z}_{+}.
\]
Letting $N \to \infty$, we conclude that $\bu
\in\ell^{2}_{\cU}({\mathbb Z}_{+})$ implies that the output sequence
$\by$ is in $\ell^{2}_{\cY}({\mathbb Z}_{+})$ with $\| \by
\|^{2}_{\ell^{2}_{\cY}({\mathbb Z}_{+})} \le \| \bu
\|^{2}_{\ell^{2}_{\cU}({\mathbb Z}_{+})}$.

Write $ \widehat u$ and $ \widehat y$ for the $Z$-transforms of $\bu$
and $\by$, respectively, i.e., $\widehat u(z) = \sum_{n=0}^{\infty}
\bu(n) z^{n}$ and  $\widehat y(z)= \sum_{n=0}^{\infty} \by(n) z^{n}$.
Since we have imposed zero-initial condition on the state, it now
follows that $\widehat y(z) = F_{\Sigma}(z) \widehat u(z)$ in a
neighborhood of 0. Since $\bu$ was chosen arbitrarily in
$\ell^{2}_\cU(\BZ_+)$, we see that $\widehat u$ is an arbitrary
element
of $H^2_\cU(\BD)$. Thus, the multiplication operator $M_{F_\Si}
\colon \widehat u \mapsto F_\Si \cdot\widehat u$ maps $H^2_\cU(\BD)$
into $H^2_\cY(\BD)$. In particular, taking $\widehat u\in
H^2_\cU(\BD)$ constant, it follows that $ F_\Si$ has an analytic
continuation to $\BD$. Furthermore, the inequality
\[
\|F_\Si \widehat u\|_{H^2_\cY(\BD)}=\|\widehat y\|_{H^2_\cY(\BD)}=\|
\by \|^{2}_{\ell^{2}_{{\mathbb Z}_{+}}(\cY)} \le \| \bu
\|^{2}_{\ell^{2}_{{\mathbb Z}_{+}}(\cU)}=\|\widehat
u\|^{2}_{H^2_\cU(\BD)},
\]
implies that the operator norm of the multiplication operator
$M_{F_\Si}$  from $H^{2}_{\cU}({\mathbb D})$ to $H^{2}_{\cY}({\mathbb
D})$ is at most 1. It is well known that the operator norm of
$M_{F_\Si}$ is the same as the supremum norm $\| F_\Si \|_{\infty} =
\sup_{z \in {\mathbb D}} \| F_\Si(z) \|$. Hence we obtain that the
analytic continuation of $F_\Si$ is in the Schur class $\cS(\cU,
\cY)$.
\end{proof}

We shall see below (see Proposition \ref{P:SaSr}) that  conversely, if the transfer function $F_\Sigma$ admits an analytic continuation to a Schur class function, then a storage function for $\Si$ exists.\medskip

\paragraph{Quadratic storage functions}   \label{S:QuadStorage}

The class of storage functions associated with solutions to the generalized KYP inequality
\eqref{KYP1b'}--\eqref{KYP1b} are the so-called {\em quadratic storage functions} described next.
We shall say that a storage function $S$ is {\em quadratic} in case there is a positive-semidefinite operator
$H$ on the state space $\cX$ so that $S$ has the form
\begin{equation}\label{QuadStorage1}
S(x)=S_H(x)= \begin{cases}  \| H^\half x \|^2 &\text{for } x \in \cD(H^\half), \\
    +\infty &\text{ otherwise.}  \end{cases}
\end{equation}

If in addition to $F_\Si$ having an analytic continuation to a Schur class function it is assumed that $\Si$ is minimal, it can in
fact be shown (see Theorem \ref{T:Sar}  below) that quadratic storage functions for $\Si$ exist; for the finite dimensional case see \cite{Wil72b}.

\begin{proposition}\label{P:QuadStorage}
Suppose that the function $S \colon \cX \to [0, \infty]$ has the form \eqref{QuadStorage1} for a
(possibly) unbounded positive-semidefinite operator $H$ on $\cX$.  Then $S_H$ is a storage function for $\Sigma$ if and only if $H$ is a positive-semidefinite solution of the generalized KYP-inequality \eqref{KYP1b'}--\eqref{KYP1b}.  Moreover, $S$ is {\em nondegenerate}
in the sense that $S_H(x) > 0$ for all nonzero $x$ in $\cX$ if and only if $H$ is positive-definite.
\end{proposition}

\begin{proof}
Suppose that $H$ solves \eqref{KYP1b'}--\eqref{KYP1b}. It is clear that $S(0)=\|H^\half 0\|^2=0$, so in order to conclude that $S$ is a storage function it remains to verify the dissipation inequality \eqref{disineq}.
Let $(\bu(n), \bx(n), \by(n))_{n \ge N_0}$ be a system trajectory with state initialization $\bx(n_0)=x_0$ for some $x_0\in\cX$ and $N_0\in\BZ$. Fix $n \ge N_0$.  If $\bx(n) \notin \cD(H^\half)$, then $S_H(\bx(n)) = \infty$
and the dissipation inequality \eqref{disineq} is automatically satisfied.
If  $\bx(n) \in \cD(H^\half)$, then \eqref{KYP1b'} implies that $\bx(n+1)=A\bx(n)+B \bu(n)\in \cD(H^\half)$. Thus $S_H(\bx(n+1))<\infty$. Replacing $x$ by $\bx(n)$ and $u$ by $\bu(n)$ in \eqref{KYP1b} and applying \eqref{dtsystem} we obtain that
\[
\left\|\mat{cc}{H^\half &0\\ 0& I_\cU}\mat{c}{\bx(n)\\ \bu(n)}\right\|^2
-\left\|\mat{cc}{H^\half&0\\ 0& I_\cY}\mat{c}{\bx(n+1)\\ \by(n)}\right\|^2
\geq 0.
\]
This can be rephrased in terms of $S_H$ as
\[
S_H(\bx(n))+\|\bu(n)\|^2-S_H(\bx(n+1))-\|\by(n)\|^2\geq 0,
\]
so that \eqref{disineq} appears after adding $S_H(\bx(n+1))$ on both sides.

Conversely, suppose that $S_H$ is a storage function.  Take  $x \in \cX$ and $u \in \cU$ arbitrarily.  Let $(\bu(n), \bx(n), \by(n))_{n \ge 0}$
be any system trajectory with initialization $\bx(0) = x$ and with $\bu(0) = u$.  Then
the dissipation inequality \eqref{disineq} with $n=0$ gives us
\begin{equation}\label{eqnStoKYP}
S_H(Ax + Bu)  \le S_H(x) + \| u \|^2 - \|y \|^2,\quad\mbox{ with }\quad y=Cx+Du.
\end{equation}
In particular, $S_H(x) < \infty$ (equivalently, $x \in \cD(H^\half)$) implies that $S_H(Ax + Bu) < \infty$
(equivalently, $Ax + B u \in \cD(H^\half)$). Specifying $u=0$ shows that $A \cD(H^\half)\subset \cD(H^\half)$ and specifying $x=0$
shows $B\cU \subset \cD(H^\half)$. Thus \eqref{KYP1b'} holds. Bringing $\|y\|^2$ in \eqref{eqnStoKYP} to the other side and writing out $S_H$ gives
\[
\|H^\half (Ax + Bu)\|^2 +\|Cx+Du\|^2  \le \|H^\half x\|^2 + \| u \|^2,
\]
which provides \eqref{KYP1b}.
\end{proof}

We say that a function $S \colon \cX \to {\mathbb R}_+=[0,\infty)$ is a {\em strict storage function} for the system $\Sigma$ in \eqref{dtsystem} if the strict dissipation inequality \eqref{diss-strict} holds, i.e., if there exists a $\delta > 0$ so that
\begin{equation}   \label{diss-strict2}
  S(\bx(n+1)) -  S(\bx(n))  + \delta \| x(n)\|^2 \le (1- \delta) \| \bu(n)\|^2 - \| \by(n) \|^2\quad (n\ge N_0)
\end{equation}
holds for all system trajectories $\{ \bu(n), \bx(n), \by(n)\}_{n \ge N_0}$, initiated at some $N_0 \in {\mathbb Z}$.
 Note that strict storage functions are not allowed to attain $+\infty$ as a value. The significance of the existence of a
strict storage function for a system $\Sigma$ is that it guarantees that the transfer function $F_\Sigma$ has
analytic continuation to a $H^\infty$-function with $H^\infty$-norm strictly less than 1 as well as a coercivity condition on $S$,
 i.e., we have the following strict version of Proposition \ref{P:storage-Schur}.

\begin{proposition}  \label{P:strictstorage-Schur}
Suppose that the system $\Sigma$ in \eqref{dtsystem} has a strict storage function $S$.  Then
\begin{enumerate}
\item[(1)] the transfer function $F_\Sigma$ has analytic continuation to a function in $H^\infty$ on the unit disk ${\mathbb D}$
with $H^\infty$-norm strictly less than 1, and
\item[(2)] $S$ satisfies a coercivity condition, i.e., there is a $\delta > 0$ so that
\begin{equation}   \label{coercive}
  S(x) \ge \delta \| x \|^2\quad (x\in\cX).
\end{equation}
\end{enumerate}
\end{proposition}

\begin{proof}
Assume that $S \colon \cX \to [0, \infty)$ is a strict storage function for $\Sigma$. Then for each system trajectory
$(\bu(n), \bx(n), \by(n)))_{n \ge 0}$  with initialization $\bx(0) = 0$, the strict dissipation inequality \eqref{diss-strict2} gives that there
is a $\delta > 0$ so that for $n\geq 0$ we have
\begin{align*}
S(\bx(n+1)) - S(\bx(n)) & \le  - \delta \| x \|^2 + (1- \delta) \| \bu(n) \|^2 - \| \by(n) \|^2 \\
& \le (1- \delta) \| \bu(n) \|^2 - \| \by(n) \|^2.
\end{align*}
Summing up over $n=0,1,2,\dots, N$ for some $N \in {\mathbb N}$  for a system trajectory $(\bu(n), \bx(n), \by(n))_{n\ge 0}$ subject to
initialization $\bx(0) = 0$ then gives
$$
0  \le S(\bx(N+1)) = S(\bx(N+1)) - S(\bx(0))  \le (1-\delta) \sum_{n=0}^N \| \bu(n) \|^2 - \sum_{n=0}^N \| \by(n) \|^2.
$$
By restricting to input sequences $\bu\in \ell^2_\cU(\BZ_+)$, it follows that the corresponding output sequences satisfy
$\by\in\ell^2_\cY(\BZ_+)$ and $\|\by\|_{\ell^2_\cU(\BZ_+)}^2 \le (1 - \delta)  \|\bu\|_{\ell^2_\cY(\BZ_+)}^2$.
Taking $Z$-transform and using the Plancherel theorem then gives
$$
 \| M_{F_\Sigma} \widehat \bu \|^2_{H^2_\cY({\mathbb D})}=
\| \widehat \by \|^2_{H^2_\cY({\mathbb D})}\le
(1- \delta) \| \widehat \bu \|^2_{H^2_\cU({\mathbb D})}.
$$
Thus $\|M_{F_\Sigma}\|\leq \sqrt{1-\delta} < 1$. This implies $F_\Sigma$ has analytic continuation to an $\cL(\cU, \cY)$-valued $H^\infty$ function with
$H^\infty$-norm at most $\|M_{F_\sigma}\|\leq\sqrt{1-\delta} < 1$.

To this point we have not made use of the presence of the term $\delta \| x(n) \|^2$ in the strict dissipation inequality \eqref{diss-strict2}.
We now show how the presence of this term leads to the validity of the coercivity condition \eqref{coercive} on $S$.  Let $x_0$
be any state in $\cX$ and let $(\bu(n), \bx(n), \by(n))_{n\ge 0}$ be any system trajectory with initialization $\bx(0) = x_0$ and $\bu(0) = 0$.
Then the strict dissipation inequality \eqref{diss-strict2} with $n=0$ gives us
$$
\delta \| x_0 \|^2 = \delta \| \bx(0) \|^2 \le S(\bx(1)) + \delta \| \bx(0)\|^2 + \| \by(0) \|^2 \le S(\bx(0)) = S(x_0),
$$
i.e., $S(x_0) \ge \delta \| x_0 \|^2$ for each $x_0 \in \cX$, verifying the validity of \eqref{coercive}.
\end{proof}

The following result classifies which quadratic storage functions $S_H$ are strict storage functions.

\begin{proposition}  \label{P:strictQuadStorage}
Suppose that $S = S_H$ is a quadratic storage function for the system $\Sigma$ in  \eqref{dtsystem}.  Then $S_H$ is a
strict storage function for  $\Sigma$  if and only if $H$ is a bounded positive-semidefinite solution of the strict KYP-inequality \eqref{KYP2}.
Any such solution is in fact strictly positive-definite.
\end{proposition}

\begin{proof}
Suppose that $S_H$ is a strict storage function
for $\Sigma$.  Then by definition $S_H(x) < \infty$ for all $x \in \cX$. Hence $\cD(H)=\cX$.  By the Closed Graph Theorem, it follows
that $H$ is bounded.  As a consequence of Proposition \ref{P:strictstorage-Schur}, $S_H$ is coercive and hence $H$ is strictly
positive-definite.  The strict dissipation inequality \eqref{diss-strict2} expressed in terms of $H$ and the system matrix
$\sbm{ A & B \\ C & D}$ becomes
$$
\| H^\half (Ax + Bu) \|^2 - \| H^\half x \|^2 + \delta \| x \|^2 \le
(1 - \delta) \| u \|^2 - \| C x + D u \|^2
$$
for all $x \in \cX$ and $u \in \cU$.  This can be expressed more succinctly as
\begin{align*}
&  \left\langle \begin{bmatrix} H & 0 \\ 0 & I \end{bmatrix} \begin{bmatrix} A & B \\ C & D \end{bmatrix} \begin{bmatrix} x \\ u \end{bmatrix},
\begin{bmatrix} A & B \\ C & D \end{bmatrix} \begin{bmatrix} x \\ u \end{bmatrix}  \right \rangle -
\left\langle \begin{bmatrix} H & 0 \\ 0 & I \end{bmatrix}
\begin{bmatrix} x \\ u \end{bmatrix}, \begin{bmatrix} x \\ u \end{bmatrix} \right\rangle  \\
& \quad \quad \quad \quad \le - \delta \left\langle \begin{bmatrix} x \\ u \end{bmatrix}, \begin{bmatrix} x \\ u \end{bmatrix} \right\rangle
\end{align*}
for all $x \in \cX$ and $u \in \cU$, for some $\delta > 0$.  This is just the spatial version of \eqref{KYP2}, so $H$ is a strictly positive-definite solution of
the strict KYP-inequality \eqref{KYP2}.  By reversing the steps one sees that $H \succeq 0$ being a solution of the strict KYP-inequality
\eqref{KYP2} implies that $S_H$ is a strict storage function.  As a consequence of Proposition \ref{P:strictstorage-Schur} we see that
then $S_H$ satisfies a coercivity condition \eqref{coercive}, so necessarily $H$ is strictly positive-definite.
\end{proof}

 \section{The available storage and required supply}\label{S:ASRS}

 In Proposition \ref{P:storage-Schur} we showed that the existence of
a storage function  (which is allowed to attain the value $+\infty$)
for a discrete-time linear system $\Si$ implies that the transfer
function $F_\Si$ associated with $\Si$ is equal to a Schur class
function on a neighborhood of 0. In this section we investigate the converse direction.
Specifically, we give explicit variational formulas for three storage functions, referred to as the
available storage function $S_a$ (defined in \eqref{Sa2}) the required supply function $S_r$
(defined in \eqref{Sr2}) and the  ``regularized'' version $\uS_r$ of the required supply
 (defined in \eqref{uSr2}).  Let $\bcU$ denote the space of
all functions $n \mapsto u(n)$ from the integers ${\mathbb Z}$ into
the input space $\cU$. Then $S_{a}$ is given by
\begin{equation}  \label{Sa2}
 S_{a}(x_{0}) =  \sup_{\bu \in \bcU,\,  n_{1} \ge 0}
\sum_{n=0}^{n_{1}}
   \left( \| \by(n) \|^{2} -  \|\bu(n)\|^{2}\right)
\end{equation}
with the supremum taken over all system trajectories $(\bu(n), \bx(n), \by(n))_{n \ge 0}$ with initialization
$\bx(0) = x_0$, while  $S_r$ is given by
\begin{equation}  \label{Sr2}
S_{r}(x_{0}) = \inf_{\bu \in \bcU, \, n_{-1} < 0}
\sum_{n=n_{-1}}^{-1} \left( \|\bu(n)\|^{2} - \| \by(n) \|^{2} \right)
\end{equation}
with the infimum taken over all system trajectories $(\bu(n), \bx(n), \by(n))_{n\ge n_{-1}}$
subject to the initialization condition $\bx(n_{-1}) =  0$ and the condition $\bx(0) = x_{0}$.

The proof that $S_a$ and $S_r$ are storage functions whenever $F_\Sigma$ is in the Schur class
requires the following preparatory lemma.  We shall use the following notation.
For an arbitrary Hilbert space $\cZ$, write $P_+$ and $P_-$ for the
orthogonal projections onto  $\ell^2_\cZ(\BZ_+)$ and $ \ell^2_\cZ(\BZ_-)$,
respectively, acting on $ \ell^2_\cZ(\BZ)$.
For integers $m\leq n$, we write $P_{[m,n]}$ for the orthogonal projection on
the subspace of sequences in $\ell^2_\cZ(\BZ)$ with support on the
coordinate positions $m, m+1,\ldots, n$.

\begin{lemma}  \label{L:prep}
Let $\Sigma$ be as in \eqref{dtsystem} and suppose that its transfer function $F_\Sigma$ is in the Schur class.
Then, for each system trajectory $(\bu(n), \bx(n), \by(n))_{n \ge 0}$ with initialization $\bx(0) = 0$, the inequality
\begin{equation}  \label{io-ineq}
\sum_{n=0}^N \| \by(n) \|^2 \le \sum_{n=0}^N \| \bu(n) \|^2
\end{equation}
holds for all $N \in {\mathbb Z}_+$.
\end{lemma}

\begin{proof}
As we have already observed, the fact that $F_\Sigma$ is in the Schur class
$\cS(\cU, \cY)$ implies that the multiplication operator $M_{F_\Sigma}$
\eqref{mult-op} has norm at most $1$ as an operator from $L^2_\cU({\mathbb T})$ to $L^2_\cY({\mathbb T})$.
If we apply the inverse $Z$-transform to the full operator $M_{F_\Sigma}$, not just to the compression
${\mathbb H}_{F_\Sigma}$ as was done to arrive at the Hankel operator $\fH_{F_\Sigma}$ in \eqref{Hankel-matrix},
we get the {\em Laurent operator}
\begin{equation}\label{Laurent0}
\frakL_{F_{\Si}}=\mat{ccc|ccc}{
\ddots&\ddots&\ddots &\ddots&\ddots&\ddots\\
\ddots&F_{0}&0 &0&0&\ddots\\
\ddots&F_{1}&F_{0} &0&0&\ddots\\
\hline
\ddots&F_{2}&F_{1} &F_0&0&\ddots\\
\ddots&F_{3}&F_{2} &F_1&F_0&\ddots\\
\ddots&\ddots&\ddots&\ddots&\ddots&\ddots}: \ell^2_\cU(\BZ)\to
\ell^2_\cY(\BZ),
\end{equation}
where $F_0, F_1, F_2, \dots$ are the Taylor coefficients of $F_\Sigma$:
\begin{equation}   \label{Taylor}
  F_n = \begin{cases} D &\text{if } n=0 \\
         C A^{n-1}B &\text{if $n \ge 1$.}  \end{cases}
\end{equation}
It is convenient to write $\fL_{F_\Sigma}$ as a $2 \times 2$-block matrix with respect to the decomposition
$\ell^2_\cU({\mathbb Z}) = \sbm{ \ell^2_\cU({\mathbb Z}_-) \\ \ell^2_{\cU}({\mathbb Z}_+)}$ of the domain
and the decomposition $\ell^2_\cY({\mathbb Z}) = \sbm{ \ell^2_\cY({\mathbb Z}_-) \\ \ell^2_{\cY}({\mathbb Z}_+)}$  of the range;
the result is
\begin{equation}  \label{Laurent}
\frakL_{F_{\Si}}=\mat{c|c}{\wtil{\frakT}_{F_\Si}&0\\\hline
\frakH_{F_\Si}& \frakT_{F_\Si}}:\mat{c}{\ell^2_\cU(\BZ_-)\\
\ell^2_\cU(\BZ_+)}\to \mat{c}{\ell^2_\cY(\BZ_-)\\ \ell^2_\cY(\BZ_+)}.
\end{equation}
Here $\frakH_{F_\Si}:\ell^2_-(\cU)\to
\ell^2_+(\cY)$ denotes the Hankel operator associated with ${F_\Si}$
already introduced in \eqref{Hankel-matrix},
$\frakT_{F_\Si}:\ell^2_+(\cU)\to \ell^2_+(\cY)$ the Toeplitz operator
associated with ${F_\Si}$, and $\widetilde \frakT_{{F_\Si}}$ the
Toeplitz operator acting from $\ell^{2}_{\cU}({\mathbb Z}_{-})$ to
$\ell^{2}_{\cY}({\mathbb Z}_{-})$ associated with ${F_\Si}$.   From the assumption that $F_\Sigma$ is in the
Schur class $\cS(\cU, \cY)$, it follows that $M_{F_\Sigma}$ is contractive, and hence also each of the operators
$\widetilde \fT_{F_\Sigma}$, $\fH_{F_\Sigma}$, and $\fT_{F_\Sigma}$ is contractive.  From the lower triangular form of
$\fT_{F_\Sigma}$ we see in addition that $\fT_{F_\Sigma}$ has the {\em causality property}:
\begin{equation}\label{causal}
P_{[0,N]} \fT_{F_\Sigma} =  P_{[0,N]} \fT_{F_\Sigma} P_{[0,N]}\quad (N\ge 0).
\end{equation}
Now suppose that $(\bu(n), \bx(n), \by(n))_{n \ge 0}$ is a system trajectory on ${\mathbb Z}_+$ with initialization
$\bx(0) = 0$.  In this case the infinite matrix identity $\by = \fT_{F_\Sigma} \bu$ holds formally. For $N \in {\mathbb Z}_+$ we have $P_{[0,N]}\bu \in \ell^2_\cU(\BZ_+)$, and by the causality property
\[
P_{[0,N]} \fT_{F_\Sigma} P_{[0,N]}\bu
=P_{[0,N]} \fT_{F_\Sigma}\bu
=P_{[0,N]}\by.
\]
Since $\fT_{F_\Sigma}$ is contractive, so is $P_{[0,N]} \fT_{F_\Sigma} P_{[0,N]}$ and thus the above identity shows that
$
 \| P_{[0,N]} \by \| \le \| P_{[0,N]} \bu \|,
 $
or, equivalently,
 \begin{equation}  \label{causal-contraction}
 \sum_{n=0}^N \| \by(n) \|^2 \le \sum_{n=0}^N \| \bu(n) \|^2
 \end{equation}
holds for each system trajectory $(\bu(n), \bx(n), \by(n))_{n\ge 0}$
 with $\bx(0) = 0$.
 \end{proof}

The proof of the following result is an adaptation of the proofs
of Theorems 1 and 2 for the continuous time setting in \cite{Wil72a}.

\begin{proposition}\label{P:SaSr}
Assume that the discrete-time linear system $\Si$ has a transfer function
$F_\Sigma$ which has an analytic continuation to a function in the Schur class $\cS(\cU,\cY)$.
Define $S_{a}$ and $S_{r}$ by \eqref{Sa2} and \eqref{Sr2}. Then
\begin{enumerate}
\item[(1)] $S_{a}$ is a storage function,
\item[(2)] $S_{r}$ is a storage function, and
\item[(3)] for each storage function $S$ for $\Sigma$ we have
$$
 S_a(x_0) \le  S(x_0) \le S_r(x_0) \text{ for all } x_0 \in \cX.
$$
\end{enumerate}
\end{proposition}

\begin{proof}
The proof consists of three parts,
corresponding to the three assertions of the proposition.

\smallskip

{(1)}
To see that $S_{a}(x_{0}) \ge 0$ for all $x_{0}\in\cX$,  choose $\bx(0)
= x_{0}$ and $\bu(n) = 0$ for $n \ge 0$ to generate a system
trajectory $(\bu(n), \bx(n), \by(n))_{n \ge 0}$ such that
$\sum_{n=0}^{n_{1}} ( \| \by(n) \|^{2} - \| \bu(n) \|^{2}) =
\sum_{n=0}^{n_{1}} \| \by(n) \|^{2} \ge 0$ for all $n_{1} \ge 0$.
From the
definition \eqref{Sa2}, we see that $S_{a}(x_{0}) \ge 0$.

By Lemma \ref{L:prep}, each system trajectory $(\bu(n), \bx(n), \by(n))_{n \ge 0}$
with initialization $\bx(0) = 0$ satisfies the inequality
\[
\sum_{n=0}^{n_1} \| \by(n) \|^{2}_{\cY} \le \sum_{n=0}^{n_1} \| \bu(n)
\|^{2}_{\cU}\quad (n_1\in\BZ_+).
\]
This observation leads to the conclusion that $S_{a}(0) \le 0$. Hence
$S_{a}(0)=0$ and thus $S_{a}$ satisfies the normalization
\eqref{normalization'}.

Now let $\{\wtil\bu(n), \wtil\bx(n), \wtil\by(n)\}_{n \ge N_0}$ be any
system trajectory initiated at some $N_0\in\BZ$.   We wish to show that this trajectory satisfies the
dissipation inequality \eqref{disineq}.  It is convenient to rewrite
this condition in the form
\[
\| \wtil\by(n) \|^{2}_{\cY} - \| \wtil\bu(n) \|^{2}_{\cU} +
S_a(\wtil\bx(n+1)) \le S_a(\wtil\bx(n))\quad (n\in \BZ).
\]
By translation invariance of the system equations \eqref{dtsystem},
without loss of generality we may take $n=0$, so we need to show
\begin{equation}   \label{distoshow}
\| \wtil\by(0) \|^{2}_{\cY} - \| \wtil\bu(0) \|^{2}_{\cU} +
S_a(\wtil\bx(1)) \le S_a(\wtil\bx(0)).
\end{equation}
We rewrite the definition \eqref{Sa2} for $S_a(\wtil\bx(1))$ in the
form
\[
S_a(\wtil\bx(1)) = \sup_{\bu \in \bcU, n_1\geq 0} \sum_{n=0}^{n_{1}}
\left( \| \by(n)
\|^{2}_{\cY} - \| \bu(n) \|^{2}_{\cU} \right),
\]
where the system trajectory $(\bu(n),\bx(n),\by(n))_{n \ge 0}$
is subject to the
initializa\-tion $\bx(0)=\wtil\bx(1)$. Again making use of the
translation invariance of the system equations, we may rewrite this
in the form
\[
S_a(\wtil\bx(1)) = \sup_{\bu \in \bcU, n_1\geq 1} \sum_{n=1}^{n_{1}}
\left( \| \by(n)
\|^{2}_{\cY} - \| \bu(n) \|^{2}_{\cU} \right),
\]
where $(\bu(n),\bx(n),\by(n))_{n \ge 0}$ is a system trajectory with
initialization now given by $\bx(1)=\wtil\bx(1)$. Substituting this expression for
$S(\wtil\bx(1))$, the left hand side of  \eqref{distoshow} reads
\[
\| \wtil\by(0) \|^{2}_{\cY} - \| \wtil\bu(0) \|^{2}_{\cU} +
\sup_{\bu \in \bcU, n_1\geq 1} \sum_{n=1}^{n_{1}} \left( \| \by(n)
\|^{2}_{\cY} - \| \bu(n) \|^{2}_{\cU} \right).
\]
This quantity indeed is bounded above by
\[
S_a(\wtil\bx(0))= \sup_{\bu \in \bcU, n_1\geq 0} \sum_{n=0}^{n_{1}}
\left( \| \by(n)
\|^{2}_{\cY} - \| \bu(n) \|^{2}_{\cU} \right),
\]
with $(\bu(n),\bx(n),\by(n))_{n \ge 0}$ a system trajectory subject to initialization $\wtil\bx(0)=\bx(0)$. Hence the
inequality \eqref{distoshow} follows as required, and $S_{a}$ is a
storage function for $\Sigma$.

\smallskip

{(2)}
Let $(\bu(n), \bx(n), \by(n))_{n \ge n_-1}$ be a system trajectory with zero-initial\-ization of the state at $n_{-1} < 0$, subject also to $\bx(0)=x_0$. Applying the result of Lemma \ref{L:prep} to this system trajectory, using the translation invariance property of $\Si$ to get a sum in \eqref{io-ineq} starting at $n_{-1}$ and ending at 0, it follows that $S_{r}(x_{0}) \ge 0$ for all $x_{0}$ in $\Rea (A|B)$.
In case $x_0\not\in\Rea(A|B)$, i.e.,  $x_{0}$ is not reachable in finitely many steps via some input signal
$\bu(n)$ ($n_{-1}\le n < 0$) with $\bx(n_{-1}) = 0$, then the definition of $S_r$ in \eqref{Sr2} gives us $S_{r}(x) = +\infty \ge 0$.
By choosing $n_{-1} = -1$ with $\bu(-1) =
0$, we see that $S_{r}(0) \le 0$. Since $S_{r}(x_{0}) \ge 0$
for each $x_{0} \in \cX$, it follows that $S_{r}$  also satisfies the
normalization \eqref{normalization'}.

An argument similar to that used in part 1 of the proof shows
that $S_{r}$ satisfies \eqref{disineq}. Indeed, note that it suffices
to show
that for each system trajectory $\{\wtil\bu(n), \wtil\bx(n),
\wtil\by(n)\}_{n \ge 0}$ we have
\begin{align}  S_{r}(\wtil\bx(1)) & \le \| \wtil\bu(0) \|^{2}_{\cU} - \|
\wtil\by(0)\|^{2}_{\cY}+ S_{r}(\wtil\bx(0))  \label{toshow2} \\
& = \inf_{\bu \in \bcU, \, n_{-1} < 0} \left\{ \| \wtil\bu(0) \|^{2}_{\cU} - \|
\wtil\by(0)\|^{2}_{\cY}+
 \sum_{n=n_{-1}}^{-1}  \left( \| \bu(n) \|^{2}_{\cU} - \| \by(n)
\|^{2}_{\cY} \right)\right\}
 \notag
\end{align}
where $(\bu(n), \bx(n), \by(n))_{n\ge n_{-1}}$ is a system trajectory subject to the initial condition
$\bx(n_{-1}) = 0$ and the terminal condition $\bx(0) = \widetilde x(0)$.
Rewrite the definition of $S_{r}(\wtil\bx(1))$ as
\[
S_{r}(\wtil\bx(1)) = \inf_{\bu \in \bcU, \, n_{-1}< 1}
\sum_{n=n_{-1}}^{0} \left( \| \bu(n) \|^{2}_{\cU} - \| \by(n)
\|^{2}_{\cY} \right),
\]
with the system trajectory $(\bu(n),\bx(n),\by(n))_{n\in\BZ}$  subject
to the initial and terminal
conditions $\bx(n_{-1}) =  0$ and $\bx(1) = \wtil\bx(1)$.  Now recognize the
argument of the $\inf$ in the right-hand side of \eqref{toshow2} as  part of the competition in the
infimum defining $S_r(\wtil\bx(1))$  to deduce the inequality \eqref{toshow2}.

\smallskip

{(3)}
Let $S$ be any storage function for $\Sigma$ and $(\bu(n), \bx(n), \by(n))_{n \ge 0}$ any system
trajectory with initialization $\bx(0) = x_0$.  Iteration of the dissipation inequality \eqref{disineq} for $S$ along the system trajectory $(\bu(n), \bx(n), \by(n))_{n \ge 0}$  as in the proof of Lemma \ref{L:finH2}
yields
$$
0 \le S(n_1 + 1) \le S(x_0) + \sum_{n=0}^{n_1} \left( \| \bu(n) \|^2 - \| \by(n) \|^2 \right)
$$
or
$$
\sum_{n=0}^{n_1} \left( \| \by(n) \|^2 - \| \bu(n) \|^2 \right) \le S(x_0).
$$
Taking the supremum in the left-hand side of the above inequality over all such system trajectories $(\bu(n), \bx(n), \by(n))_{n \ge 0}$
and all $n_1 \ge 0$ yields $S_a(x_0) \le S(x_0)$ and the first part of (3) is verified.

Next let $x_{0}\in\cX$ be arbitrary.
If $(\bu(n), \bx(n), \by(n))_{n\ge n_{-1}}$ is
any system trajectory with state-initializa\-tion $x(n_{-1}) = 0$
and $x(0) = x_{0}$, applying Lemma \ref{L:finH2} with $N_0=n_{-1}$ and
$N=-1-n_{-1}$ gives us that
\begin{equation}   \label{Stilde-ineq}
S(x_{0})  \le \sum_{n=n_{-1}}^{-1} \left( \|
\bu(n)\|^{2}_{\cU} - \| \by(n) \|^{2}_{\cY} \right).
\end{equation}

Taking the infimum of the right-hand side over all such system
trajectories gives us $S(x_{0}) \le S_{r}(x_{0})$. Here we
implicitly assumed that the state $x_{0} \in \cX$ is reachable.  If
$x_{0}$ is not reachable, there are no such system trajectories, and
taking the infimum over an empty set leads to $S_{r}(x_{0}) =
\infty$, in which case $S(x_0)\le S_{r}(x_{0})$ is also valid. Hence $S(x_{0}) \le S_{r}(x_{0})$ holds for all possible
$x_{0}\in\cX$.  This completes the verification of the second part of (3).
\end{proof}

Combining Proposition \ref{P:SaSr} with Proposition \ref{P:storage-Schur} leads to the following.

\begin{corollary}  \label{C:storageSchur}  A discrete-time linear system $\Sigma$ in \eqref{dtsystem} has a transfer function $F_\Sigma$ with an analytic continuation in the Schur class if and only if $\Sigma$ has a storage function $S$.
\end{corollary}

\begin{proof}
The sufficiency is Proposition \ref{P:storage-Schur}.  For the necessity direction, by Proposition \ref{P:SaSr}
we may choose $S$ equal to either $S_a$ or $S_r$.
\end{proof}

We next impose a minimality assumption on $\Sigma$ and in addition
assume that $F_{\Sigma}$ has an analytic continuation in the Schur
class $\cS(\cU, \cY)$, i.e., we make the following assumptions:
\begin{equation} \label{A}
 \hspace*{-.15cm}    \left\{ \!\!\! \begin{array}{l}
 \mbox{\em $\Si$ is minimal, i.e., $(C,A)$ is observable and $(A,B)$ is
controllable,}  \\
 \mbox{\em  and $F_{\Sigma}$ has an analytic continuation to a function in
$\cS(\cU, \cY)$.}
\end{array} \right.
\end{equation}

Our next goal is to understand storage functions from a more operator-theoretic
point of view.  We first need some preliminaries.

Recall the Laurent operator $\frakL_{F_\Si}$ in \eqref{Laurent0}.
From
the $2 \times 2$-block form for $\frakL_{{F_\Si}}$ in
\eqref{Laurent}, we see that
\begin{equation}\label{LFids}
\begin{aligned}
I - \frakL_{{F_\Si}} \frakL_{{F_\Si}}^{*} &= \begin{bmatrix}
D_{\widetilde
\frakT_{{F_\Si}}^{*}}^{2} & - \widetilde \frakT_{{F_\Si}}
\frakH_{{F_\Si}}^{*} \\
 - \frakH_{{F_\Si}} \widetilde \frakT_{{F_\Si}}^{*} &
 D_{\frakT_{{F_\Si}}^{*}}^{2} - \frakH_{{F_\Si}} \frakH_{{F_\Si}}^{*}
\end{bmatrix};\\
I - \frakL_{{F_\Si}}^{*} \frakL_{{F_\Si}}
&= \begin{bmatrix}
D_{\widetilde
  \frakT_{{F_\Si}}}^{2} - \frakH_{{F_\Si}}^{*} \frakH_{{F_\Si}} &
-\frakH_{F_\Si}^*\frakT_{{F_\Si}}\\
  -\frakT_{{F_\Si}}^* \frakH_{F_\Si} &  D_{\frakT_{{F_\Si}}}^{2}
\end{bmatrix}.
\end{aligned}
\end{equation}
where in general we use the notation $D_{X}$ for the defect operator
$D_{X} = (I - X^{*} X)^{\half}$ of a contraction operator $X$. Since
${F_\Si}$ is assumed to be a Schur class function, $\frakT_{F_\Si}$
and $\widetilde \frakT_{{F_\Si}}$ are contractions, and hence
$D_{\frakT_{{F_\Si}}}$, $D_{ \frakT_{{F_\Si}}^{*}}$, $D_{\widetilde
\frakT_{{F_\Si}}}$ and $D_{\widetilde
\frakT_{{F_\Si}}^{*}}$ are well defined.

\begin{lemma}\label{L:SaSrOpForm}
Let the discrete-time linear system $\Si$ in \eqref{dtsystem} satisfy
the assumptions \eqref{A}. The available storage function $S_a$ and
required supply function $S_r$ can then be written in operator form as
\begin{align}
\label{SaOpForm}
S_{a}(x_{0}) &= \!\! \sup_{ \bu \in \ell^{2}_{\cU}({\mathbb Z}_{+})}
\| \bW_{o} x_{0} + \frakT_{{F_\Si}} \bu
\|^{2}_{\ell^{2}_{\cY}({\mathbb
 Z}_{+})} - \| \bu \|^{2}_{\ell^{2}_{\cU}({\mathbb Z}_{+})}\ (x_0\in \cD(\bW_o)) \\
S_r(x_0)&=\inf_{\bu\in\ell_{\tu{fin},\cU}(\BZ_-),\, x_0=\bW_c \bu}\|D_{\wtil\fT_{F_\Si}}\bu\|^2\quad   (x_0\in\cX),
\label{SrOpForm}
\end{align}
and $S_{a}(x_{0})=+\infty$ for $x_0\not\in\cD(\bW_o)$. Here $\bW_o$ and $\bW_c$ are the observability and controllability operators defined via \eqref{bWo1}--\eqref{bWc*2} and $\ell_{\tu{fin},\cU}(\BZ_-)$ is the linear manifold of finitely supported sequences in $\ell^2_\cU(\BZ_-)$. In particular, $S_r(x_0)<\infty$ if and only if $x_0\in\Rea(A|B)$.
\end{lemma}

\begin{proof}
We shall use the notation $P_\pm$ and $P_{[m,n]}$ as introduced in the
discussion immediately preceding the statement of Lemma \ref{L:prep}.

We start with $S_{a}$. For each system trajectory $(\bu(n),\bx(n),\by(n))_{n \ge 0}$ with initialization $\bx(0) = x_0$ and with $\bu\in \ell^2_\cU(\BZ_+)$ by linearity we have
\[
\by= \bW_o x_0 + \fT_{F_\Sigma} \bu.
\]
Now note that, for each system trajectory  $(\bu(n),\bx(n),\by(n))_{n \ge 0}$ with initialization $\bx(0) = x_0$ but with $\bu$
not necessarily in $\ell^2_\cU(\BZ_+)$ and with $n_1\geq 0$, by the causality property \eqref{causal}, as in the proof of
Lemma \ref{L:prep} we see that  we can replace $\bu$ with $P_{[0,n_1]}\bu\in \ell_{{\rm fin}, \cU}({\mathbb Z}_{+})
\subset \ell^2_\cU (\BZ_+)$ within the supremum in \eqref{Sa2} without changing the value.
Therefore, the value of $S_a$ at $x_0$ can be rewritten in operator form as
\begin{align}
&S_{a}(x_{0})=\notag \\
&=\!\!\! \sup_{\bu \in \ell_{{\rm fin}, \cU}({\mathbb Z}_{+}), \,
 n_{1} \ge 0} \| P_{[0,n_{1}]} (\bW_{o}x_{0} + \frakT_{{F_\Si}} \bu
)\|^{2}_{\ell^{2}_{\cY}({\mathbb Z}_{+})}
 - \| P_{[0, n_{1}]} \bu
 \|^{2}_{\ell^{2}_{\cU}({\mathbb Z}_{+})} \label{sup'-pre}
\end{align}
where we use the notation $\ell_{{\rm fin}, \cU}({\mathbb Z}_+)$ for $\cU$-valued sequences
on ${\mathbb Z}_+$ of finite support.

If $x_0 \notin \cD(\bW_o)$ so that $\bW_o x_0 \notin \ell^2_\cY({\mathbb Z}_+)$, the above formulas are to be interpreted algebraically, and we may choose $\bu = 0$ and take the limit as $n_1 \to \infty$ to see that $\bW_o(x_0) = +\infty$.

Now assume $x_0\in \cD(\bW_o)$. Fix $\bu \in \ell_{{\rm fin}, \cU}({\mathbb Z}_+)$ and take the limit as $n_1 \to +\infty$ in the right hand side
 of  \eqref{sup'-pre} to see that an equivalent expression for $S_a(x_0)$ is
 $$
 S_a(x_0) = \sup_{\bu \in \ell_{{\rm fin}, \cU}({\mathbb Z}_+)} \| \bW_o x_0 + \fT_{F_\Sigma} \bu \|^2 - \| \bu \|^2.
 $$
Since $\ell_{{\rm fin}, \cU}({\mathbb Z}_+)$ is dense in $\ell^2_\cU({\mathbb Z}_+)$ and $\fT_{F_\Sigma}$ is a bounded
operator, we see that another equivalent expression for $S_a(x_0)$ is \eqref{SaOpForm}.
This completes the verification of \eqref{SaOpForm}.

We next look at $S_r$.  Let $(\bu(n), \bx(n), \by(n))_{n \ge n_{-1}}$ be any system trajectory with initialization $x(n_{-1}) = 0$ for some $n_{-1}<0$.
Let us identify $\bu$ with an element $\bu \in \ell_{{\rm fin}, \cU}({\mathbb Z}_-)$ by ignoring the values of $\bu$ on
${\mathbb Z}_+$ and defining $\bu(n) = 0$ for $n < n_{-1}$.  Then, as a consequence of item (1) in
Proposition \ref{P:Wc}, the constraint $\bx(0) = x_0$ in \eqref{Sr2} can be written in operator form as $\bW_c \bu = x_0$.
Furthermore, since $(\bu(n), \bx(n), \by(n))_{n \ge n_{-1}}$ is a system trajectory with zero state initialization at $n_{-1}$,
it follows that
$$
   \by|_{{\mathbb Z}_-} = \widetilde \fT_{F_\Sigma} \bu.
$$
We conclude that a formula for $S_r$ equivalent to \eqref{Sr2} is
$$
S_r(x_0) = \inf_{ \bu \in \ell^2_{{\rm fin}, \cU}({\mathbb Z}_-) \colon \bW_c \bu = x_0}
\| \bu \|^2 - \| \widetilde \fT_{F_\Sigma} \bu \|^2
$$
which in turn has the more succinct formulation \eqref{SrOpForm}. If $x_0 \in \Rea(A|B)$, then the infimum in \eqref{SrOpForm} is taken over a nonempty set, so that $S_r(x_0)<\infty$. On the other hand, if $x_0 \not\in \Rea(A|B)$, then the infimum is taken over an empty set, so that $S_r(x_0)=\infty$.
\end{proof}

To compute storage functions more explicitly for the case where assumptions \eqref{A} are
in place, it will be convenient to restrict to what we shall call {\em $\ell^2$-regular storage functions} $S$,
namely, storage functions $S$ which assume finite values on $\im \bW_c$:
\begin{equation}  \label{reg-storage}
 x_0 = \bW_c \bu \text{ where } \bu \in \cD(\bW_c) \Rightarrow S(x_0) < \infty.
 \end{equation}
We shall see in the next result that $S_a$ is $\ell^2$-regular.  However, unless if $\Rea(A|B)$ is equal to the range of
$\bW_c$, the required supply $S_r$ will  not be $\ell^2$-regular (by the last assertion of Lemma \ref{L:SaSrOpForm}).

To remedy this situation, we introduce the following modification $\uS_r$ of the required supply $S_r$, which we shall call the {\em $\ell^2$-regularized required supply}:
\begin{equation}  \label{uSr2}
 \uS_r(x_0) = \inf_{\bu \in \cD(\bW_c) \colon \bW_c \bu = x_0} \sum_{n=-\infty}^{-1}  \left( \| \bu(n) \|^2 - \| \by(n) \|^2 \right)
 \end{equation}
 where $\bu \in \ell^2_\cU({\mathbb Z}_-)$ determines $\by \in \ell^2_\cY({\mathbb Z}_-)$ via the system input/output map:\  $\by = \widetilde \fT_{F_\Sigma} \bu$.  Thus  formula \eqref{uSr2} can be written more succinctly in operator form as
 \begin{align}
  \uS_r(x_0) & = \inf_{\bu \in \cD(\bW_c) \colon \bW_c \bu = x_0}
   \| \bu \|^2_{\ell^2_\cU({\mathbb Z}_-)} - \| \widetilde \fT_{F_\Sigma} \bu \|^2_{\ell^2_\cY({\mathbb Z}_-)}
    \notag  \\ & =     \inf_{\bu \in \cD(\bW_{c}), \, \bW_{c} \bu = x_{0}}
\|  D_{\widetilde \frakT_{{F_\Si}}} \bu
\|^{2}_{\ell^{2}_{\cY}({\mathbb Z}_{-})} \text{ for } x_0\in \im \bW_c.
 \label{uSrOpForm}
  \end{align}
It is clear that $\uS_r(x_0)<\infty$ if and only if $x_0\in \im \bW_c$.
 Since the objective in the infimum defining $\uS_r$ in \eqref{uSrOpForm} is the same as the
 objective in the infimum defining $S_r$ in \eqref{SrOpForm} but the former  infimum is taken over an a priori larger set,
 it follows directly that $S_r(x_0)\geq \uS_r(x_0)$ for all $x_0\in\cX$, as can also be seen as a consequence of
 Proposition \ref{P:SaSr} once we show that $\uS_r$ is a storage function for $\Sigma$.
 From either of the formulas we see  that $0\leq\uS_r(x_0)$ and that $\uS_r(x_0) <\infty$ exactly when $x_0$ is
 in the range of $\bW_c$.  Hence once we show that $\uS_r$ is a storage function, it follows that $\uS_r$ is an
 $\ell^2$-regular storage function and is a candidate to be the largest such.  However at this stage we have
 only partial results in this direction, as laid out in the next result.

 \begin{proposition}  \label{P:uSr}  Assume that $\Sigma$ is a system satisfying the assumptions \eqref{A}
 and let the function $\uS_r \colon \im \bW_c \to {\mathbb R}_+$ be given by \eqref{uSrOpForm}.  Then:
 \begin{enumerate}
 \item[(1)]  $S_a$ and  $\uS_r$ are  $\ell^2$-regular storage functions.

 \item[(2)] $\uS_r$ is ``almost'' the largest $\ell^2$-regular storage function in the following sense:
 if $S$ is another $\ell^2$-regular storage function such that either
 \begin{enumerate}
 \item $S$ is $\cD(\bW_c^*)$-weakly continuous in the sense that:  given a sequence $\{ x_n\} \subset \im \bW_c$ and $x_c \in \im \bW_c$ such that
 $$
 \lim_{n \to \infty} \langle x, x_n \rangle_\cX  = \langle x, x_c \rangle_\cX \text{ for all } x \in \cD(\bW_c^*),
 $$
 then $\lim_{n \to \infty} S(x_n) = S(x_0)$, or

 \item $\bW_c$ is bounded
 and $S$ is continuous on $\cX$ {\rm(}with respect to the norm topology on $\cX${\rm)},
 \end{enumerate}
 then $S(x_0) \le \uS_r(x_0)$ for all $x_0\in\cX$.
  \end{enumerate}
  \end{proposition}

\begin{proof}
We first prove item (1), starting with the claim for $S_a$. Since by assumption $\Si$ is minimal and $F_\Si$ has an analytic continuation to a Schur class function, by item (1) of Proposition \ref{P:HankelDecs}, $\im \bW_c\subset \cD(\bW_o)$. So on $\im \bW_c$, the available storage $S_a$ is given by \eqref{SaOpForm}. It remains to show that for $x_0\in \im \bW_c$ the formula for $S_a(x_0)$ in \eqref{SaOpForm} gives a finite value. So assume $x_0\in \im \bW_c$, say $x_0=\bW_c \bu_-$ for a $\bu_-\in\ell^2_\cU(\BZ_-)$. Choose $\bu_+\in\ell^2_\cU(\BZ_+)$ arbitrarily and define $\bu\in\ell^2_\cU(\BZ)$ by setting $P_- \bu=\bu_-$ and $P_+ \bu=\bu_+$. Then $\bW_o x_0=\bW_o \bW_c \bu_-= \fH_{F_\Si}\bu_-$. Thus, using the decomposition of $\fL_{F_\Si}$ in \eqref{Laurent} and the fact that $\|\fL_{F_\Si}\|\leq 1$, we find that
\begin{align*}
&\| \bW_{o} x_{0} + \frakT_{{F_\Si}} \bu_+\|^{2} - \| \bu_+ \|^{2}
 =\| \fH_{F_\Si}\bu_- + \frakT_{{F_\Si}} \bu_+\|^{2} - \| \bu_+ \|^{2}\\
&\qquad\qquad=\| P_+\fL_{F_\Si} \bu\|^{2} - \|P_+ \bu \|^{2}
= \|P_- \bu \|^{2}+ \| P_+\fL_{F_\Si} \bu\|^{2} - \| \bu \|^{2}\\
&\qquad\qquad\leq \|P_- \bu \|^{2}=\|\bu_-\|^2.
\end{align*}
Since the upper bound $\|\bu_-\|^2$ is independent of the choice of $\bu_+\in \ell^2_\cU(\BZ_+)$, we can take the supremum over all $\bu_+\in \ell^2_\cU(\BZ_+)$ to arrive at the inequality $S_a(x_0)\leq \|\bu_-\|^2<\infty$.

Next we prove the statement of item (1) concerning $\uS_r$.
By the discussion immediately preceding the statement of the proposition, it follows that $\uS_r$ is an $\ell^2$-regular storage function once we show that $\uS_r$ is a storage function, that is, $\uS_r(0) = 0$ and that
$\uS_r$ satisfies the dissipation inequality \eqref{disineq}.

If $x_0 = 0$, we can choose $\bu = 0$ as the argument in the right hand side of \eqref{uSrOpForm} to conclude that
$\uS_r(0) \le 0$.  As we have already seen that $\uS_r(x_0) \ge 0$ for all $x_0$, we conclude that
$\uS_r(0) = 0$.

To complete the proof of item (1), it remains to show that $\uS_r$ satisfies the dissipation inequality \eqref{disineq}.  By shift invariance we may take $n=N_0 = 0$ in \eqref{disineq}. If $\bx(0) \notin \im \bW_c$, then $\uS_r(x_0) = \infty$ and \eqref{disineq} holds trivially.  We therefore assume that $(\widetilde \bu(n),  \widetilde \bx(n),
\widetilde \by(n))_{n \ge 0}$ is a system trajectory with initialization $\widetilde \bx(0) = x_0 = \bW_c \bu_-$ for some
$\bu_- \in \cD(\bW_o)$  and the problem is to show
\begin{align}
&\uS_r(\widetilde  \bx(1))  \le \|\widetilde  \bu(0) \|^2  - \| \widetilde \by(0) \|^2 +
\uS_r(\widetilde \bx(0))= \label{tocheck} \\
 &  \quad = \inf_{\bu \in \cD(\bW_c) \colon \bW_c \bu = \widetilde \bx(0)} \left[
 \|\widetilde  \bu(0) \|^2  - \| \widetilde \by(0) \|^2 + \sum_{n=-\infty}^{-1} \left( \| \bu(n) \|^2 - \| \by(n) \|^2
 \right) \right],   \notag
\end{align}
where  $\by = \widetilde \fT_{F_\Sigma} \bu$.
As $(\widetilde \bu(n), \widetilde \bx(n), \widetilde \by(n))_{n \ge 0}$ is a system trajectory initiated at 0, we know that
$\widetilde \bx(1) = A \widetilde \bx(0) + B \widetilde \bu(0)$ and
$\widetilde \by(0) = C \widetilde \bx(0) + D \widetilde \bu(0)$.
On the other hand, by translation-invariance of the system equations \eqref{dtsystem} we may rewrite the formula
\eqref{uSr2} for $\uS_r(\widetilde \bx(1))$  as
\begin{equation}  \label{uSr-shift}
\uS_r(\widetilde \bx(1)) = \inf_{ \bu'  \in \cD(\bW_c^{(1)}) \colon \bW_c^{(1)} \bu = \widetilde \bx(1)}
    \sum_{n=-\infty}^0 \left(  \| \bu'(n) \|^2 - \| \by'(n) \|^2 \right),
\end{equation}
where $\bW_c^{(1)}$ is the shifted observability operator discussed in Remark \ref{R:Wc} and
where $\by' =\widetilde \fT_{F_\Sigma}^{(1)} \bu'$;  here now $\bu'$ is supported on $(-\infty, 0]$ rather than on
 ${\mathbb Z}_- = (-\infty, 0)$ and $\widetilde \fT_{F_\Sigma}^{(1)}$ is the shift of $\widetilde \fT_{F_\Sigma}$
 from the interval ${\mathbb Z}_-$ to the interval $(-\infty, 0]$.
 Let us write sequences $\bu' \in \ell^2_\cU((-\infty, 0])$ in the form $\bu' = (\bv', v')$ as in Remark \ref{R:Wc}
 where $\bv' \in \ell^2_\cU({\mathbb Z}_-)$ and $v' \in \cU$.  As observed in Remark \ref{R:Wc},
 $$
 \bW_c^{(1)}(\bv', v) = A \bW_c \bv' + B v'.
 $$
 Furthermore, from the structure of the Laurent operator $\fL_{F_\Sigma}$ \eqref{Laurent} we read off that
 \begin{equation}  \label{shift-Toeplitz}
   \widetilde \fT_{F_\Sigma}^{(1)} (\bv', v') =
   \left( \widetilde \fT_{F_\Sigma} \bv', \sum_{k=-\infty}^{-1} C A^{-k-1} B \bv'(k) + D v'  \right)
\end{equation}
where the series converges at least in the weak topology of $\cY$.  For $\bv' \in \cD(\bW_c)$, we know from
Proposition \ref{P:WcWo'} that $\bW_c \bv'$ is given by
\begin{equation}   \label{Wc-formula}
\bW_c \bv' = \sum_{k=-\infty}^{-1} A^{-k-1} B \bv'(k)
\end{equation}
where the series converges $\cD(\bW_c^*)$-weakly.  We also know under our standing assumption \eqref{A} that
$\Obs(C|A) \subset \cD(\bW_c^*)$ (see Proposition \ref{P:HankelDecs} (2)), and hence in particular
 $C^* y \in \cD(\bW_c^*)$ for all $y \in \cY$.   This observation combined with the formula \eqref{Wc-formula}
 implies that
 $$
 C \bW_c \bv' = \sum_{k=-\infty}^{-1} C A^{-k-1} B \bv'(k)
 $$
  where the series converges weakly in $\cY$.
 This combined with \eqref{shift-Toeplitz} gives us
 $$
 \widetilde \fT_{F_\Sigma}^{(1)} (\bv', u) =
 \left( \widetilde \fT_{F_\Sigma} \bv', C \bW_c \bv' + Dv'  \right).
 $$
 Thus the formula \eqref{uSr-shift} for $\uS_r(\widetilde \bx(1))$ can be written out in more detail as
 \begin{equation}   \label{uSrtildex(1)}
\hspace*{-.2cm}   \uS_r(\widetilde \bx(1)) =\!\!\! \inf_{ (\bv', v') \in \cT' } \! \left\{\!
 ( \| \bv' \|^2 - \| \widetilde \fT_{F_\Sigma} \bv'  \|^2) + \| u \|^2 - \| C \bW_c \bv' + Du\|^2\!
 \right\}
 \end{equation}
 where
 \begin{equation}   \label{cT'}
\cT' : =  \{ \bv' \in \cD(\bW_c),  \, v' \in \cU  \colon A \bW_c \bv' + Bv' = \widetilde \bx(1)\}.
 \end{equation}
 Note that the infimum \eqref{tocheck} can be identified with the infimum \eqref{uSrtildex(1)}
 if we restrict the free parameter $(\bv', v')$ to lie in the subset
  $$
  \cT = \{ (\bv', v') \in \cT' \colon
   \bW_c \bv' = \widetilde \bx(0), \quad v' = \widetilde \bu(0)\}.
 $$
 As the infimum of an objective function over a given set $\cT$
 is always bounded above by the infimum of the same objective function over a smaller set $\cT' \subset \cT$, the
 inequality \eqref{tocheck} now  follows as wanted.

 It remains to address item (2), i.e., to show that $S(x_0) \le \uS_r(x_0)$ for any other storage function
 $S$ satisfying appropriate hypotheses.
If $x_0 \notin \im \bW_c$, $\uS_r(x_0) = \infty$ and the desired inequality holds trivially, so we assume that
$x_0 = \bW_c \bu$ for some $\bu \in \cD(\bW_c)$.  Let us approximate $\bu$ by elements of
$\ell_{\tu{fin}, \, \cU}({\mathbb Z}_-)$ in the natural way:
$$
   \bu_K(n) = \begin{cases}  \bu(n) &\text{for } -K \le n \le -1, \\
          0 &\text{for } n< -K \end{cases}
$$
for $K=1,2,\dots$, and set $x_K = \bW_c \bu_K$.  We let $(\bu(n), \bx(n), \by(n))_{n \ge -K}$ be a system trajectory
with $\bu(n) = \bu_K(n)$ and with the state initialization $\bx(-K) = 0$.  Then, as $\bx(0)$ will then be equal to
$x_K$,  iteration of the dissipation inequality \eqref{disineq} gives us
\begin{equation}   \label{tS-ineq}
 S(x_K) \le \sum_{n=-K}^{-1} \left( \| \bu_K(n) \|^2 - \| \widetilde \fT_{F_\Sigma} \bu_K(n) \|^2
 \right).
\end{equation}
We seek to let $K \to \infty$ in this inequality. As $\bu_K \to \bu$ in the norm topology of $\ell^2_\cU({\mathbb Z}_-)$
and $\| \widetilde \fT_{F_\Sigma} \| \le 1$ since $F$ is in the Schur class by assumption, it is clear that the right hand side
of \eqref{tS-ineq} converges to
$$
\| \bu \|^2_{\ell^2_\cU({\mathbb Z}_-) } - \| \widetilde \fT_{F_\Sigma} \bu \|^2_{\ell^2_\cU({\mathbb Z}_-)}=
 \| D_{\widetilde \fT_{F_\Sigma}} \bu \|^2_{\ell^2_\cU({\mathbb Z}_-)}
$$
as $K \to \infty$.
On the other hand, as a consequence of the characterization \eqref{limit-c} of the action of
$\bW_c$,  it follows that $x_K = \bW_c \bu_K$ converges to $x_0 = \bW_c \bu$ in the $\cD(\bW_c^*)$-weak sense.
Hence, if $S$ is continuous with respect to the $\cD(\bW_c^*)$-weak topology as described in the statement of item (a),  we see that $S(x_K) \to S(x_0)$ as $K \to \infty$ and we arrive at the limiting version of inequality \eqref{tS-ineq}:
\begin{equation}  \label{limit-wS-ineq}
 S(x_0) \le  \| \bu\|^2 - \| \widetilde \fT_{F_\Sigma} \bu \|^2  =
  \| D_{\widetilde \fT_{F_\Sigma}}  \bu \|^2_{\ell^2_\cU({\mathbb Z}_-)}.
\end{equation}
We may now take the infimum over all $\bu \in \cD(\bW_c)$ with $\bW_c \bu = x_0$ to arrive at the desired inequality $S(x_0) \le \uS_r(x_0)$. This proves item (a) of (2).
If $\bW_c$ is bounded, then $x_K = \bW_c \bu_K$ converges in norm to
$\bW_c \bu = x_0$.  If $S$ is continuous with respect to the norm topology on $\cX$, then
$S(x_K) \to S(x_0)$
and we again arrive at the limit inequality \eqref{limit-wS-ineq},
from which the desired inequality $S(x_0) \le \uS_r(x_0)$ again follows.
This completes the verification of item (2)
in Proposition \ref{P:uSr}.
\end{proof}

\begin{remark}
Note that the fact that $S_a$ is $\ell^2$-regular can alternatively be seen from the fact that $\uS_r$ is
a $\ell^2$-regular storage function combined with the first inequality in item (3) of Proposition \ref{P:SaSr}.
\end{remark}

Collecting some of the observations on the boundedness of $S_a$ and $\uS_r$ from the above results we obtain the following corollary. The inequalities in \eqref{ineqs} follow directly from \eqref{SaOpForm} and \eqref{uSrOpForm}.

\begin{corollary}\label{C:boundedSauSr}
Assume $\Si$ as in \eqref{dtsystem} is a system satisfying the assumptions \eqref{A}. Define $S_a$ by \eqref{Sa2} and $\uS_r$ by \eqref{uSr2}. For $x_0\in\im \bW_o$ we have
\begin{equation}\label{ineqs}
\|\bW_ox_0\|^2\leq S_a(x_0)\leq \uS_r(x_0)\leq \|\bu_-\|^2
\end{equation}
for all $\bu_-\in\cD(\bW_c)$ with $x_0=\bW_c \bu_-$, with the last inequality being vacuous if $x_0\not\in\im \bW_c$, in which case $\uS_r(x_0)=\infty$. Hence
\begin{align*}
\uS_r(x_0)<\infty \quad &\Longleftrightarrow \quad x_0\in \im \bW_c,\\
x_0\in \im \bW_c \quad \Longrightarrow\quad &S_a(x_0)<\infty \quad \Longrightarrow\quad x_0\in\cD(\bW_o).
\end{align*}
In particular, $\uS_r$ is finite-valued if and only if $\im \bW_c=\cX$, that is, $\Si$ is $\ell^2$-exactly controllable, and $S_a$ is finite-valued in case $\Si$ is $\ell^2$-exactly controllable.
\end{corollary}

 Since ${F_\Si}$ is assumed to be a Schur class function,
$\frakL_{{F_\Si}}$ is a contraction, so that $I - \frakL_{{F_\Si}}
\frakL_{{F_\Si}}^{*}$ and $I - \frakL_{{F_\Si}}^{*} \frakL_{{F_\Si}}$
are positive-semidefinite operators. We can thus read off from the
$(2,2)$-entry in the right-hand side of the first identity and the $(1,1)$-entry in
the right hand side of the second  identity of \eqref{LFids} that
\begin{equation}   \label{Hankel-est}
D_{\frakT_{{F_\Si}}^{*}}^{2} \succeq \frakH_{{F_\Si}}
\frakH_{{F_\Si}}^{*}
\ands
D_{\widetilde \frakT_{{F_\Si}}}^{2}\succeq \frakH_{{F_\Si}}^{*}
\frakH_{{F_\Si}}.
 \end{equation}
The observability and controllability assumptions of \eqref{A}
imply that the observability operator $\bW_o:\cD(\bW_o)\to
\ell^2_\cY(\BZ_+)$ and the controllability operator
$\bW_c:\cD(\bW_c)\to\cX$ are closed densely defined operators that
satisfy the properties listed in Propositions \ref{P:WcWo'} and \ref{P:HankelDecs}.
As spelled out in Proposition \ref{P:HankelDecs}, the Hankel operator $\frakH_{{F_\Si}}$ admits the
factorizations
\begin{equation}\label{HankelFact}
\frakH_{F_\Si}|_{\cD(\bW_{c})} = \bW_{o} \bW_{c}\ands
\frakH_{F_\Si}^*|_{\cD(\bW_{o}^*)} = \bW_{c}^* \bW_{o}^*.
 \end{equation}

 Using the Douglas factorization lemma \cite{Douglas} together with the factorizations
 \eqref{HankelFact}, we arrive at the following result.  The proof also requires use of
 the Moore-Penrose generalized inverse $X^\dagger$ of
a densely defined closed linear Hilbert-space operator $X:\cD(X)\to\cH_2$, with $\cD(X)\subset\cH_1$:
we define $X^\dagger  \colon \cD(X^\dagger) = (\im X \oplus (\im X)^\perp)  \to \cH_1$  by
\begin{equation}   \label{MP}
\left\{ \begin{array}{rcl}
X^\dagger (X h_1) & = & P_{ (\kr X)^\perp} h_1, \\
X^\dagger|_{(\im X)^\perp} & = & 0.
\end{array}  \right.
\end{equation}
Then $X^\dagger$ is also closed and has the properties
$$
  X^\dagger X = P_{(\kr X)^\perp}|_{\cD(X)}, \quad
  X X^\dagger = P_{\overline{\im} X}|_{ \im X  \oplus (\im X)^\perp }.
$$
In particular, if $X$ is bounded and surjective, then $X^\dagger$ is a bounded
right inverse of $X$, and, if $X$ is bounded, bounded below and injective, then $X^\dagger$ is
a bounded left inverse of $X$.

 \begin{lemma}\label{L:fact}
Let the discrete-time linear system $\Si$ in \eqref{dtsystem} satisfy
the assumptions in \eqref{A}. Then:
\begin{enumerate}
    \item[(1)]
    There exists a unique closable operator $X_a$
with domain $\im \bW_c$ mapping into
$ (\kr D_{\frakT_{F_{\Sigma}}^{*}})^\perp  \subset
\ell^2_\cY(\BZ_+)$
so that we have the factorization
\begin{equation}   \label{fact1}
\bW_{o}|_{\im \bW_{c}} = D_{\frakT_{{F_\Si}}^{*}} X_{a}.
\end{equation}
Moreover, if we let $\overline{X}_{a}$ denote the closure of
$X_{a}$, then $\overline{X}_{a}$ is injective.

\item[(2)] There exists a unique closable operator $X_{r}$ with domain
$\im \bW_{o}^{*}$ mapping into
$(\kr D_{\widetilde \frakT_{F_{\Sigma}}})^\perp
\subset \ell^{2}_{\cU}({\mathbb Z}_{-})$
so that we have the factorization
\begin{equation}  \label{fact2}
    \bW_{c}^{*}|_{\im \bW_{o}^{*}} = D_{\widetilde
\frakT_{F_{\Sigma}}}
    X_{r}.
\end{equation}
Moreover, if we let $\overline{X}_{r}$ denote the closure of $X_{r}$,
then $\overline{X}_{r}$ is injective.
\end{enumerate}
\end{lemma}

\begin{proof}
    As statement (2) is just a dual version of statement (1), we only
    discuss the proof of (1) in detail.

Apply the Douglas factorization lemma to the first of the
inequalities in
\eqref{Hankel-est}  to get the existence of a unique contraction
operator
\[
Y_a:\ell^{2}_{\cU}({\mathbb Z}_{-})\to (\kr
D_{\frakT_{F_{\Sigma}}^{*}})^\perp  \subset \ell^{2}_{\cY}({\mathbb Z}_{+})
\]
such that
\[
D_{\frakT_{{F_\Si}}^{*}}Y_{a}
=\frakH_{{F_\Si}},
\quad\mbox{so that, by \eqref{HankelFact},}\quad
D_{\frakT_{{F_\Si}}^{*}}Y_{a}|_{\cD(\bW_{c})}
=\bW_{o} \bW_{c}.
\]
If we let $\bW_{c}^{\dagger}$  be the Moore-Penrose generalized inverse \eqref{MP} of $\bW_c$,
then
\[
\bW_{c}^{\dagger} (x) = {\text{arg min }} \{
\|\bu\|^{2}_{\ell^{2}_{\cU}({\mathbb Z}_{-})} \colon
\bu\in\cD(\bW_c),\  x =\bW_{c} \bu \}\quad (x\in \im \bW_c).
\]
Since $\bW_c$ is closed, $\kr \bW_c$ is a closed subspace of
$\ell^2_\cU(\BZ_-)$ and for all $\bu\in\cD(\bW_c)$ with $x = \bW_{c}
\bu$ we have
$\bW_{c}^{\dagger} (x)= \bu-P_{\kr \bW_c}\bu$.
We next define $X_{a} \colon \im \bW_{c} \to \ell^{2}_{\cY}({\mathbb
Z}_{+})$
by
$$
   X_{a} = Y_{a} \bW_{c}^{\dagger}.
$$
Then $X_{a}$ is a well-defined, possibly unbounded, operator on the
dense domain $\cD(X_{a}) = \im \bW_{c}$.  Moreover  we have
\[
D_{\frakT_{{F_\Si}}^{*}} X_a = D_{\frakT_{{F_\Si}}^{*}} Y_{a}
\bW_{c}^{\dagger} =
\frakH_{{F_\Si}} \bW_{c}^{\dagger} = \bW_{o} \bW_{c}
\bW_{c}^{\dagger} =
\bW_{o}|_{\im \bW_{c}}.
\]
Hence $X_a$ provides the factorization  \eqref{fact1}.
Furthermore, $X_a = Y_{a} \bW_{c}^{\dagger}$ implies that $\im X_a
\subset \im Y_a$, so that $\im X_a \subset (\kr
D_{\frakT_{{F_\Si}}^{*}})^\perp$.  Moreover, from the factorization
\eqref{fact1} we see that this property makes the choice of $X_{a}$
unique.

We now check that $X_a$ so constructed is closable. Suppose that
$\{x_{0}^{(k)}\}_{k \ge 0}$ is a sequence of vectors in $\im \bW_{c}$
such that $\lim_{k \to \infty} x_{0}^{(k)} = 0$ in $\cX$-norm, while
$\lim_{k \to \infty } X_a x_{0}^{(k)} = \by$ in
$\ell^{2}_{\cY}({\mathbb Z}_{+})$-norm. As $D_{\frakT_{{F_\Si}}^{*}}$
is bounded, it follows that
\[
\lim_{k \to \infty} \bW_{o} x_{0}^{(k)} =  \lim_{k \to \infty}
D_{\frakT_{{F_\Si}}^{*}} X_a x_{0}^{(k)}
= D_{\frakT_{{F_\Si}}^{*}} \by  \text{ in } \ell^{2}_{\cY}({\mathbb
Z}_{+})\text{-norm.}
\]
Since $\bW_{o}$ is a closed operator and we have $x_{0}^{(k)} \to 0$
in $\cX$-norm, it follows that $D_{\frakT_{{F_\Si}}^{*}} \by = 0$.
As $\im X_a \subset (\kr D_{\frakT_{{F_\Si}}^{*}})^{\perp}$ and $X_a
x_{0}^{(k)} \to \by$, we also have that $\by \in (\kr
D_{\frakT_{{F_\Si}}^{*}})^{\perp}$. It follows that $\by = 0$, and
hence $X_a$ is closable.

Let $\overline{X}_{a}$ be the closure of $X_{a}$.  We check that
$\overline{X}_{a}$ is injective as follows. The vector $x_{0}$ being
in  $\cD(\overline{X}_{a})$ means that there is a sequence of vectors
$\{x^{(k)}_{0}\}_{k \ge 1}$ contained in $\cD(X_{a})$ with $\lim_{k \to \infty}
x_{0}^{(k)} = x_{0}$ in $\cX$ and
$\lim_{k \to \infty} X_{a} x_{0}^{(k)} = \by$ for some
$\by \in \ell^{2}_{\cY}({\mathbb Z}_{+})$.  The condition that
$\overline{X}_{a} x = 0$ means that in addition $\by = 0$.  Since
$D_{\frakT_{F_{\Sigma}}^{*}}$ is bounded, it then follows that
$\lim_{k \to \infty} D_{\frakT_{F_{\Sigma}}^{*}}  X_{a} x_{0}^{(k)} =
0$, or, by \eqref{fact1}
$$
 \lim_{k \to \infty}\bW_{o} x_{0}^{(k)} = 0.
$$
As we also have $\lim_{k \to \infty} x_{0}^{(k)}  = x_{0}$ in $\cX$
and
$\bW_{o}$ is a closed operator, it follows that $x_{0} \in
\cD(\bW_{o})$ and $\bW_{o} x_{0} = 0$.  As $\bW_{o}$ is injective, it
follows that $x_{0} = 0$.  We conclude that $\overline{X}_{a}$ is
injective as claimed.
\end{proof}

Using the closed operators $\overline{X}_a$ and $\overline{X}_r$
defined in Lemma
\ref{L:fact} we now define (possibly unbounded)  positive-definite operators $H_a$ and $H_r$
so that the storage functions $S_a$ and $\uS_r$ have the quadratic forms $S_a = S_{H_a}$ and $\uS_r = S_{H_r}$ as in \eqref{QuadStorage1}.

We start with $H_a$.   Since
$\oX_{a}$ is closed, there is a good polar factorization
$$\oX_a =
U_{a} |\oX_a|$$
(see \cite[Theorem VIII.32]{RS}); in detail,
$\oX_{a}^{*} \oX_{a}$ is selfadjoint with positive selfadjoint
square-root $|\oX_a| = (\oX_{a}^{*} \oX_{a})^{\half}$ satisfying
$\cD(|\oX_a|) = \cD(\oX_{a})$, $U_{a}$ is a partial isometry with
initial space equal to $({\rm Ker}\, \oX_{a})^{\perp}$ and final
space equal to $\overline{\rm Im}\, \oX_{a}$ so that we have the
factorization $\oX_{a} = U_{a} |\oX_{a}|$.

Now set
\begin{equation}  \label{Ha-def}
H_a=\oX_{a}^{*}
\oX_{a},  \quad H_a^{\half}=|\oX_a|.
\end{equation}
As noted in Lemma \ref{L:fact}, $\oX_{a}$ is
injective, and thus $H_a$ and $H_a^{\half}$ are injective as well,
and as a result $U_{a}$ is an isometry.

We proceed with the definition of $H_r$. As the properties of
$\oX_{r}$ parallel those of $\oX_{a}$, $\oX_{r}$ has a good polar
decomposition $\oX_{r} = U_{r} | \oX_{r}|$ with $| \oX_{r}|$ and
$U_{r}$
having similar properties as $| \oX_{a}|$ and $U_{a}$, in particular,
$\oX_{r}^*\oX_{r}$ and  $|\oX_{r}|$ are injective and  $U_{r}$ is an
isometry.
We then define
\begin{equation}   \label{Hr-def}
 H_{r} = \left( \oX_{r}^{*} \oX_{r} \right)^{-1}, \quad
 H_{r}^{\half} = | \oX_{r} |^{-1}.
\end{equation}
We shall also need a modification of the factorization
\eqref{fact2}.  For $\bu \in \cD(\bW_{c})$ and $x \in \im
\bW_{o}^{*}$, let us note that
\begin{align*}
\langle \bW_{c} \bu, x \rangle_{\cX}
&=\langle \bu, \bW_{c}^{*} x \rangle_{\ell^{2}_{\cU}({\mathbb
    Z}_{-})} = \langle \bu, D_{\widetilde \frakT_{F_{\Sigma}}} X_{r} x
    \rangle_{\ell^{2}_{\cU}({\mathbb Z}_{-})} \text{ (by
    \eqref{fact2})}\\
& = \langle  D_{\widetilde \frakT_{F_{\Sigma}}} \bu, X_{r} x
 \rangle_{\ell^{2}_{\cU}({\mathbb Z}_{-})}.
\end{align*}
The end result is that then $D_{\widetilde \frakT_{F_{\Sigma}}} \bu$
is in $\cD(X_{r}^{*})$ and $X_{r}^{*}  D_{\widetilde
\frakT_{F_{\Sigma}}} \bu = \bW_{c} \bu$.  In summary we have
the following adjoint version of the factorization  \eqref{fact2}:
\begin{equation}  \label{fact3}
   \bW_{c} = X_{r}^{*}  D_{\widetilde \frakT_{F_{\Sigma}}}
   |_{\cD(\bW_{c})}.
\end{equation}

In the following statement we use the notion of a {\em core} of a closed, densely defined operator $\Gamma$ between
two Hilbert spaces $\cH$ and $\cK$ (see \cite{RS} or \cite{Kato}), namely:  a dense linear submanifold $\cD$ is said
to be a {\em core} for the closed, densely defined operator $X$ with domain $\cD(X)$ in $\cH$ mapping
into $\cK$ if, given any $x \in \cD(X)$, there is a sequence $\{ x_n \}_{n \ge 1}$ of points in $\cD$ such that
$\lim_{n \to \infty} x_n = x$ and also $\lim_{n \to \infty} X x_n = X x$.

\begin{theorem}\label{T:Sar}
Let the discrete-time linear system $\Si$ in \eqref{dtsystem} satisfy
the assumptions in \eqref{A}. Define $X_a$, $\oX_{a}$,
$X_{r}$, $\oX_r$ as in Lemma
\ref{L:fact} and the closed operators $H_{a}$ and
$H_r$ as in the preceding discussion. Then the available storage
function $S_{a}$ and required supply function $S_r$ are given
by
\begin{align}
S_{a}(x_{0}) = \| \oX_{a} x_{0} \|^{2} = \| H_{a}^{\half}
x_{0}\|^{2}\quad (x_0\in \im \bW_{c}), \label{form1}\\
\uS_{r}(x_{0}) = \| | \oX_{r} |^{-1} x_{0}
\|^{2}=\|H_{r}^{\half} x_{0}\|^2\quad (x_0\in \im \bW_c)
\label{form2}.
\end{align}
In particular, the available storage $S_{a}$ and $\ell^2$-regularized required supply
$\uS_r$ agree with quadratic storage functions on $\im \bW_c$.

Moreover, $\im \bW_c$ is a core for $H_a^\half$ and $\im \bW_o^*$ is a core for $H_r^{-\half}$.
\end{theorem}

\begin{proof}
By Lemma \ref{L:fact}, in the operator form of $S_a$ derived in Lemma
\ref{L:SaSrOpForm} we can replace $\bW_{o} x_{0}$ by
$D_{\frakT_{{F_\Si}}^{*}} X_{a} x_{0}$, leading to
\begin{equation}\label{sup}
S_{a}(x_{0}) = \sup_{ \bu \in \ell^{2}_{\cU}({\mathbb Z}_{+})}
\| D_{\frakT_{{F_\Si}}^{*}} X_{a} x_{0} + \frakT_{{F_\Si}} \bu
\|^{2}_{\ell^{2}_{\cY}({\mathbb
 Z}_{+})} - \| \bu \|^{2}_{\ell^{2}_{\cU}({\mathbb Z}_{+})}.
\end{equation}
For $x_0\in \im \bW_c$ and each $\bu\in\ell^2_\cU(\BZ_+)$ we have
\begin{align*}
&\| D_{\frakT_{{F_\Si}}^{*}} X_{a} x_{0} + \frakT_{{F_\Si}} \bu \|^{2}
- \| \bu \|^{2}=\\
&\qquad\qquad=\| D_{\frakT_{{F_\Si}}^{*}} X_{a} x_{0} \|^{2} +
2\,\text{Re}\,\langle D_{\frakT_{{F_\Si}}^{*}} X_{a} x_{0},
\frakT_{{F_\Si}} \bu \rangle  + \| \frakT_{{F_\Si}}\bu \|^{2}- \|\bu
\|^{2}\\
&\qquad\qquad=\| D_{\frakT_{{F_\Si}}^{*}} X_{a} x_{0} \|^{2} +
2\,\text{Re}\,\langle D_{\frakT_{{F_\Si}}^{*}} X_{a} x_{0},
\frakT_{{F_\Si}} \bu \rangle  - \| D_{\frakT_{{F_\Si}}}\bu \|^{2}\\
&\qquad\qquad=\| D_{\frakT_{{F_\Si}}^{*}} X_{a} x_{0} \|^{2} +
2\,\text{Re}\,\langle X_{a} x_{0}, D_{\frakT_{{F_\Si}}^{*}}
\frakT_{{F_\Si}} \bu \rangle - \| D_{\frakT_{{F_\Si}}} \bu \|^{2}  \\
&\qquad\qquad= \| D_{\frakT_{{F_\Si}}^{*}} X_{a} x_{0} \|^{2} +
2\,\text{Re}\,\langle X_{a} x_{0}, \frakT_{{F_\Si}}
D_{\frakT_{{F_\Si}}} \bu \rangle - \| D_{\frakT_{{F_\Si}}} \bu \|^{2}
\\
&\qquad\qquad=  \| D_{\frakT_{{F_\Si}}^{*}} X_{a} x_{0} \|^{2} +
2\,\text{Re}\,\langle \frakT_{{F_\Si}}^{*} X_{a} x_{0},
D_{\frakT_{{F_\Si}}} \bu \rangle - \| D_{\frakT_{{F_\Si}}} \bu \|^{2}
\\
&\qquad\qquad=  \| D_{\frakT_{{F_\Si}}^{*}} X_{a} x_{0} \|^{2} +
\|\frakT_{{F_\Si}}^{*} X_{a}
x_{0} \|^{2} - \| \frakT_{{F_\Si}}^{*} X_{a} x_{0} -
D_{\frakT_{{F_\Si}}} \bu \|^{2}\\
&\qquad\qquad=  \| X_{a} x_{0} \|^{2} - \| \frakT_{{F_\Si}}^{*} X_{a}
x_{0} -
D_{\frakT_{{F_\Si}}} \bu \|^{2}.
 \end{align*}
By construction, we have $\im X_{a} \subset (\kr
D_{\frakT_{{F_\Si}}^{*}})^{\perp} =\overline{\im}
D_{\frakT_{{F_\Si}}^{*}}$. Using that
$\frakT_{{F_\Si}}^{*}D_{\frakT_{{F_\Si}}^{*}}=D_{\frakT_{{F_\Si}}}
\frakT_{{F_\Si}}^{*}$, we obtain
$$
\frakT_{{F_\Si}}^{*} \overline{\im}
D_{\frakT_{{F_\Si}}^{*}}\subset \overline{\im} D_{\frakT_{{F_\Si}}}.
$$
Thus $\im \frakT_{{F_\Si}}^{*} X_{a} \subset \overline{\im}
D_{\frakT_{{F_\Si}}}$. Hence there is a sequence $\bu_{k}$  of input
signals in $\ell^{2}_{\cU}({\mathbb Z}_{+})$ so that  $\|
\frakT_{{F_\Si}}^{*} X_{a} x_{0} - D_{\frakT_{{F_\Si}}} \bu_{k} \|
\to 0$ as $k \to \infty$. We conclude that for $x_0\in \im \bW_c$ the
supremum in \eqref{sup} is given by
\[
S_{a}(x_{0}) = \| X_{a} x_{0} \|^{2}=\| \oX_{a} x_{0} \|^{2}=\|
H_a^{\half} x_{0} \|^{2}.
\]

Let $x_0\in \im\bW_c$. Given a $\bu \in\cD(\bW_c)$, by the
factorization \eqref{fact3} we see that
$\bW_c\bu=x_0$ if and only if $X_r^{*}
D_{\widetilde\frakT_{F_\Si}}\bu=x_0$. Therefore, we have
\begin{align*}
\uS_r(x_0)
&=\inf_{\bu \in \cD(\bW_{c}), \, X_r^{*}
D_{\widetilde\frakT_{F_\Si}}\bu=x_0}
\|  D_{\widetilde \frakT_{{F_\Si}}} \bu \|^{2}
=\inf_{\bv \in D_{\widetilde \frakT_{{F_\Si}}}\cD(\bW_{c}), \, X_r^{*}
\bv=x_0}
\|  \bv \|^{2}.
\end{align*}
A general property of operator closures is $\oX_{r}^{*} = X_{r}^{*}$.
Hence
\begin{equation}   \label{inf1}
\uS_{r}(x_{0}) = \inf_{\bv \in D_{\widetilde
\frakT_{{F_\Si}}}\cD(\bW_{c}), \, \oX_r^{*}\bv=x_0} \|
\bv \|^{2}.
\end{equation}
As $x_{0} \in \im \bW_{c}$ by assumption, the factorization
\eqref{fact3} gives us a $\bu_{0} \in \cD(\bW_{c})$ so that
\begin{equation}   \label{x0}
  x_{0} = \oX_{r}^{*} D_{\widetilde \frakT_{F_{\Sigma}}} \bu_{0}.
\end{equation}
In particular, $x_{0}$ has the form $x_{0} = \oX_{r}^{*} \bv_{0}$ with
$\bv_{0} \in D_{\widetilde \frakT_{F_{\Sigma}}} \cD(\bW_{c})$.  From \eqref{x0} we see that the general
solution $\bv \in \cD(\oX_{r}^{*})$ of $x_{0} = \oX_{r}^{*} \bv$ is
\begin{equation}   \label{gensol}
  \bv = D_{\widetilde \frakT_{F_{\Sigma}}} \bu_{0} + k
\text{ where } k \in \kr \oX_{r}^{*}.
\end{equation}
By construction the target space for $X_{r}$ (and $\oX_{r}$) is $\left(
\kr D_{\widetilde \frakT_{F_{\Sigma}}}\right)^{\perp}$ so the domain
space for $\oX_{r}^{*}$ is $( \kr D_{\widetilde
\frakT_{F_{\Sigma}}})^\perp$ and $\kr \oX_r^* \subset \overline{\im} D_{\widetilde \fT_{F_\Sigma}}$.
Hence the infimum in   \eqref{inf1} remains unchanged if we relax the
constraint $\bv \in D_{\widetilde \frakT_{F_{\Sigma}}} \cD(\bW_c)$ to just $\bv
\in \cD(\oX_{r}^{*})$, i.e.,
\begin{equation}   \label{inf2}
    \uS_{r}(x_{0}) = \inf_{ \bv \in \cD(\oX_{r}^{*}), \, \oX_{r}^{*} \bv
    = x_{0}} \| \bv\|^{2}.
\end{equation}
In terms of the polar decomposition $\oX_{r} = U_{r} | \oX_{r}|$ for
$\oX_{r}$, we have
$$
\oX_{r}^{*} =  | \oX_{r}| U_{r}^{*}
$$
with
$$
\cD(\oX_{r}^{*})  = \{ \bu \in \overline{\im} D_{\widetilde
\frakT_{F_{\Sigma}}} \colon U_{r}^{*} \bu \in \cD(|\oX_{r}|) =
\cD(\oX_{r}) \}.
$$
Since $|\oX_{r}|$ is injective and $U_{r}$ is an isometry with range
equal to $(\kr \oX_{r}^{*})^{\perp}$, the constraint $|\oX_{r}|
U_{r}^{*} \bv = \oX_{r}^{*} \bv = x_{0}$ is equivalent to
$$
 P_{(\kr \oX_{r}^{*})^{\perp}} \bv = U_{r} U_{r}^{*} \bv = U_{r}
 |\oX_{r}|^{-1} x_{0}.
$$
Since we want to minimize $\|\bv\|^2$ with $P_{(\kr
\oX_{r}^*)^\perp}\bv$
equal to $U_r|\oX_{r}^*|^{-1}x_0\in \cD(\oX_r)$, it is clear that
this is achieved at $\bv_\textup{opt}=U_r|\oX_{r}^*|^{-1}x_0$, so that
\[
\uS_r(x_0)=\|\bv_\textup{opt}\|^2=\|U_r|\oX_{r}|^{-1}x_0\|^2=\||\oX_{r}|^{-1}x_0\|^2
=\|H_r^\half x_0\|^2,
\]
as claimed.

It remains to verify the last assertion concerning the core properties of $\im \bW_c$ and $\im \bW_o^*$.
By definition $H_a^\half = | \overline{X}_a|$ where $\overline{X}_a$ is defined to be the closure of the $X_a =
\overline{X}_a|_{\im \bW_c}$.  Hence $\im \bW_c$ by definition is a core for $\overline{X}_a$ from which it immediately follows
that $\im \bW_c$ is a core for $H_a^\half = | \overline{X}_a |$.  That $\im \bW_o^*$ is a core for $H_r^{-\half} = | \overline{X}_r |$
follows in the same way via a dual analysis.
\end{proof}

\section{The dual system $\Si^*$}  \label{S:dual}

In this section we develop a parallel theory for the dual system $\Si^*$ of $\Si$, which is the system with system matrix equal to the adjoint of \eqref{sysmat} evolving in backward-time.

\subsection{Controllability, observability, minimality and transfer function for the dual system}

With the discrete-time linear system $\Sigma$ given by \eqref{dtsystem} with system matrix $M = \sbm{ A & B \\ C & D }$ we associate the
dual system $\Sigma^*$ with system matrix $M^* = \sbm{ A^* & C^* \\ B^* & D^* } \colon \sbm{ \cX \\ \cY} \to \sbm{ \cX \\ \cU}$.
It will be convenient for our formalism here to let the dual system evolve in backward time; we therefore define the system
$\Sigma^*$ to be given by the system input/state/output equations
\begin{equation}  \label{dtsystem*}
\Sigma^* \colon = \left\{ \begin{array}{rcl}
\bx_*(n-1) & = & A^* \bx_*(n) + C^* \bu_*(n), \\
\by_*(n) & = & B^* \bx_*(n) + D^* \bu_*(n).    \end{array}
\right.
\end{equation}
If we impose a final condition $\bx_*(-1) = x_0$ and feed in an input-sequence $\{ \bu(n) \}_{n \in {\mathbb Z}_-}$, one can solve recursively
to get, for $n \le -1$,
$$
\left\{ \begin{array}{rcl}
\bx_*(n) & = &  A^{* -n-1} x_0 + \sum_{j=n+1}^{-1} A^{* -n+j} C^* \bu_*(j), \\
  \by_*(n) & = & B^* A^{* -n-1} x_0 + \sum_{j = n+1}^{-1} B^* A^{* -n+j-1} C^* \bu_*(j) + D^* \bu_*(n).
\end{array}  \right.
$$
Alternatively, the $Z$-transform
$
\{\bx_*(n)\}_{n \in {\mathbb Z}_-} \mapsto \widehat \bx_*(\lambda) =
\sum_{n=-\infty}^{-1} \bx_*(n) \lambda^n
$
may be applied directly to the system equations \eqref{dtsystem*}.  Combining this with the observation that
\begin{align*}
\sum_{n=-\infty}^{-1} \bx_*(n-1) \lambda^n & = \lambda \left( \sum_{n=-\infty}^{-1} \bx_*(n-1) \lambda^{n-1} \right)
 = \lambda \left( \sum_{n=-\infty}^{-2} \bx_*(n) \lambda^{n} \right) \\
 & = \lambda \left( \widehat \bx_*(\lambda) - x_0 \lambda^{-1} \right) = \lambda \widehat \bx_*(\lambda) - x_0.
\end{align*}
converts the first system equation in \eqref{dtsystem*} to
$$
 \lambda \widehat \bx_*(\lambda) - x_0 = A^* \widehat \bx_*(\lambda) + C^* \widehat \bu_*(\lambda)
$$
leading to the $Z$-transformed version of the whole system:
$$
\left\{ \begin{array}{rcl}
  \widehat \bx_*(\lambda) & = & (\lambda I - A^*)^{-1} x_0 + (\lambda I - A^*)^{-1} C^* \widehat \bu_*(\lambda), \\
  \widehat \by_*(\lambda) & = & B^* (\lambda I - A^*)^{-1} x_0 + F_{\Sigma^*}(\lambda) \widehat \bu_*(\lambda),
\end{array} \right.
$$
where the {\em transfer function} $F_{\Sigma^*}(\lambda)$ for the system $\Sigma^*$ is then given by
\begin{align}
 F_{\Sigma^*}(\lambda) & = D^* + B^* (\lambda I - A^*)^{-1} C^*  \notag  \\
    & = D^* + \lambda^{-1} (I - \lambda^{-1} A^*)^{-1} C^*  = F_\Sigma(1/\overline{\lambda})^*  \label{transfunc*}
\end{align}
which is an analytic function on a neighborhood of the point at $\infty$ in the complex plane.  Moreover, $F_{\Sigma^*}$ has analytic
continuation to a function analytic on the exterior of the unit disk ${\mathbb D}_e : = \{ \lambda \in {\mathbb C} \colon |\lambda| > 1\}
 \cup \{\infty\}$ exactly when $F_\Sigma$ has analytic continuation to a function analytic on the unit disk ${\mathbb D}$ with equality of
 corresponding $\infty$-norms:
 $$
 \| F_{\Sigma_*}  \|_{\infty, {\mathbb D}_e} : = \sup_{\lambda \in {\mathbb D}_e}
    \| F_{\Sigma_*}(\lambda) \| =
  \sup_{\lambda \in {\mathbb D}}  \| F_\Sigma(\lambda) \|  =: \| F_\Sigma \|_{\infty, {\mathbb D}}.
 $$

All the analysis done up to this point for the system $\Sigma$ has a dual analogue for the system $\Sigma^*$.
In particular, the observability operator $\bW_{*o}$ for the dual system is obtained by running the system \eqref{dtsystem*}
with final condition $\bx_*(-1) = x_0$ and input string $\bu_*(n) = 0$ for $n \le -1$, resulting in the output string
$\{ B^* A^{* (-n-1)} x_0 \}_{n \in {\mathbb Z}_-}$.  Since we are interested in a setting with operators on $\ell^2$,
we define the {\em observability operator} $\bW_{*o}$ for $\Sigma^*$ to have domain
$$
  \cD(\bW)_{*o} = \{ x_0 \in \cX \colon \{ B^* A^{*(- n-1)} x_0\}_{n\in\BZ_-}  \in \ell^2_\cU({\mathbb Z}_-)\}
  $$
  with action given by
$$
  \bW_{*o} x_0 = \{ B^* A^{* (-n-1)} x_0 \}_{n \in {\mathbb Z}_-} \text{ for } x_0 \in \cD(\bW_{*o}).
$$
Note that $\bW_{*o}$ so defined is exactly the same as the adjoint controllability operator $\bW_c^*$ for the original system
\eqref{bWc*1}--\eqref{bWc*2}, and in fact viewing this operator as $\bW_{*o}$ gives a better control-theoretic interpretation
for this operator.   Similarly it is natural to define the adjoint controllability operator for the adjoint system $(\bW_{*c})^*$
by
$$
 \cD((\bW_{*c})^*) = \{ x_0 \in \cX \colon \{ C A^n x_0 \}_{n \in {\mathbb Z}_+} \in \ell^2_\cY({\mathbb Z}_+)\}
  = \cD(\bW_o)
$$
with action given by
$$
  \bW_{*c}^* x_0 = \{ C A^n x_0 \}_{n \in {\mathbb Z}_+} = \bW_o x_0.
$$
In view of the equalities
\begin{equation}   \label{identifications}
  \bW_{*o} = \bW_c^*, \quad (\bW_{*c})^* = \bW_o, \quad (\bW_{*o})^* = \bW_c, \quad \bW_{*c} = \bW_o^*,
\end{equation}
one can work out the dual analogue of Proposition \ref{P:WcWo'}, either by redoing the original proof with the backward-time
system $\Sigma^*$ in place of the forward-time system $\Sigma$, or simply by making the substitutions \eqref{identifications}
in the statement of the results.

Let us now assume that $F_\Sigma$ has analytic continuation to a bounded analytic $\cL(\cU, \cY)$-valued function on the unit disk,
or equivalently,  $F_{\Sigma^*}$ has analytic continuation to a bounded analytic  $\cL(\cY, \cU)$-valued function on the exterior of
the unit disk ${\mathbb D}_e$.  Then $F_\Sigma$ and $F_{\Sigma^*}$ can be identified
via strong nontangential boundary-value limits with $L^\infty$-functions on the unit circle ${\mathbb T}$;
the relations between these boundary-value functions is simply
$$
   F_{\Sigma^*}(\lambda) = F_\Sigma(\lambda)^* \quad (\mbox{a.e. } \la\in\BT)
$$
with the consequence that the associated multiplication operators
$$ M_{F_\Sigma} \colon L^2_\cU({\mathbb T}) \to L^2_\cY({\mathbb T}), \quad
M_{F_{\Sigma^*}} \colon L^2_\cY({\mathbb T}) \to L^2_\cU({\mathbb T})
$$
given by
$$
 M_{F_\Sigma} \colon \widehat \bu(\lambda) \mapsto \widehat  F_\Sigma(\lambda) \cdot \widehat \bu(\lambda), \quad
 M_{F_{\Sigma^*}} \colon \widehat \bu_*(\lambda) \mapsto \widehat  F_{\Sigma_*}(\lambda) \cdot \widehat \bu_*(\lambda)
 $$
 are adjoints of each other:
$$
  (M_{F_\Sigma})^* = M_{F_{\Sigma^*}}.
$$
Note also that $M_{F_\Sigma}$ maps $H^2_\cU({\mathbb D})$ into $H^2_\cY({\mathbb D})$ while
$M_{F_{\Sigma^*}} = M_{F_\Sigma}^*$ maps
$(H^2_\cY)^\perp: = L^2_\cY({\mathbb T}) \ominus H^2_\cY({\mathbb D}) \cong H^2_\cY({\mathbb D}_e)$ into
$(H^2_\cU)^\perp := L^2_\cU({\mathbb T}) \ominus H^2_\cU({\mathbb D}) \cong H^2_\cU({\mathbb D}_e)$.

It is natural to define the frequency-domain Hankel operator ${\mathbb H}_{F_{\Sigma^*}}$ for the adjoint system
as the operator from $H^2_\cY({\mathbb D}_e)^\perp = H^2_\cY({\mathbb D})$ (the past from the point of view of the
backward-time system $\Sigma^*$) to $H^2_\cU({\mathbb D}_e) =H^2_\cU({\mathbb D})^\perp$
(the future from the point of view of $\Sigma^*$) by
\begin{equation}  \label{Hankel-identification}
  {\mathbb H}_{F_{\Sigma^*}} = P_{H^2_\cU({\mathbb D})^\perp} M_{F_{\Sigma^*}}|_{H^2_\cY({\mathbb D})} =
  ( {\mathbb H}_{F_\Sigma})^*.
\end{equation}
After application of the inverse $Z$-transform, we see that the time-domain version $\fH_{F_{\Sigma^*}}$  of the Hankel
operator for $\Sigma^*$ is just the adjoint $(\fH_{F_\Sigma})^*$ of the time-domain version of the Hankel operator for $\Sigma$,
namely
$$
 \fH_{F_{\Sigma^*}} = [ B^* A^{*(-i + j -1)} C^* ]_{i < 0, j \ge 0} \colon \ell^2_\cY({\mathbb Z}_+) \to \ell^2_\cU({\mathbb Z}_-).
$$
from which we see immediately the formal factorization
\begin{equation}  \label{Hankel-fact*}
\fH_{F_{\Sigma^*}} = \operatorname{col}_{i<0} [B^* A^{*( -i -1)}] \cdot
\operatorname{row}_{j \ge 0} [A^{*j} C^* ] = \bW_{*o} \bW_{*c} = \bW_c^* \bW_o^*.
\end{equation}
With all these observations in place, it is straightforward to formulate the dual version of Proposition \ref{P:HankelDecs},
again, either by redoing the proof of Proposition \ref{P:HankelDecs} with the backward-time system $\Sigma^*$ in place of
the forward-time system $\Sigma$, or by simply substituting the identifications
\eqref{identifications} and \eqref{Hankel-identification}.

Note next that an immediate consequence of the identifications \eqref{identifications} is that
$\ell^2$-exact controllability for $\Sigma$ is the same as $\ell^2$-exact observability for $\Sigma^*$
and $\ell^2$-exact observability for $\Sigma$ is the same as $\ell^2$-exact controllability for $\Sigma^*$.
With this observation in hand, the dual version of Proposition \ref{P:ell2implics} is immediate.

\subsection{Storage functions for the adjoint system}

Let $S_*$ be a function from $\cX$ to $[0, \infty]$.  In parallel with what is done in Section \ref{S:Storage}, we define
$S_*$ to be  a {\em storage function for the system $\Sigma^*$} if
\begin{equation}   \label{disineq*}
  S_*(\bx_*(n-1))) \le S_*(\bx_*(n)) + \| \bu_*(n) \|^2 - \| \by_*(n) \|^2_\cY \text{ for } n \le N_0
\end{equation}
holds over all system trajectories $(\bu_*(n), \bx_*(n), \by_*(n))_{n \le N_0}$  of the system $\Sigma^*$ in \eqref{dtsystem*} with state initialization
$\bx_*(N_0) = x_0$ for some $x_0 \in \cX$ at some $N_0 \in {\mathbb Z}$, and $S_*$ is normalized to satisfy
\begin{equation}   \label{normalization*}
  S_*(0) = 0.
\end{equation}
Then by redoing the proof of
Proposition \ref{P:storage-Schur} with the backward-time system $\Sigma_*$ in place of the forward-time system $\Sigma$,
we arrive at the following dual version of Proposition \ref{P:storage-Schur}.

\begin{proposition}   \label{P:storage-Schur*}  Suppose that the system $\Sigma^*$ in \eqref{dtsystem*} has a storage function
$S_*$ as in \eqref{disineq*} and \eqref{normalization*}.  Then the transfer function $F_{\Sigma^*}$ of $\Sigma^*$ defined by
\eqref{transfunc*} has an analytic continuation to the exterior unit disk ${\mathbb D}_e$ in the Schur class
$\cS_{{\mathbb D}_e}(\cY, \cU)$.
\end{proposition}

Note that by the duality considerations already discussed above, an equivalent conclusion is that $F_\Sigma$ has analytic
continuation to the unit disk in the Schur class $\cS(\cU, \cY)$ over the unit disk.

We say that $S_*$ is a {\em quadratic storage function}  for $\Sigma^*$ if $S_*$ is a storage function of the form
\begin{equation}  \label{QuadStorage1*}
  S_*(x) = S_{H_*}(x) = \begin{cases}
    \| H_*^\half x \|^2 &\text{for } x \in \cD(H_*^\half), \\
    +\infty &\text{otherwise.}  \end{cases}
\end{equation}
where $H_*$ is a (possibly) unbounded positive-semidefinite operator on $\cX$.    To analyze quadratic storage functions for
$\Sigma^*$, we introduce the adjoint KYP-inequality: we say that the bounded selfadjoint operator $H$ on $\cX$ satisfies the
{\em adjoint KYP-inequality} if
\begin{equation}  \label{KYP1*}
  \begin{bmatrix} A & B \\ C & D \end{bmatrix} \begin{bmatrix} H_* & 0 \\ 0 & I_\cU \end{bmatrix}
  \begin{bmatrix} A & B \\ C & D \end{bmatrix}^* \preceq \begin{bmatrix} H_* & 0 \\ 0 & I_\cU \end{bmatrix}.
\end{equation}
More generally, for a  (possibly) unbounded positive-semidefinite operator $H_*$ on $\cX$, we say that $H_*$
satisfies the {\em generalized KYP-inequality} if, for all $x \in \cD(H_*^\half)$ we have
\begin{equation}   \label{KYP1b'*}
A^* \cD(H_*^\half) \subset \cD(H_*^\half), \quad C^* \cY \subset \cD(H_*^\half),
\end{equation}
and for all $x_* \in \cD(H_*^\half)$ and $u_* \in \cY$ we have
\begin{equation}   \label{KYP1b*}
  \left\| \begin{bmatrix} H_*^\half & 0 \\ 0 & I_\cY \end{bmatrix} \begin{bmatrix} x_* \\ u_* \end{bmatrix} \right\|^2
  - \left\| \begin{bmatrix} H_*^\half & 0 \\ 0 & I_\cU \end{bmatrix}
 \begin{bmatrix} A^* & C^* \\ B^* & D^* \end{bmatrix} \begin{bmatrix} x_* \\ u_* \end{bmatrix} \right\|^2
 \ge 0.
 \end{equation}
 Then the dual version of Proposition \ref{P:QuadStorage} is straightforward.

 \begin{proposition} \label{P:QuadStorage*} Suppose that the function $S_*$ has the form \eqref{QuadStorage1*}
  for a (possibly) unbounded positive-semidefinite operator $H_*$ on $\cX$.   Then $S_*$ is a storage function for $\Sigma^*$ if and only if
 $H_*$ is a solution of the generalized adjoint-KYP inequality \eqref{KYP1b'*}--\eqref{KYP1b*}.  In particular,
 $S_*$ is a finite-valued storage function for $\Sigma^*$ if and only if $H$ is a bounded positive-semidefinite operator
 satisfying the adjoint KYP-inequality \eqref{KYP1*}.
  \end{proposition}

  We next discuss a direct connection between positive-definite solutions $H$ of the KYP-inequality \eqref{KYP1}
  and positive-definite solutions $H_*$ of the adjoint KYP-inequality \eqref{KYP1*}.
First let us suppose that $H$ is a bounded strictly positive-definite solution of the KYP-inequality \eqref{KYP1}.  Set
$$
Q = \begin{bmatrix} H^\half & 0 \\ 0 & I_\cY\end{bmatrix}  \begin{bmatrix} A & B \\ C & D \end{bmatrix} \begin{bmatrix}
H^{-\half} & 0 \\ 0 & I_\cU \end{bmatrix}.
$$
 Then the KYP-inequality \eqref{KYP1} is equivalent to $Q^* Q \preceq I$, i.e., the
fact that the operator $Q \colon \sbm{ \cX \\ \cU} \to \sbm{ \cX \\ \cY}$ is a contraction operator. But then the adjoint $Q^*$ of $Q$ is
also a contraction operator, i.e., $Q Q^* \preceq I$.  Writing out
$$
Q^* = \begin{bmatrix} H^{-\half} & 0 \\ 0 & I_\cU \end{bmatrix}
\begin{bmatrix} A^* & C^* \\ B^* & D^* \end{bmatrix} \begin{bmatrix} H^\half & 0 \\ 0 & I_\cY \end{bmatrix}
$$
and rearranging gives
$$
\begin{bmatrix} A & B \\ C & D \end{bmatrix} \begin{bmatrix} H^{-1} & 0 \\ 0 & I_\cU \end{bmatrix}
\begin{bmatrix} A^* & C^* \\ B^* & D^* \end{bmatrix} \preceq \begin{bmatrix} H^{-1} & 0 \\ 0 & I_\cY \end{bmatrix},
$$
i.e.,  $H_* := H^{-1}$ is a solution of the adjoint KYP-inequality \eqref{KYP1*} for the adjoint system $\Sigma_*$.
Conversely, by flipping the roles of $\Sigma$ and $\Sigma^*$ and using that $\Sigma^{**} = \Sigma$,
we see that if $H_*$ is a bounded, strictly positive-definite solution of the adjoint KYP-inequality \eqref{KYP1*},
then $H : = H_*^{-1}$ is a bounded, strictly positive-definite solution of the KYP-inequality \eqref{KYP1}.

The same correspondence between solutions of the generalized KYP-inequality \eqref{KYP1b'}--\eqref{KYP1b} for $\Sigma$
and solutions of the generalized KYP-inequality for the adjoint system
\eqref{KYP1b'*}--\eqref{KYP1b*}  continues to hold, but the details are
more delicate, as explained in the following proposition.  For an alternative proof see Proposition 4.6 in \cite{AKP06}.

\begin{proposition}  \label{P:KYPduality}  Suppose $\Sigma$ in \eqref{dtsystem} is a linear system with system matrix
$M = \sbm{ A & B \\ C & D}$ while $\Sigma^*$ is the adjoint system \eqref{dtsystem*} with system matrix $M^* =
\sbm{ A^* & C^* \\ B^* & D^* }$.  Then the {\rm(}possibly unbounded{\rm)} positive-definite operator $H$ is a solution of the generalized
KYP-inequality \eqref{KYP1b'}--\eqref{KYP1b} for $\Sigma$ if and only if $H^{-1}$ is a positive-definite solution of
the generalized KYP-inequality \eqref{KYP1b'*}--\eqref{KYP1b} for $\Sigma^*$.
 \end{proposition}

 \begin{proof}
 Suppose that the positive-definite operator $H$ with dense domain $\cD(H)$ in $\cX$ solves the generalized
 KYP-inequality \eqref{KYP1b'}--\eqref{KYP1b}.  Define an operator $Q \colon \sbm{ \im H^\half \\ \cU } \to \sbm{ \im H^\half \\ \cY}$
 by
 $$
   Q \colon \begin{bmatrix} H^\half & 0 \\ 0 & I_\cU \end{bmatrix} \begin{bmatrix} x \\ u \end{bmatrix} \mapsto
   \begin{bmatrix} H^\half & 0 \\ 0 & I_\cY \end{bmatrix} \begin{bmatrix} A & B \\ C & D \end{bmatrix} \begin{bmatrix} x \\ u \end{bmatrix}
$$
for $x \in \cD(H^\half)$ and $u \in \cU$.  We can write the formula for $Q$ more explicitly in terms of $x' = H^\half x \in \im H^\half$
as
$$
  Q \colon \begin{bmatrix} x' \\ u \end{bmatrix} \mapsto \begin{bmatrix} H^\half & 0 \\ 0 & I_\cY \end{bmatrix}
  \begin{bmatrix} A & B \\ C & D \end{bmatrix} \begin{bmatrix} H^{-\half} & 0 \\ 0 & I_\cU \end{bmatrix} \begin{bmatrix}
  x' \\ u \end{bmatrix}
 $$
 for $x' \in \im H^\half$ and $u \in \cU$.
The content of the generalized KYP-inequality \eqref{KYP1b'}--\eqref{KYP1b} is that $Q$ is a well-defined contraction operator
from $\sbm{ \im H^\half \\ \cU}$ into $\sbm{\cX \\ \cY}$ and hence has a uniquely determined contractive extension to a
contraction operator from $\sbm{ \cX \\ \cU}$ to $\sbm{ \cX \\ \cY}$.  Let us now choose arbitrary vectors
$x \in \cD(H^\half)$, $x_* \in \cD(H^{-\half}) = \im H^\half$, $u \in \cU$, $u_* \in \cY$ and set
$x' = H^\half x$, $x_*' = H^{-\half} x_*$. Then we compute on the one hand
 \begin{align*}
 \left \langle \begin{bmatrix} A & B \\ C & D \end{bmatrix} \begin{bmatrix} x \\ u \end{bmatrix}, \begin{bmatrix} x_* \\ u_* \end{bmatrix}
 \right\rangle & =
 \left\langle \begin{bmatrix} A & B \\ C & D \end{bmatrix} \begin{bmatrix} H^{-\half} & 0 \\ 0 & I \end{bmatrix}
 \begin{bmatrix} x' \\ u \end{bmatrix}, \begin{bmatrix} H^\half x_*' \\ u_* \end{bmatrix} \right\rangle \\
 & =
 \left\langle \begin{bmatrix} H^\half  & 0 \\ 0 & I \end{bmatrix} \begin{bmatrix} A & B \\ C & D \end{bmatrix}
  \begin{bmatrix} H^{-\half} & 0 \\ 0 & I \end{bmatrix} \begin{bmatrix} x' \\ u \end{bmatrix}, \,
  \begin{bmatrix}  x_*' \\ u_* \end{bmatrix} \right\rangle  \\
  & =  \left\langle Q \begin{bmatrix} x' \\ u \end{bmatrix}, \, \begin{bmatrix} x_*' \\ u_* \end{bmatrix} \right \rangle
  = \left \langle  \begin{bmatrix} x' \\ u \end{bmatrix}, \, Q^* \begin{bmatrix} x_*' \\ u_* \end{bmatrix} \right\rangle \\
  & = \left\langle \begin{bmatrix} H^\half x \\ u \end{bmatrix}, Q^* \begin{bmatrix} H^{-\half} x_* \\ u_* \end{bmatrix}
  \right\rangle
 \end{align*}
 while on the other hand
 $$
 \left\langle \begin{bmatrix} A & B \\ C & D \end{bmatrix} \begin{bmatrix} x \\ u \end{bmatrix}, \, \begin{bmatrix} x_* \\ u_* \end{bmatrix}
 \right\rangle = \left\langle \begin{bmatrix} x \\ u \end{bmatrix}, \, \begin{bmatrix} A^* & C^* \\ B^* & D^* \end{bmatrix}
 \begin{bmatrix} x_* \\ u_* \end{bmatrix} \right\rangle.
 $$
 We thus conclude that
 $$
 \left\langle \begin{bmatrix} H^\half & 0 \\ 0 & I \end{bmatrix}
 \begin{bmatrix} x \\ u \end{bmatrix}, Q^* \begin{bmatrix} H^{-\half} x_* \\ u_* \end{bmatrix}
 \right\rangle  = \left\langle \begin{bmatrix} x \\ u \end{bmatrix},  \begin{bmatrix} A^* & C^* \\ B^* & D^* \end{bmatrix}
 \begin{bmatrix} x_* \\ u_* \end{bmatrix} \right\rangle
 $$
 for all $\begin{bmatrix} x \\ u \end{bmatrix}$ in $\cD \left( \begin{bmatrix} H^\half & 0 \\ 0 & I \end{bmatrix} \right)$.
 Hence
 \begin{equation}   \label{Q*1}
 Q^* \begin{bmatrix} H^{-\half} x_* \\ u_* \end{bmatrix} \in \cD\left( \begin{bmatrix} H^\half & 0 \\ 0 & I \end{bmatrix}^*\right)
= \cD\left( \begin{bmatrix} H^\half & 0 \\ 0 & I \end{bmatrix}\right)
\end{equation}
and
\begin{equation}  \label{Q*2}
  \begin{bmatrix} H^\half & 0 \\ 0 & I \end{bmatrix} Q^* \begin{bmatrix} H^{-\half} x_* \\ u_* \end{bmatrix} =
  \begin{bmatrix} A^* & C^* \\ B^* & D^* \end{bmatrix} \begin{bmatrix} x_* \\ u_* \end{bmatrix}
\end{equation}
where $x_* \in \cD(H^{-\half})$ and $u_* \in \cY$ are arbitrary.  From the formula \eqref{Q*2} we see that
$A^* \colon \cD(H^{-\half}) \to \im H^\half = \cD(H^{-\half})$ and that $C^* \colon \cY \to \im H^\half = \cD(H^{-\half})$, i.e.,
condition \eqref{KYP1b'*} holds with $H_* = H^{-1}$.  Let us now rewrite equation \eqref{Q*2} in the form
$$
Q^* \begin{bmatrix} H^{-\half} x_* \\ u_* \end{bmatrix} = \begin{bmatrix} H^{-\half} & 0 \\ 0 & I \end{bmatrix} \begin{bmatrix}
A^* & C^* \\ B^* & D^* \end{bmatrix} \begin{bmatrix} x_* \\ y_* \end{bmatrix}.
$$
Using that $Q^*$ is a contraction operator now gives us the spatial KYP-inequality \eqref{KYP1b*} with $H_* = H^{-1}$.  This completes
the proof of Proposition \ref{P:KYPduality}.
 \end{proof}

We next pursue the dual versions of the results of Section \ref{S:ASRS} concerning the {\em available storage} and
{\em required supply} as well as the $\ell^2$-regularized required supply.

First of all let us note that the Laurent operator $\frakL_{F_{\Si^*}}$ of $F_{\Si^*}$, i.e., the inverse $Z$-transform version of the multiplication operator $M_{F_{\Sigma^*}}  = (M_{F_\Sigma})^*$, is just the adjoint of the Laurent operator $\fL_{F_\Sigma}$ given by \eqref{Laurent0}.
We can rewrite $\frakL_{F_{\Si^*}}$ in the convenient block form
\begin{equation}  \label{Laurent*}
\fL_{F_{\Sigma^*}} = \left[ \begin{array}{c|c}
    \fT_{F_{\Sigma^*}} & \fH_{F_{\Sigma^*}}  \\
    \hline
    0 & \widetilde \fT_{F_{\Sigma^*}} \end{array} \right]  =
   \left[  \begin{array}{c|c} (\widetilde \fT_{F_\Sigma})^* & ( \fH_{F_\Sigma})^* \\
    \hline
 0 & ( \fT_{F_\Sigma})^* \end{array} \right]
   \end{equation}
where the Toeplitz operators associated with the adjoint system $\Sigma^*$ are given by
\begin{align*}
& \fT_{F_{\Sigma^*}}   = (\fL_{F_\Sigma})^*|_{\ell^2_\cY({\mathbb Z}_-)}  = (\widetilde \fT_{F_\Sigma})^*, \\
& \widetilde \fT_{F_{\Sigma^*}} = P_{\ell^2_\cU({\mathbb Z}_+)} ( \fL_{F_\Sigma})^*|_{\ell^2_\cY({\mathbb Z}_+)}
= ( \fT_{F_\Sigma})^*
\end{align*}
and where the Hankel operator for the adjoint system (already introduced as the inverse $Z$-transform version
of the frequency-domain Hankel operator ${\mathbb H}_{F_{\Sigma^*}}$ given by \eqref{Hankel-identification}) has the explicit
representation in terms of the Laurent operator $\fL_{F_{\Sigma^*}}  = (\fL_{F_\Sigma})^*$:
$$
  \fH_{F_{\Sigma^*}} = P_{\ell^2_\cU({\mathbb Z}_-)} ( \fL_{F_\Sigma})^*|_{\ell^2_\cU({\mathbb Z}_+)}.
$$

Let $\bcU_*$ be the space of all functions $n \mapsto \bu_*(n)$ from the integers ${\mathbb Z}$ into the input space $\cY$
for the adjoint system $\Sigma^*$.
We define the available storage for the adjoint system $S_{*a}$ by
\begin{equation}  \label{Sa2*}
S_{*a}(x_0) = \sup_{\bu \in \bcU_*, n_{-1} < 0 } \sum_{n=n_{-1}}^{n=-1} (\|\by_*(n) \|^2 - \| \bu_*(n) \|^2 )
\end{equation}
where the supremum is taken over all adjoint-system trajectories
$$
(\bu_*(n), \bx_*(n), \by_*(n) )_{n \le -1}
$$
(specified by the adjoint-system equations \eqref{dtsystem*} running in backwards time) with final condition
$\bx_*(-1) = x_0$.  Similarly, the dual required supply $S_{*r}$ is given by
\begin{equation}   \label{Sr2*}
S_{*r} (x_0) = \inf_{ \bu \in \bcU, \, n_1 \ge 0} \sum_{n=0}^{n_1} ( \| \bu_*(n) \|^2 - \| \by_*(n) \|^2)
\end{equation}
where the infimum is taken over system trajectories $(\bu_*(n), \bx_*(n), \by_*(n) )_{n \le n_1}$
subject to the boundary conditions $\bx_*(n_1) = 0$ and  $\bx(-1) = x_0$.  Then one applies the analysis
behind the proof of Proposition \ref{P:SaSr} to the backward-time system $\Sigma^*$ in place of the forward-time
system $\Sigma$ to see that $S_{*a}$ and $S_{*r}$ are both storage functions for $\Sigma^*$ and furthermore
$S_{*a}(x_0) \le S_*(x_0) \le S_{*r}(x_0)$, $x_0\in\cX$, for any other $\Sigma^*$-storage function $S_*$.  We shall however be primarily
interested in the $\ell^2$-regularized dual required supply $\uS_{*r}$, rather than in $S_{*r}$, defined by
\begin{equation}  \label{uSr2*}
\uS_{*r}(x_0) = \inf_{\bu \in \cD(\bW_{*c}) \colon \bW_{*c} \bu = x_0}
\sum_{n=0}^\infty \left(\| \bu_*(n) \|^2 - \| \by_*(n) \|^2 \right).
\end{equation}
Furthermore,  by working out the backward-time analogues of the analysis in Section \ref{S:ASRS}, one can see
that $\uS_{*r}$ is also a storage function for $\Sigma^*$, and that the definitions of $S_{*a}$ and $S_{*r}$ can be
reformulated in a more convenient operator-theoretic form:
\begin{align}
& S_{*a}(x_0)     =  \sup_{\bu_* \in \ell^2_\cY({\mathbb Z}_-)} \| \bW_{*o} x_0 + \fT_{F_{\Sigma^*}}
 \bu_*\|^2_{\ell^2_\cU({\mathbb Z}_-)}
- \| \bu_* \|^2_{\ell^2_\cY({\mathbb Z}_-)} \notag \\
&  =
\sup_{\bu_* \in \ell^2_\cY({\mathbb Z}_-)} \| \bW_{c}^* x_0 +  \widetilde \fT_{F_{\Sigma}}^*
 \bu_*\|^2_{\ell^2_\cU({\mathbb Z}_-)}
- \| \bu_* \|^2_{\ell^2_\cY({\mathbb Z}_-)} \text{ for } x_0 \in \cD(\bW_c^*)
\label{SaOpform*}
\end{align}
with $S_{*a}(x_0) = + \infty$ if $x_0 \notin \cD(\bW_c^*)$,
while
\begin{align}
& \uS_{*r}(x_0) =  \inf_{\bu_* \in \cD(\bW_{*c}) \colon \bW_{*c} \bu_* = x_0} \| \bu_* \|^2_{\ell^2_\cY({\mathbb Z}_+)} -
\| \widetilde \fT_{F_{\Sigma^*}} \bu_* \|^2_{\ell^2_\cU({\mathbb Z}_+)}  \notag \\
& = \inf_{\bu_* \in \cD(\bW_{o}^*) \colon \bW_{o}^* \bu_* = x_0} \| \bu_* \|^2_{\ell^2_\cY({\mathbb Z}_+)} -
\|  \fT_{F_{\Sigma}}^* \bu_* \|^2_{\ell^2_\cU({\mathbb Z}_+)}  \notag \\
& =  \inf_{\bu_* \in \cD(\bW_{o}^*), \, \bW_{o}^* \bu_* = x_0} \| D_{ \fT_{F_\Sigma}^*} \bu_* \|^2 .
\label{uSrOpForm*}
\end{align}

By notational adjustments to the arguments in the proof of Theorem \ref{T:Sar}, we arrive at the following
formulas for $S_{*a}$ and $\uS_{*r}$ on $\im \bW_o^*$.

\begin{theorem}  \label{T:Sar*}  Let the operators $\overline{X}_a$, $\overline{X}_r$ be as in Lemma \ref{L:fact}
and define operators $H_a$ and $H_r$ as in \eqref{Ha-def} and \eqref{Hr-def}.  Then the dual available storage
$S_{*a}$ and the dual $\ell^2$-regularized required supply are given  {\rm(}on a suitably restricted domain{\rm)} by
\begin{align}
&  S_{*a}(x_0) = \| \overline{X}_r x_0 \|^2 = \| H_r^{-\half} x_0 \|^2 \text{ for } x_0 \in \im \bW_o^*,  \label{form1*} \\
& \uS_{*r}(x_0) = \| | \overline{X}_a|^{-1} x_0 \|^2 = \| H_a^{- \half} x_0 \|^2 \text{ for } x_0 \in \im \bW_o^*.
\label{form2*}
\end{align}
\end{theorem}

Let us associate extended-real-valued functions $S_{H_a}$, $S_{H_r}$, $S_{H_r^{-1}}$, $S_{H_a^{-1}}$ with the positive-definite
operators $H_a$, $H_r$, $H_r^{-1}$, $H_a^{-1}$ as in \eqref{QuadStorage1}.  Theorems \ref{T:Sar} and \ref{T:Sar*} give us
the close relationship between these functions and the functions $S_a$, $\uS_r$ (storage functions for $\Sigma$) and
$S_{*a}$, $\uS_{*r}$ (storage functions for $\Sigma^*$), namely:
\begin{align}
& S_a(x) = S_{H_a}(x), \quad \uS_{r}(x) = S_{H_r}(x) \text{ for } x \in \im \bW_c,  \notag \\
& S_{*a}(x) = S_{H_r^{-1}}(x), \quad \uS_{*r}(x) = S_{H_a^{-1}}(x) \text{ for } x \in \im \bW_o^*.
\label{storage-quadratic}
\end{align}
In general we do not assert that equality holds in any of the four equalities in \eqref{storage-quadratic} for all $x \in \cX$.
Nevertheless it is the case that $S_{H_a}$ and $S_{H_r}$ are storage functions for $\Sigma$ and $S_{H_r^{-1}}$ and
$S_{H_a^{-1}}$ are storage functions for $\Sigma^*$, as we now explain.

\begin{proposition}   \label{P:QuadStorageFuncs}  Let $H_a$, $H_r$, $H_r^{-1}$, $H_a^{-1}$ be the positive-definite
operators as in Theorems \ref{T:Sar} and \ref{T:Sar*}.  Then the following hold:
\begin{enumerate}
\item[(1)] $S_{H_a}$ and $S_{H_r}$ are nondegenerate storage functions for $\Sigma$, or equivalently, $H_a$ and $H_r$ are
positive-definite solutions of the generalized KYP-inequality \eqref{KYP1b'}--\eqref{KYP1b} for $\Sigma$.

\item[(2)] $S_{H_r^{-1}}$ and $S_{H_a^{-1}}$ are storage functions for $\Sigma^*$, or equivalently, $H_r^{-1}$
and $H_a^{-1}$ are positive-definite solutions of the generalized KYP-inequality \eqref{KYP1b'*}--\eqref{KYP1b*} for $\Sigma^*$.
\end{enumerate}
\end{proposition}

\begin{proof}
The fact that $S_H$ is a nondegenerate storage function for $\Sigma$ (respectively $\Sigma^*$) if and only if
$H$ is a positive-definite solution of the generalized KYP-inequality for $\Sigma$ (respectively $\Sigma^*$) is a consequence of
Proposition \ref{P:QuadStorage} and its dual Proposition \ref{P:QuadStorage*}.  We shall use these formulations interchangeably.

We know that $S_{H_a}(x) = S_a(x)$ for $x \in \im \bW_c$.  Furthermore as a consequence of \eqref{intertwine1} with
$\widetilde u = 0$ and of
\eqref{Wc-fin} with $n_{-1} = -1$, we see that $\im \bW_c$ is invariant under $A$ and  contains $\im B$.
Thus condition \eqref{KYP1b'} holds
with $\im \bW_c$ in place of $\cD(H^\half)$.  The facts that $S_{H_a}$ agrees with $S_a$ on $\im \bW_c$ and that
$S_a$ is a storage function for $\Sigma$ implies that the inequality \eqref{KYP1b} holds for $x \in \im \bW_c$ and $u \in \cU$:
\begin{equation}\label{KYP1b-Wc}
\left\| \begin{bmatrix} H_a^{\half} \! & \! 0 \\ 0 \! & \! I_{\cU}
\end{bmatrix} \begin{bmatrix} x \\ u \end{bmatrix} \right\|^{2}
- \left\| \begin{bmatrix} H_a^{\half} \! & \! 0 \\ 0 \! & \! I_{\cY}
\end{bmatrix} \begin{bmatrix} A\! & \! B \\ C \! & \! D \end{bmatrix}
\begin{bmatrix} x \\ u \end{bmatrix} \right\|^{2} \ge 0.
\end{equation}
As noted at the end of Theorem \ref{T:Sar}, $\im \bW_c$ is a core for $H_a^\half$; hence, given $x \in \cD(H_a)$, there is a sequence
of points $\{ x_n\}_{n \ge 1}$ contained in $\im \bW_c$ such that $\lim_{n \to \infty} x_n = x$ and $\lim_{n \to \infty} H^\half x_n =
H^\half x$.  As each $x_n \in \im \bW_c$, we know that the inequality \eqref{KYP1b-Wc} holds with $x_n$ in place of $x$ for all
$n = 1,2,\dots$.  We may now take limits in this inequality to see that the inequality continues to hold with $x = \lim_{n \to \infty}
x_n \in \cD(H_a^\half)$, i.e., condition \eqref{KYP1b} holds with $H_a$ in place of $H$.  Thus $H$ is a solution of the
generalized KYP-inequality for $\Sigma$.  That $H_r^{-1}$ is a solution of the generalized KYP-inequality for $\Sigma^*$
now follows by applying the same analysis to $\Sigma^*$ rather than to $\Sigma$.
Finally, the fact that $H_a$ (respectively, $H_r^{-1}$) is a positive-definite solution of the generalized KYP-inequality for
$\Sigma$ (respectively for $\Sigma^*$) implies that $H_a^{-1}$ (respectively, $H_r$) is a positive-definite solution of the
generalized KYP-inequality for $\Sigma^*$ (respectively, $\Sigma$)
as a consequence of Proposition \ref{P:KYPduality}.
\end{proof}

\section{Order properties of solutions of the generalized KYP-inequality and finer results for special cases}  \label{S:order}

We have implicitly been using an order relation on storage functions, namely:  we say that $S_1 \le S_2$ if $S_1(x_0)
\le S_2(x_0)$ for all $x_0 \in \cX$.  For the case of quadratic storage functions $S_{H_1}$ and $S_{H_2}$ where $H_1$ and $H_2$ are two positive-semidefinite solutions of the generalized KYP-inequality \eqref{KYP1b'}--\eqref{KYP1b}, the induced ordering $\le$ on
positive-semidefinite (possibly unbounded) operators can be
 defined as follows:  given two positive-semidefinite operators $H_1$ with dense domain $\cD(H_1)$ and $H_2$ with dense
 domain $\cD(H_2)$ in $\cX$,  we say that {\em   $H_1 \le H_2$ if $\cD(H_2^\half) \subset \cD(H_1^\half)$ and }
 \begin{equation}  \label{H1leH2}
   \| H^\half_1 x \|^2 \le \| H^\half_2 x \|^2 \text{ for all } x \in \cD(H_2^\half).
 \end{equation}
In case $H_1$ and $H_2$ are bounded positive-semidefinite operators, one can see that $H_1 \le H_2$ is equivalent to
$H_1 \preceq H_2$ in the sense of the inequality between quadratic forms:  $\langle H_1 x, x \rangle \le \langle H_2 x, x \rangle$,
i.e., in the Loewner partial order:  $H_2 - H_1 \succeq 0$.  This ordering $\le$ on (possibly unbounded) positive-semidefinite operators has appeared in the more general context of closed quadratic forms $S_H$ (not necessarily storage functions for some dissipative
system $\Sigma$) and associated semibounded selfadjoint operators $H$ (not necessarily solving some generalized KYP-inequality);
see formula (2.17) and the subsequent remark in the book of Kato \cite{Kato}.  This order has been studied in the setting of
solutions of a generalized KYP-inequality in the paper of Arov-Kaashoek-Pik \cite{AKP06}.  Here we
offer a few additional such  order properties which follow from the results developed here.
Recall that the notion of a {\em core} of a closed, densely defined linear operator was introduced in the paragraph preceding
Theorem \ref{T:Sar}.

\begin{theorem}  \label{T:order}
 Assume that the system $\Sigma$ in \eqref{dtsystem} satisfies the standing
assumption \eqref{A} and $H_a$ and $H_r$ are defined by \eqref{Ha-def} and \eqref{Hr-def}.
Let $H$ be any positive-definite solution of the generalized KYP-inequality
\eqref{KYP1b'}--\eqref{KYP1b}.
\begin{enumerate}
\item[(1)]  Assume that $\im \bW_c$ is a core for $H^\half$.  Then  we have the operator inequality
\begin{equation}  \label{HaH-ineq}
  H_a \le H
\end{equation}
and furthermore $\im \bW_o^* \subset \cD(H^{-\half})$.

\item[(2)]  Assume that $\im \bW_o^*$ is a core for $H^{-\half}$.  Then we have the operator inequality
\begin{equation}   \label{HHr-ineq}
  H \le H_r
\end{equation}
and furthermore $\im \bW_c \subset \cD(H^\half)$.
\end{enumerate}
\end{theorem}

\begin{proof}
We deal with (1) and (2) in turn.

{(1)}
Suppose  that $H$ is a positive-definite solution of the generalized KYP-inequality such that
$\im \bW_c$ is a core for $H^\half$.
 From Theorem \ref{T:Sar}, we know that $S_a(x) = \| H_a^\half x\|$ for $x \in \im \bW_c$.  Since $S_a$ is the
smallest storage function
(see Proposition \ref{P:SaSr}) and $S_H$ is a storage function, it follows that
\begin{equation}   \label{Ha-Hineq}
\| H_a^\half x \|^2   = S_a(x) \le S_H(x) = \| H^\half x \|^2 \text{ for } x \in \im \bW_c.
\end{equation}
Let now $x$ be an arbitrary point of $\cD(H^\half)$.
Since $\im \bW_c$ is a core for $H^\half$, we can find a sequence $\{x_n \}_{n \ge 1}$ of points
in $\im \bW_c$ such that $x_n \to x$ and $H^\half x_n \to H^\half x$.  In particular $H^\half x_n$ is a Cauchy sequence
and the inequality
\begin{align*}
 & \| H_a^\half x_n - H_a^\half x_m \|^2 = \| H_a ^\half(x_n - x_m ) \|^2   \\
 & \quad \quad \le  \| H^\half(x_n - x_m ) \|^2  =
  \| H^\half x_n -  H^\half x_m \|^2
\end{align*}
implies that $\{ H_a^\half x_n \}_{n \ge 1}$ is Cauchy as well, so converges to some $y \in \cX$. As $H_a$ is closed,
we get that $x \in \cD(H_a)$ and $y = H_a^\half x$.  We may then take limits in the inequality $\|H_a^\half x_n \|^2 \le \| H^\half x_n \|^2$ holding for all
$n$ (a consequence of \eqref{Ha-Hineq}) to conclude that $\| H_a^\half x \|^2 \le \| H^\half x \|^2$, i.e., $H_a \le H$, i.e.,
 \eqref{HaH-ineq} holds.

Recall next from Corollary \ref{C:boundedSauSr} that
$\| \bW_o x_0 \|^2 \le S_a(x_0)$, where we now also know from Theorem \ref{T:Sar} that
$S_a(x_0) = \| H^\half x_0 \|^2$ for $x_0 \in \im \bW_c$.  We thus have the chain of operator inequalities
$$
  \bW_o^* \bW_o \le H_a \le H.
$$
By Proposition 3.4 in \cite{AKP05}, we may equivalently write
$$
  H^{-1} \le H_a^{-1} \le (\bW_o^* \bW_o)^{-1}.
$$
In particular $\cD( | \bW_o |^{-1}) \subset \cD(H^{- \half})$.  If we introduce the polar decomposition
$\bW_o = U_o | \bW_o |$ for $\bW_o$, we see  that $\bW_o^* = | \bW_o | U_o^*$ and hence
$\im \bW_o^* = \im | \bW_o|$.  Thus
$$
 \cD( | \bW_o |^{-1}) = \im | \bW_o| = \im \bW_o^*
 $$
 and it follows that $\im \bW_o^* \subset \cD(H^{-\half})$ and the verification of (1) is complete.

{(2)}  We now suppose that $H$ is a positive-definite solution of the generalized KYP-inequality such that
$\im \bW_o^*$ is a core for $H^{-\half}$.  By the applying the result of part (1) to the adjoint system $\Sigma^*$, we see
that $H_r^{-1} \le H^{-1}$ and that $\im \bW_c \subset H^\half$.  If we apply the result of Proposition 3.4 in \cite{AKP05},
we see that $H_r^{-1} \le H^{-1}$ implies that (is actually equivalent to) $H \le H_r$, completing the verification of (2).
\end{proof}

\begin{remark}  \label{R:ineq-chain} By the last assertion in Theorem \ref{T:Sar}, we know that  $\im \bW_c$ is a core for
$H_a^\half$ and that $\im \bW_o^*$ is a core for $H_r^{-\half}$.  Also by Proposition \ref{P:QuadStorageFuncs} we know
that $H_a$ and $H_r$ are positive-definite solutions of the generalized KYP-inequality for $\Sigma$.
Thus item (1) in Theorem \ref{T:order} may be rephrased as follows:

\begin{itemize}
\item
{\sl The set ${\mathcal GS}_c$ consisting of all positive-definite solutions $H$ of the generalized KYP-inequality \eqref{KYP1b'}--\eqref{KYP1b}  for $\Sigma$
such that $\im \bW_c$ is a core  for $H^\half$ has the solution $H_a$ as a minimal element
with respect to the ordering $\le$.}
\end{itemize}

\noindent
 Similarly item (2) in Theorem \ref{T:order} may be rephrased as:

 \begin{itemize}
 \item {\em The set ${\mathcal GS}_o$ consisting of all positive-definite solutions $H$ of the generalized KYP-inequality \eqref{KYP1b'}--\eqref{KYP1b} such that $\im \bW_o^*$ is a core for  $H^{-\half}$  has the solution $H_r$ as a maximal element with respect to the ordering $\le$.}
 \end{itemize}

\noindent
 It would be tempting to say:
  \begin{itemize}
 \item {\em The set ${\mathcal GS}$ consisting of all positive-definite solutions $H$ of the generalized KYP-inequality \eqref{KYP1b'}--\eqref{KYP1b}
 such that $\im \bW_c$  a core for $H^\half$ and $\im \bW_o^*$ is a core for $H^{-\half}$ has $H_a$ as a minimal element and $H_r$
 as a maximal element with respect to the ordering $\le$.}
 \end{itemize}
 However, while the above results imply that $\im \bW_c \subset \cD(H_r^{\half})$ and that
 $\im \bW_o^* \subset \cD(H_a^{-\half})$,
   we have not been able to show in general that $\im \bW_c$ is a core for $H_r^\half$ or that $\im \bW_o^*$ is a core for
 $H_a^{-\half}$.  Such a more satisfying symmetric statement does hold in the pseudo-similarity framework for the analysis of
 solutions of generalized KYP-inequalities (see Proposition 5.8 in \cite{AKP06}).
\end{remark}

We now consider the case that $\Si$ is not only controllable and/or observable, but has the stronger $\ell^2$-exact controllability or
$\ell^2$-exact observability condition, or both, i.e., $\ell^2$-exact minimality. We first consider the implications on $H_a$ and $H_r$.

\begin{proposition}\label{P:ell2minImplsHaHr}
Let $\Sigma$ be a system as in \eqref{dtsystem} such that
assumption \eqref{A} holds.
\begin{itemize}
\item[(1)] If $\Sigma$ is $\ell^2$-exactly controllable, then $H_a$ and $H_r$ are bounded.

\item[(2)] If $\Sigma$ is $\ell^2$-exactly observable, then $H_a$ and $H_r$ are boundedly invertible.

\item[(3)] $\Sigma$ is $\ell^2$-exactly minimal, i.e., both
$\ell^2$-exactly controllable and $\ell^2$-exactly observable,
then $H_a$ and $H_r$ are both bounded and boundedly invertible.
\end{itemize}
\end{proposition}

\begin{proof}
We discuss each of (1), (2), (3) in turn.

{(1)} Item (1) follows directly from the fact that $\im \bW_c$ is contained in both $\cD(H_a)$ and $\cD(H_r)$ together with the
Closed Graph Theorem.

{(2)}  From the last assertion in Theorem \ref{T:Sar}, we know that $\im \bW_c$ is a core for $H_a^\half$.  Then
item (1) in Theorem \ref{T:order} implies that $\im \bW_o^* \subset \cD(H_a^{-\half})$.  If $\im
\bW_o^* = \cX$, the Closed Graph Theorem then gives us that $H_a^{-\half}$ is bounded.

Also part of the last assertion of Theorem \ref{T:Sar} is the statement that $\bW_o^*$ is a core for $H_r^{-\half}$, so
in particular $\im \bW_o^* \subset \cD(H_r^{-\half})$.  Then again the Closed Graph Theorem implies that $H_r^{-\half}$
is bounded.

{(3)}.  Simply combine the results of items (1) and (2).
\end{proof}

Next we consider general positive-definite solutions to the generalized KYP-inequality.

\begin{proposition}\label{P:ell2minImplsH}
Suppose that $\Sigma$ is a system as in \eqref{dtsystem} such that assumption \eqref{A} holds and that $H$ is any
positive-definite solution of the generalized KYP-inequality.
\begin{itemize}
\item[(1)]   Suppose that $\Sigma$ is $\ell^2$-exactly controllable and that $\im \bW_c \subset \cD(H^\half)$
{\rm(}as is the case  e.g.\ if $\im \bW_o^*$ is a core for $H^{-\half}${\rm)}.
Then $H$ is bounded and furthermore
$$
        H_a \le H.
$$

\item[(2)] Suppose that $\Sigma$ is $\ell^2$-exactly observable and that $\im \bW_o^*  \subset \cD(H^{-\half})$
{\rm(}as is the case e.g.\ if $\im \bW_c$ is a core for $H^\half${\rm)}.
Then $H^{-1}$ is bounded and furthermore
$$
   H \le H_r.
$$

\item[(3)] Suppose that $\Sigma$ is both $\ell^2$-exactly controllable and $\ell^2$-exactly observable and
that either {\rm(a)} $\im \bW_c \subset \cD(H^\half)$ or {\rm(b)} $\im \bW_o^* \subset \cD(H^{-\half})$.  Then $H$
is bounded and boundedly invertible and we have the inequality chain
\begin{equation}  \label{HaHrineq}
   H_a \le H \le H_r.
\end{equation}
\end{itemize}
\end{proposition}

\begin{proof}
First note that the fact that the parenthetical hypotheses in items (1) and (2) are stronger than the given
hypotheses is a consequence of the final assertions in parts (1) and (2) of Theorem \ref{T:order}.  We now deal with the rest of (1), (2), (3).

{(1)} If we assume that $\cX = \im \bW_c \subset \cD(H^\half)$, then $H^\half$ (and hence also $H$) is bounded by
the Closed Graph Theorem.  Moreover,  as $\im \bW_c = \cD(H^\half)$, in particular $\im \bW_c$ is a core for $H^\half$
and the inequality $H_a \le H$ follows from Theorem \ref{T:order} (1).

{(2)} Similarly, if we assume $\cX = \im \bW_o^* \subset \cD(H^{-\half})$, then $H^{-\half}$ is bounded by the Closed Graph Theorem.  As $\im \bW_o^* = \cD(H^{-\half})$, in particular $\im \bW_o^*$ is a core for $H^{-\half}$ and $H \le H_r$ follows as a consequence of Theorem \ref{T:order} (2).

{(3)}  If $\cX = \im \bW_c \subset \cD(H^\half)$, then in fact $\im \bW_c = \cD(H^\half)$ so $\im \bW_c$ is a core for
$H^\half$.  By Theorem \ref{T:order}, it follows that $\im \bW_o^* \subset \cD(H^{-\half})$ and hence hypothesis (b) is a consequence
of hypothesis (a) when combined with all the other hypotheses in (3).  Similarly hypothesis (a) is a consequence of
hypothesis (b).  Hence there is no loss of generality in assuming that both (a) and (b) hold.  Then the verification of (3)
is completed by simply combining the results of (1) and (2).
\end{proof}

\section{Proofs of Bounded Real Lemmas}  \label{S:BRLproof}

We now put all the pieces together to give a storage-function proof of Theorem~\ref{T:BRLinfstan}.

\begin{proof}[Proof of Theorem \ref{T:BRLinfstan}]
We are given a minimal system $\Sigma$ as in \eqref{dtsystem} with transfer function $F_\Sigma$ in the Schur class
$\cS(\cU, \cY)$.

\smallskip

\noindent
{\em Proof of sufficiency.} For the sufficiency direction, we assume either that there exists a positive-definite solution $H$
of the generalized KYP-inequality \eqref{KYP1b'}--\eqref{KYP1b} (statement (1)) or a bounded and boundedly invertible solution
$H$ of the KYP-inequality \eqref{KYP1} (statements (2) and (3)).   As the latter case is a particular version of the former case,
it suffices to assume that we have a positive-definite solution of the generalized KYP-inequality \eqref{KYP1b'}--\eqref{KYP1b}.
We are to show that then $F_\Sigma$ is in the Schur class $\cS(\cU, \cY)$.

Given such a generalized solution of the KYP-inequality,  Proposition \ref{P:QuadStorage} guarantees us
that $S_H$ is an (even quadratic) storage function for $\Sigma$.  Then $F_\Sigma$ has analytic continuation to a
Schur class function by Proposition \ref{P:storage-Schur}.

\smallskip

\noindent
{\em Proof of necessity in statement (1):}  We assume that $\Sigma$ is minimal and that $F_\Sigma$ has
analytic continuation to a Schur-class function, i.e., assumption \eqref{A} holds.  Then Proposition \ref{P:QuadStorageFuncs}
gives us two choices $H_a$ and $H_r$ of positive-definite solutions of the generalized KYP-inequality
\eqref{KYP1b'}--\eqref{KYP1b}.

\smallskip

\noindent
{\em Proof of necessity in statement (2):}  We assume that $\Sigma$ is exactly controllable and exactly observable
with transfer function $F_\Sigma$ having analytic continuation to the Schur class.  From Proposition
\ref{P:HankelDecs} (1)  we see that $\im \bW_c \supset \Rea(A|B) = \cX$ and that $\cD(\bW_o) \supset \Rea (A|B) = \cX$
while from item (2) in the same proposition we see that
$\im \bW_o^* \supset \Obs(C|A) = \cX$ and that  $\cD(\bW_c^*) \supset \Obs(C|A) = \cX$.  Hence by the Closed Graph
Theorem, in fact $\bW_c$ and $\bW_o^*$ are bounded in addition to being surjective.  In particular
$\Sigma$ is $\ell^2$-exactly controllable and $\ell^2$-exactly observable, so this case actually falls under
item (3) of  Theorem \ref{T:BRLinfstan}, which we will prove next.

\smallskip

\noindent
{\em Proof of necessity in statement (3):}  We now assume that $\Sigma$ is $\ell^2$-exactly controllable and $\ell^2$-exactly
observable with $F_\Sigma$ having analytic continuation to a function in the Schur class $\cS(\cU, \cY)$
and we want to produce a bounded and boundedly invertible solution $H$ of the KYP-inequality
\eqref{KYP1}.  In particular, $\Sigma$ is minimal (controllable and observable),  so Proposition \ref{P:QuadStorageFuncs}
gives us two solutions $H_a$ and $H_r$ of the generalized KYP-inequality.
But any  solution $H$ of the generalized KYP-inequality becomes a solution of the standard KYP-inequality
\eqref{KYP1} if it happens to be the case that $H$ is bounded.  By the result of item (3) in
Proposition \ref{P:ell2minImplsHaHr}, both $H_a$ and $H_r$ are bounded and boundedly invertible under our
$\ell^2$-minimality assumptions.  Thus in this case  $H_a$ and $H_r$ serve as two choices for bounded, strictly positive-definite
solutions of the KYP-inequality, as needed.
\end{proof}

We are now ready also for a storage-function proof of Theorem \ref{T:BRLinfstrict}.

\begin{proof}[Proof of Theorem \ref{T:BRLinfstrict}] The standing assumption for both directions is that $\Sigma$ is a
linear system as in \eqref{dtsystem} with exponentially stable state operator $A$.

\smallskip

\noindent
{\em Proof of necessity:}  Assume that there exists a bounded strictly positive-definite solution $H$ of the
strict KYP-inequality.  By Proposition \ref{P:strictQuadStorage}, $S_H$ is a strict
storage function for $\Sigma$.  Then by Proposition \ref{P:strictstorage-Schur}, $F_\Sigma$ has analytic continuation to an
$\cL(\cU, \cY)$-valued $H^\infty$-function with $H^\infty$-norm strictly less than 1 as wanted.  The fact that
$A$ is exponentially stable implies that $F_\Sigma$ has analytic continuation to a slightly larger disk
beyond ${\mathbb D}$, and the fact that $H$ is strictly positive-definite implies that $S_H$ has the additional
coercivity property $S_H(x) \ge \epsilon_0 \| x \|^2$ for some $\epsilon_0 > 0$.

\smallskip

\noindent
{\em Proof of sufficiency:}
We are assuming that  $\Sigma$ has state operator $A$
exponentially stable and with transfer function $F_\Sigma$ in the strict Schur class.
The exponential stability of $A$ (i.e.  $A$ has spectral radius $r_{\rm spec}(A) < 1$) means that
the series
$$
    \bW_{o}^{*}\by = \sum_{k=0}^{\infty} A^{*k} C^{*}\by(k)\ \
(\by\in\ell^2_\cY(\BZ_+)), \quad
    \bW_{c}\bu = \sum_{k=0}^{\infty} A^{k} B\bu(k)\ \
(\bu\in\ell^2_\cU(\BZ_-))
$$
are norm-convergent (not just in the weak sense as in Proposition \ref{P:WcWo'}), and hence $\bW_c$ and $\bW_o$ are bounded.
However it need not be the case that $\bW_c$ or $\bW_o^*$ be surjective, so we are not in a position to
apply part (3) of Theorem \ref{T:BRLinfstan} to the system $\Sigma$.
The adjustment for handling this difficulty which also ultimately produces bounded and boundedly invertible solutions
of the strict KYP-inequality \eqref{KYP2} is what we shall call {\em $\epsilon$-regularization reduction}. It goes back at least to Petersen-Anderson-Jonkheere \cite{PAJ} for the finite dimensional case, and was extended to the infinite dimensional case in our previous paper \cite{KYP1}. We recall the procedure here for completeness and because we refer to it in a subsequent remark.

Since $\spec (A) < 1$, the resolvent expression $(I - \lambda A)^{-1}$ is uniformly bounded for all $\lambda$ in the unit disk
${\mathbb D}$.  Since we are now assuming that $F_\Sigma$ is
in the strict Schur class, it follows that we can choose $\epsilon >0$ sufficiently small so that the
augmented matrix function
\begin{equation}   \label{Fepsilon}
F_{\epsilon}(\lambda) : = \begin{bmatrix} F(\lambda) & \epsilon \lambda C
(I - \lambda A)^{-1} \\ \epsilon \lambda (I -  \lambda A)^{-1} B & \epsilon^{2} \lambda
(I - \lambda A)^{-1} \\ \epsilon I_{\cU} & 0 \end{bmatrix}
\end{equation}
is in the strict Schur class $\cS^{o}( \cU \oplus \cX, \cY \oplus \cX
\oplus \cU)$. Note that
\[
 F_{\epsilon}(\lambda) = \begin{bmatrix} D & 0 \\ 0 & 0 \\ \epsilon I_{\cU}
 & 0 \end{bmatrix} + \lambda  \begin{bmatrix}  C \\ \epsilon I_{\cX} \\ 0
 \end{bmatrix} (I - \lambda A)^{-1} \begin{bmatrix} B & \epsilon I_{\cX}
 \end{bmatrix}
\]
and hence
\begin{equation}   \label{breal}
 M_{\epsilon} =  \begin{bmatrix}  \bA & \bB \\ \bC & \bD
\end{bmatrix} : =
  \mat{c|cc}{
    A &  B & \epsilon I_{\cX}\\
    \hline C &  D & 0 \\
    \epsilon I_{\cX} & 0 & 0 \\
    0  & \epsilon I_{\cU} & 0}
\end{equation}
is a realization for $ F_{\epsilon}(\lambda)$.   Suppose that we can find a bounded and boundedly invertible
positive-definite operator $H$ satisfying the KYP-inequality \eqref{KYP1} associated with the system
$\Sigma_\epsilon$:
\begin{equation}   \label{KYP-epsilon}
   \begin{bmatrix} \bA^{*} & \bC^{*} \\ \bB^{*} & \bD^{*}
\end{bmatrix}
       \begin{bmatrix} H & 0 \\ 0 & I_{\cY \oplus \cX \oplus \cU}
       \end{bmatrix} \begin{bmatrix} \bA & \bB \\ \bC & \bD
   \end{bmatrix} \preceq \begin{bmatrix} H & 0 \\ 0 & I_{\cU \oplus
   \cX} \end{bmatrix}.
\end{equation}
Spelling this out gives
$$
\begin{bmatrix} A^{*}HA + C^{*}C + \epsilon^{2}I_{\cX} & A^{*}H B +
    C^{*}D & \epsilon A^{*}H \\
 B^{*}HA + D^{*}C & B^{*} H B + D^{*} D + \epsilon^{2}I_{\cU} &
 \epsilon B^{*} H \\ \epsilon HA & \epsilon HB & \epsilon^{2} H
\end{bmatrix} \preceq \begin{bmatrix} H & 0 & 0 \\ 0 & I_{\cU} & 0 \\
0 & 0 & I_{\cX} \end{bmatrix}.
$$
By crossing off the third row and third column, we get the inequality
$$
\begin{bmatrix} A^{*} H A + C^{*} C + \epsilon^{2} I_{\cX} & A^{*}H B
    + C^{*} D \\ B^{*}H A + D^{*} C & B^{*} H B + D^{*} D +
    \epsilon^{2} I_{\cU} \end{bmatrix} \preceq \begin{bmatrix} H & 0
    \\ 0 & I_{\cU} \end{bmatrix}
$$
or
$$
\begin{bmatrix} A^{*} & C^{*} \\ B^{*} & D^{*} \end{bmatrix}
    \begin{bmatrix} H & 0 \\ 0 & I_{\cY} \end{bmatrix}
	\begin{bmatrix} A & B \\ C & D \end{bmatrix} + \epsilon^{2}
	    \begin{bmatrix} I_{\cX} & 0 \\ 0 & I_{\cU} \end{bmatrix}
		\preceq \begin{bmatrix} H & 0 \\ 0 & I_{\cU}
	    \end{bmatrix}
$$
leading us to the strict KYP-inequality \eqref{KYP2} for the original system $\Sigma$ as wanted.

It remains only to see why there is a bounded and boundedly invertible solution $H$ of \eqref{KYP-epsilon}.
It is easily checked that the system $\Sigma_\epsilon$ is exactly controllable and exactly minimal, since $\bB$
and $\bC^*$ are both already surjective; as observed in the proof of necessity in item (2) of
Theorem \ref{T:BRLinfstan},  since $F_{\Sigma_\epsilon}$ is in the Schur class it then follows that $\Sigma_\epsilon$
is $\ell^2$-exactly controllable and $\ell^2$-exactly observable as well.  Hence we can appeal to
either items (2) or (3) of Theorem \ref{T:BRLinfstan} to conclude that indeed the KYP-inequality
\eqref{KYP-epsilon} has a bounded and boundedly invertible positive-definite solution.  This is what is done in
\cite{KYP1}, where the State-Space-Similarity approach is used to prove items (2) and (3) in Theorem
\ref{T:BRLinfstan} rather the storage-function approach as is done here.
 \end{proof}

  \begin{remark} \label{R:HaHr-notbounded}
 Let $\Si$ and $F_\Si$ satisfy the conditions
 of strict Bounded Real Lemma (Theorem \ref{T:BRLinfstrict}). Define
 the $\ep$-augmented system $\Si_\ep$ as in \eqref{breal}. We then obtain bounded, strictly positive-definite
 solutions $H_{a,\ep}$ and $H_{r,\ep}$ of the strict KYP
inequality \eqref{KYP2}, and consequently, by Proposition \ref{P:ell2minImplsH} (3) all bounded or
bounded below solutions $H$ to the generalized KYP inequality \eqref{KYP1b'}--\eqref{KYP1b} for
$\Sigma_{\epsilon}$ satisfy $H_{a,\ep} \le H \le H_{r,\ep}$ and hence are in fact bounded,
strictly positive-definite solutions to the  KYP inequality \eqref{KYP1} for the
original system $\Sigma$.  An application of Theorem \ref{T:order} together with the observation that $\cD$ being
a core for the bounded operator $X$ on $\cX$ is the same as $\cD$ being dense in $\cX$ leads to the conclusion that
the operators $H_a$ and $H_r$ associated with the original system satisfy $H_a \le H_{a, \epsilon}$ and
$H_r^{-1} \le H_{r, \epsilon}^{-1}$ and hence are bounded.
However, this by itself is not enough to conclude that $H_a$ and $H_r^{-1}$ are also bounded below.
\end{remark}

\paragraph
{\bf Acknowledgements}
This work is based on the research supported in part by the National Research Foundation of South Africa.
Any opinion, finding and conclusion or recommendation expressed in this material is that of the authors and the
NRF does not accept any liability in this regard.

It is a pleasure to thank Olof Staffans for enlightening discussion (both verbal and written)
and Mikael Kurula for his useful comments while
visiting the first author in July 2017.

\end{document}